\documentclass{nm}
\renewcommand\thisyear {200x}

\renewcommand\thisvolume{xx}

\usepackage{charter}
\usepackage[charter]{mathdesign}
\usepackage{booktabs}
\usepackage{amssymb}
\usepackage{amscd}
\usepackage{amsmath}
\usepackage{amsthm}
\usepackage{mathrsfs}
\usepackage{subfig}
\usepackage{paralist}
\usepackage{graphicx}
\usepackage{epstopdf}
\DeclareMathAlphabet\gothic{U}{euf}{m}{n}
\DeclareMathAlphabet{\mathcall}{OMS}{cmsy}{m}{n}
\usepackage{url}

\newcommand{\ul}{\textbf}

\newcommand{\R}{\mathbb{R}}

\newcommand{\desda}{\Leftrightarrow}

\newcommand{\bcs}{\begin{cases}}
	\newcommand{\ecs}{\end{cases}}

\newcommand{\bqn}{\begin{equation}\begin{aligned}}
\newcommand{\eqn}{\end{aligned}\end{equation}}

\newcommand{\www}{\mbox{\boldmath$\omega$}}

\graphicspath{{Images//}}

\begin{document}
	
	
	\markboth{Jiong Zhang, \and Remco Duits, \and Gonzalo Sanguinetti, \and Bart M. ter Haar Romeny}{Numerical Approaches for Linear $\mbox{Diffusions on}$ $SE(2)$}
	\title{Numerical Approaches for Linear Left-invariant $\mbox{Diffusions on}$ $SE(2)$, their Comparison to Exact Solutions, and their $\mbox{Applications}$ in Retinal Imaging.}
	
	\author[Jiong Zhang, Remco Duits, Gonzalo Sanguinetti and Bart M. ter Haar Romeny]{Jiong Zhang \affil{1,}\footnote{Corresponding and joint main authors. \textit{Email address:}
			J.Zhang1@tue.nl (J. Zhang), R.Duits@tue.nl (R. Duits) }, Remco Duits\affil{1,}\affil{2,}$^*$, Gonzalo Sanguinetti\affil{2} and Bart M. ter Haar Romeny\affil{3,}\affil{1}}
	
	\address{
		$^1$Department of Biomedical Engineering, Biomedical Image Analysis (BMIA) \\
		$^2$Department of Mathematics and Computer Science, Image Science and Technology (IST/e)\\
		$^{1,2}$Eindhoven University of Technology, P.O. Box 513, 5600 MB, Eindhoven, The Netherlands.\\
		$^3$Northeastern University, Shenyang, China.}
	%
	%
	%
	
	\begin{abstract}
		Left-invariant PDE-evolutions on the roto-translation group $SE(2)$ (and their resolvent equations) have been widely studied in the fields of cortical modeling and image analysis. They include hypo-elliptic diffusion (for contour enhancement) proposed by Citti $\&$ Sarti, and Petitot, and they include the direction process (for contour completion) proposed by Mumford. This paper presents a thorough study and comparison of the many numerical approaches, which, remarkably, are missing in the literature. Existing numerical approaches can be classified into 3 categories: Finite difference methods, Fourier based methods (equivalent to $SE(2)$-Fourier methods), and stochastic methods (Monte Carlo simulations). There are also 3 types of exact solutions to the PDE-evolutions that were derived explicitly (in the spatial Fourier domain) in previous works by Duits and van Almsick in 2005. Here we provide an overview of these 3 types of exact solutions and explain how they relate to each of the 3 numerical approaches. We compute relative errors of all numerical approaches to the exact solutions, and the Fourier based methods show us the best performance with smallest relative errors. We also provide an improvement of Mathematica algorithms for evaluating Mathieu-functions, crucial in implementations of the exact solutions. Furthermore, we include an asymptotical analysis of the singularities within the kernels and we propose a probabilistic extension of underlying stochastic processes that overcomes the singular behavior in the origin of time-integrated kernels. Finally, we show retinal imaging applications of combining left-invariant PDE-evolutions with invertible orientation scores.
	\end{abstract}
	
	
	
	\maketitle
	
	\section{Introduction} \label{section:Introduction}
	Hubel and Wiesel \cite{Hube59a} discovered that certain visual cells in cats' striate cortex have a directional preference.
	It has turned out that there exists an intriguing and extremely precise spatial and directional organization into so-called cortical hypercolumns, see Figure~\ref{fig:VisualCortex}.
	A hypercolumn can be interpreted as a ``visual pixel'', representing the optical world at a single location, neatly decomposed into a complete set of orientations. Moreover, correlated horizontal connections run parallel to the cortical surface and link columns across the spatial visual field with a shared orientation preference, allowing cells to combine visual information from spatially separated receptive fields.
	Synaptic physiological studies of these horizontal pathways in cats' striate cortex show that neurons with aligned receptive field sites excite each other \cite{Bosking}. Apparently, the visual system not only constructs a score of local orientations, but also accounts for context and alignment by excitation and inhibition \emph{a priori}, which can be modeled by left-invariant PDE's and ODE's on $SE(2)$ \cite{Petitot,Citti,DuitsPhDThesis,Duits2007IJCV,Boscain3,August,BenYosef2012a,Chirikjian2,MashtakovNMTMA,Gonzalo,SartiCitteCompiledBook2014,Zweck,DuitsAMS1,DuitsAMS2,DuitsAlmsick2008,FrankenPhDThesis,BarbieriArxiv2013,Mumford}.
	\begin{figure}[!b]
		\centering
		\includegraphics[width= 0.6\hsize]{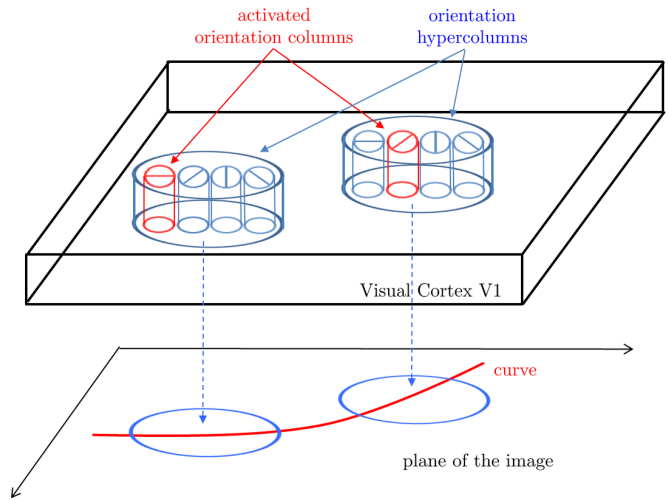}
		\caption{The orientation columns in the primary visual cortex.}
		\label{fig:VisualCortex}
	\end{figure}
	Motivated by the orientation-selective cells, so-called orientation scores are constructed
	by lifting all elongated structures (in 2D images) along an extra orientation dimension \cite{Kalitzin97,DuitsPhDThesis,Duits2007IJCV}. The main advantage of using the orientation score is that we can disentangle the elongated structures involved in a crossing allowing for a crossing preserving flow.
	
	Invertibility of the transform between image and score is of vital importance, to both tracking \cite{BekkersJMIV} and enhancement\cite{Sharma2014,Franken2009IJCV}, as we do not want to tamper data-evidence in our continuous coherent state transform\cite{Alibook,Zweck} before actual processing takes place. This is a key advantage over related state-of-the-art methods\cite{AugustPAMI,MashtakovNMTMA,Zweck,Boscain3,Citti}.
	\begin{figure}[!htbp]
		\centering
		\includegraphics[width=.7\textwidth]{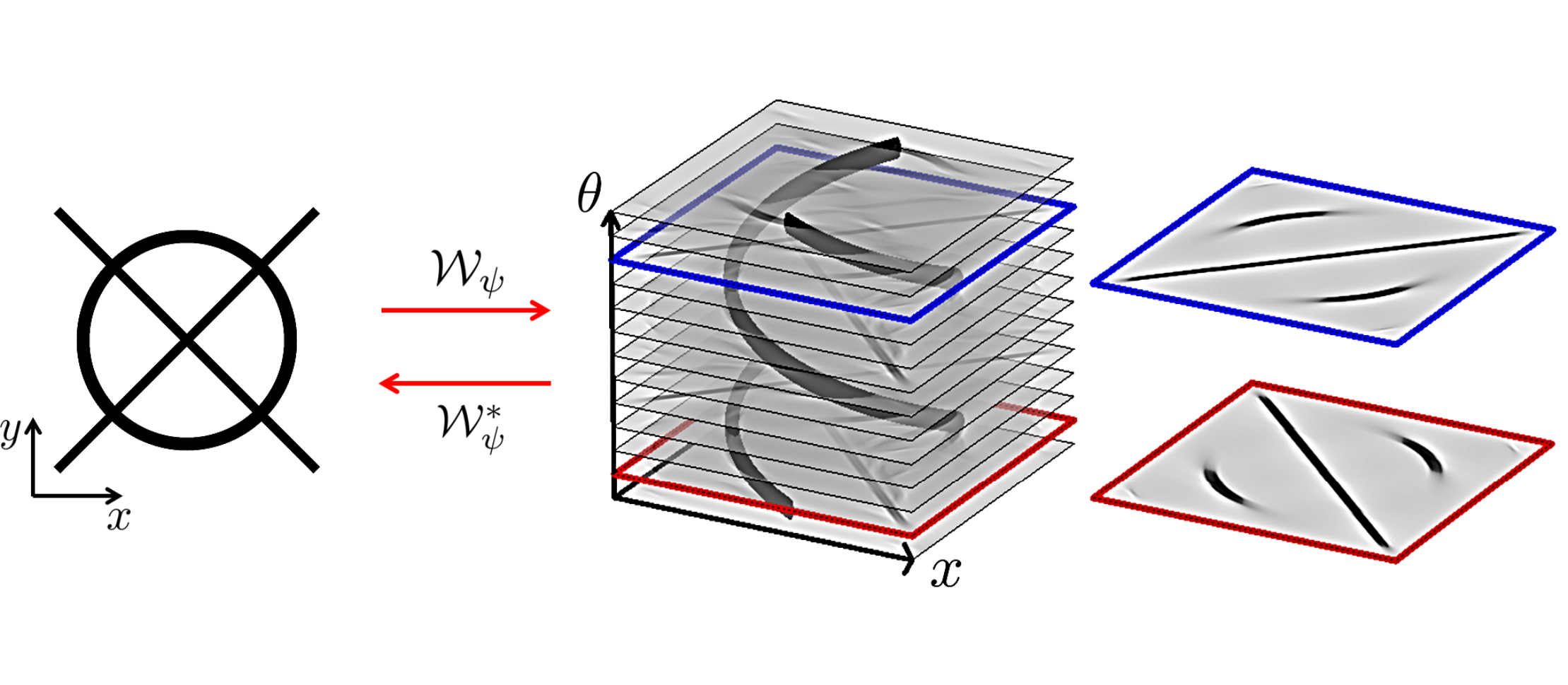}
		\caption{Real part of an orientation score of an example image.}
		\label{fig:OSIntro}
	\end{figure}
	
	Invertible orientation scores (see Figure~\ref{fig:OSIntro}) are obtained via a unitary transform between the space of disk-limited images $\mathbb{L}_{2}^{\varrho}(\R^{2}):=\{f \in \mathbb{L}_{2}(\R^{2}) \; |\; \textrm{support}\{\mathcall{F}_{\R^{2}}f\} \subset B_{\ul{0},\varrho}\}$ (with $\varrho>0$ close to the Nyquist frequency and $B_{\ul{0},\varrho}=\{\www \in \R^{2}\;|\; \|\www\|\leq \varrho\}$), and the space of orientation scores. The space of orientation scores is a specific reproducing kernel vector subspace \cite{DuitsPhDThesis,Aronszajn1950,Alibook} of $\mathbb{L}_{2}(\R^{2}\times S^{1})$, see Appendix~\ref{app:new} for the connection with continuous wavelet theory. The transform from an image $f$ to an orientation score $U_f:=\mathcall{W}_\psi f$ is constructed via an anisotropic convolution kernel $\psi \in \mathbb{L}_{2}(\R^{2}) \!\cap\! \mathbb{L}_{1}(\R^{2})$:
	\begin{equation} \label{OrientationScoreConstruction}
	U_f(\ul{x},\theta)=(\mathcall{W}_\psi [f])(\ul{x},\theta)=\int_{\R^2}\overline{\psi(\ul{R}_{\theta}^{-1}(\ul{y}{-\ul{x}}))}f(\ul{y})d\ul{y},
	\end{equation}
	where $\mathcall{W}_\psi$ denotes the transform and $\small \ul{R}_\theta=
	\left( \begin{array}{ccc}
	\cos\theta & -\sin\theta \\
	\sin\theta & \cos\theta \\
	\end{array} \right).$
	Exact reconstruction is obtained by
	\begin{equation}\label{OrientationScoreReconstruction}
	\begin{aligned}
	f(\ul{x})
	=(\mathcall{W}_\psi^*[U_f])(\ul{x})
	=\left(\mathcall{F}_{\R^2}^{-1}\left[M_\psi^{-1}\mathcall{F}_{\R^2}\left[\frac{1}{2\pi}\int_0^{2\pi}(\psi_\theta*U_f(\cdot,\theta))d\theta\right]\right]\right)(\ul{x}),
	\end{aligned}
	\end{equation}
	for all $\ul{x} \in \R^2$, where $\mathcall{F}_{\R^2}$ is the unitary Fourier transform on $\mathbb{L}_2(\R^2)$ and $M_\psi \in C(\R^2, \R)$ is given by $M_\psi(\pmb{\omega})=\int_0^{2\pi}|\hat{\psi}(\ul{R}_\theta^{-1}\pmb{\omega})|^2 d\theta$ for all $\pmb{\omega} \in \R^{2}$, with $\hat{\psi}:=\mathcall{F}_{\R^2}\psi$, $\psi_{\theta}(\ul{x})=\psi(R_{\theta}^{-1}\ul{x})$. Furthermore, $\mathcall{W}_\psi^*$ denotes the adjoint of wavelet transform $\mathcall{W}_\psi:\mathbb{L}_2(\R^2)\rightarrow \mathbb{C}_{K}^{SE(2)}$, where the reproducing kernel norm on the space of orientation scores, $\mathbb{C}_{K}^{SE(2)}=\{\mathcall{W}_{\psi}f \; |\; f \in \mathbb{L}_{2}(\R^{2})\}$, is explicitly characterized in \cite[Thm.4, Eq.~11]{Duits2007IJCV}. Well-posedness of the reconstruction is controlled by $M_\psi$\cite{Duits2007IJCV,BekkersJMIV}. For details see Appendix~\ref{app:new}. Regarding the choice of $\psi$ in our algorithms, we rely on the wavelets proposed in \cite[ch:4.6.1]{DuitsPhDThesis},\cite{BekkersJMIV}.
	
	In this article, the invertible orientation scores serve as the initial condition of left-invariant $\mbox{(non-)}$ linear PDE evolutions on the rotation-translation group $\R^2 \rtimes SO(2) \equiv SE(2)$, where by definition, \\$\R^d \rtimes S^{d-1}:=\R^d \rtimes SO(d)/(\{0\} \times SO(d-1))$. Now in our case $d=2$, so $\R^2 \rtimes S^1=\R^2 \rtimes SO(2)$ and we identify rotations with orientations. The primary focus of this article, however, is on the numerics and comparison to the exact solutions of linear left-invariant PDE's on $SE(2)$. Here by left-invariance and linearity we can restrict ourselves in our numerical analysis to the impulse response, where the initial condition is equal to  $\delta_e=\delta_0^x \otimes \delta_0^y \otimes \delta_0^\theta$, where $\otimes$ denotes the tensor product in distributional sense.
	
	In fact, we consider all linear, second order, left-invariant evolution equations and their resolvents on $\mathbb{L}_{2}(\R^{2} \rtimes  S^{1}) \equiv \mathbb{L}_2(SE(2))$, which actually correspond to the forward Kolmogorov equations of left-invariant stochastic processes. Specifically, there are two types of stochastic processes we will investigate in the field of imaging and vision:
	\begin{compactitem}
		\item The contour enhancement process as proposed by Citti et al.\cite{Citti} in the cortical modeling.
		\item The contour completion process as proposed by Mumford \cite{Mumford} also called the direction process.
	\end{compactitem}
	In image analysis, the difference between the two processes is that the contour enhancement focuses on the de-noising of elongated structures, while the contour completion aims for bridging the gap of interrupted contours since it contains a convection part.
	
	Although not being considered in this article, we mention related 3D $(SE(3))$ extensions of these processes and applications (primarily in DW-MRI) in \cite{Creusen2013,MomayyezSiakhal2013,ReisertSE3-2012}. Most of our numerical findings in this article apply to these $SE(3)$ extensions as well.
	
	Many numerical approaches for implementing left-invariant PDE's on $SE(2)$ have been investigated intensively in the fields of cortical modeling and image analysis. Petitot introduced a geometrical model for the visual cortex V1 \cite{Petitot}, further refined to the $SE(2)$ setting by Citti and Sarti \cite{Citti}. A method for completing the boundaries of partially occluded objects based on stochastic completion fields was proposed by Zweck and Williams\cite{Zweck}. Also, Barbieri et al.\cite{BarbieriArxiv2013} proposed a left-invariant cortical contour perception and motion integration model within a 5D contact manifold. In the recent work of Boscain et al.\cite{Boscain3}, a numerical algorithm for integration of a hypoelliptic diffusion equation on the group of translations and discrete rotations $SE(2,N)$ is investigated. Moreover, some numerical schemes were also proposed by August et al. \cite{August,AugustPAMI} to understand the direction process for curvilinear structure in images. Duits, van Almsick and Franken\cite{DuitsAMS1,DuitsAMS2,DuitsAlmsick2008,FrankenPhDThesis,MarkusThesis,DuitsPhDThesis} also investigated different models based on Lie groups theory, with many applications to medical imaging.
	
	The numerical schemes for left-invariant PDE's on $SE(2)$ can be categorized into 3 approaches:
	\begin{compactitem}
		\item Finite difference approaches.
		\item Fourier based approaches, including $SE(2)$-Fourier methods.
		\item Stochastic approaches.
	\end{compactitem}
	Recently, several explicit representations of exact solutions were derived \cite{DuitsCASA2005,DuitsCASA2007,DuitsAMS1,MarkusThesis,DuitsAlmsick2008,Boscain1}. In this paper we will set up a structured framework to compare all the numerical approaches to the exact solutions. \\
	\textbf{Contributions of the article:}
	In this article, we:
	\begin{compactitem}
		\item compare all numerical approaches (finite difference methods, a stochastic method based on Monte Carlo simulation and Fourier based methods) to the exact solution for contour enhancement/completion. We show that the Fourier based approaches perform best and we also explain this theoretically in Theorem \ref{th:RelationofFourierBasedWithExactSolution};
		\item provide a concise overview of all exact approaches;
		\item implement exact solutions, including improvements of Mathieu-function evaluations in $\textit{Mathematica}$;
		\item establish explicit connections between exact and numerical approaches for contour enhancement;
		\item analyze the poles/singularities of the resolvent kernels;
		\item propose a new probabilistic time integration to overcome the poles, and we prove this via new simple asymptotic formulas
		for the corresponding kernels that we present in this article;
		\item show benefits of the newly proposed time integration in contour completion via stochastic completion fields \cite{Zweck};
		\item analyze errors when using the $\textbf{DFT}$ (Discrete Fourier Transform) instead of the $\textbf{CFT}$ (Continuous Fourier Transform) to transform exact formulas in the spatial Fourier domain to the $SE(2)$ domain;
		\item apply left-invariant evolutions as preprocessing before tracking the retinal vasculature via the ETOS-algorithm \cite{BekkersJMIV} in optical imaging of the eye.
	\end{compactitem}
	\vspace{1.5ex}
	\textbf{Structure of the article:} In Section 2 we will briefly describe the theory of the $SE(2)$ group and left-invariant vector fields. Subsequently, in Section 3 we will discuss the linear time dependent $\mbox(\text{convection-})$ diffusion processes on $SE(2)$ and the corresponding resolvent equation for contour enhancement and contour completion. In Subsection~\ref{IterationResolventOperators} we provide improved kernels via iteration of resolvent operators and give a probabilistic interpretation.
	Then we show the benefit in stochastic completion fields.
	For completeness, the fundamental solution and underlying probability theory for contour enhancement/completion is explained in Subsection~\ref{section:FundamentalSolutions}.
	
	In Section 4 we will give the full implementations for all our numerical schemes for contour enhancement/completion, i.e. explicit and implicit finite difference schemes, numerical Fourier based techniques, and the Monte-Carlo simulation of the stochastic approaches. Then, in Section 5, we will provide a new concise overview of all three exact approaches in the general left-invariant PDE-setting. Direct relations between the exact solution representations and the numerical approaches are also given in this section. After that, we will provide experiments with different parameter settings and show the performance of all different numerical approaches compared to the exact solutions. Finally, we conclude our paper with applications on retinal images to show the power of our multi-orientation left-invariant diffusion with an application on complex vessel enhancement, i.e. in the presence of crossings and bifurcations.

	\section{The $SE(2)$ Group and Left-invariant Vector Fields}
	\label{section:The $SE(2)$ Group and Left-invariant Vector Fields}

	\subsection{The Euclidean Motion Group $SE(2)$ and Representations}\label{section:The Euclidean Motion Group $SE(2)$ and Group Representations}
	
	An orientation score $U:SE(2) \to \mathbb{C}$ is defined on the Euclidean motion group $SE(2)=\R^2 \rtimes S^1$. The group product on $SE(2)$ is given by
	\begin{equation}
	gg'=(\ul{x},\theta)(\ul{x}',\theta')=(\ul{x}+\ul{R}_\theta \cdot \ul{x}',\theta+\theta'), \quad \textit{for all} \quad g,g' \in SE(2).
	\end{equation}
	The translation and rotation operators on an image $f$ are given by $(\mathcall{T}_\ul{x}f)(\ul{y})=f(\ul{y}-\ul{x})$ and $(\mathcall{R}_\theta f)(\ul{x})=f((\ul{R}_\theta)^{-1}\ul{x})$. Combining these operators yields the unitary $SE(2)$ group representation $\mathcall{U}_g=\mathcall{T}_\ul{x} \circ \mathcall{R}_\theta$. Note that $g h \mapsto \mathcall{U}_{gh}=\mathcall{U}_{g} \mathcall{U}_{h}$ and $\mathcall{U}_{g^{-1}}=\mathcall{U}_{g}^{-1}=\mathcall{U}_{g}^{*}$.
	We have
	\begin{equation}
	\forall g \in SE(2):(\mathcall{W}_\psi \circ \mathcall{U}_g)= (\mathcall{L}_g \circ \mathcall{W}_\psi)
	\end{equation}
	with group representation $g \mapsto \mathcall{L}_{g}$ given by $\mathcall{L}_{g}U(h)=U(g^{-1}h)$, and consequently, the effective operator $\Upsilon:=\mathcall{W}_\psi^* \circ \Phi \circ \mathcall{W}_\psi$ on the image domain (see Figure~\ref{fig:ImageProcessingViaOS}) commutes with rotations and translations if the operator $\Phi$ on the orientation score satisfies
	\begin{align}\label{rel}
	\Phi \circ \mathcall{L}_g=\mathcall{L}_g \circ \Phi, \quad \textit{for all}\quad g \in SE(2).
	\end{align}
	Moreover, if $\Phi$ maps the space of orientation scores onto itself, sufficient condition (\ref{rel}) is also necessary for rotation and translation covariant image processing (i.e. $\Upsilon$ commutes with $\mathcall{U}_{g}$ for all $g \in SE(2)$).
	For details and proof see \cite[Thm.21, p.153]{DuitsPhDThesis}.
	However, operator $\Phi$ should not be right-invariant, i.e. $\Phi$ should not commute with the right-regular representation $g \mapsto \mathcall{R}_{g}$ given by $\mathcall{R}_{g}U(h)=U(hg)$, as  $\mathcall{R}_{g}\mathcall{W}_{\psi}=\mathcall{W}_{\mathcall{U}_{g}\psi}$ and operator $\Upsilon$ should rather take advantage from the anisotropy of the wavelet $\psi$.
	
	We conclude that by our construction of orientation scores \emph{only left-invariant operators are of interest}.
	Next we will discuss the left-invariant derivatives (vector-fields) on smooth functions on $SE(2)$, which we will employ in the PDE of interest presented in Section~\ref{section:The PDE's of Interest}. For an intuition of left-invariant processing on orientation scores (via left-invariant vector fields) see
	Figure~\ref{fig:ImageProcessingViaOS}.
	\begin{figure}
		\centering
		\includegraphics[width=.9\textwidth]{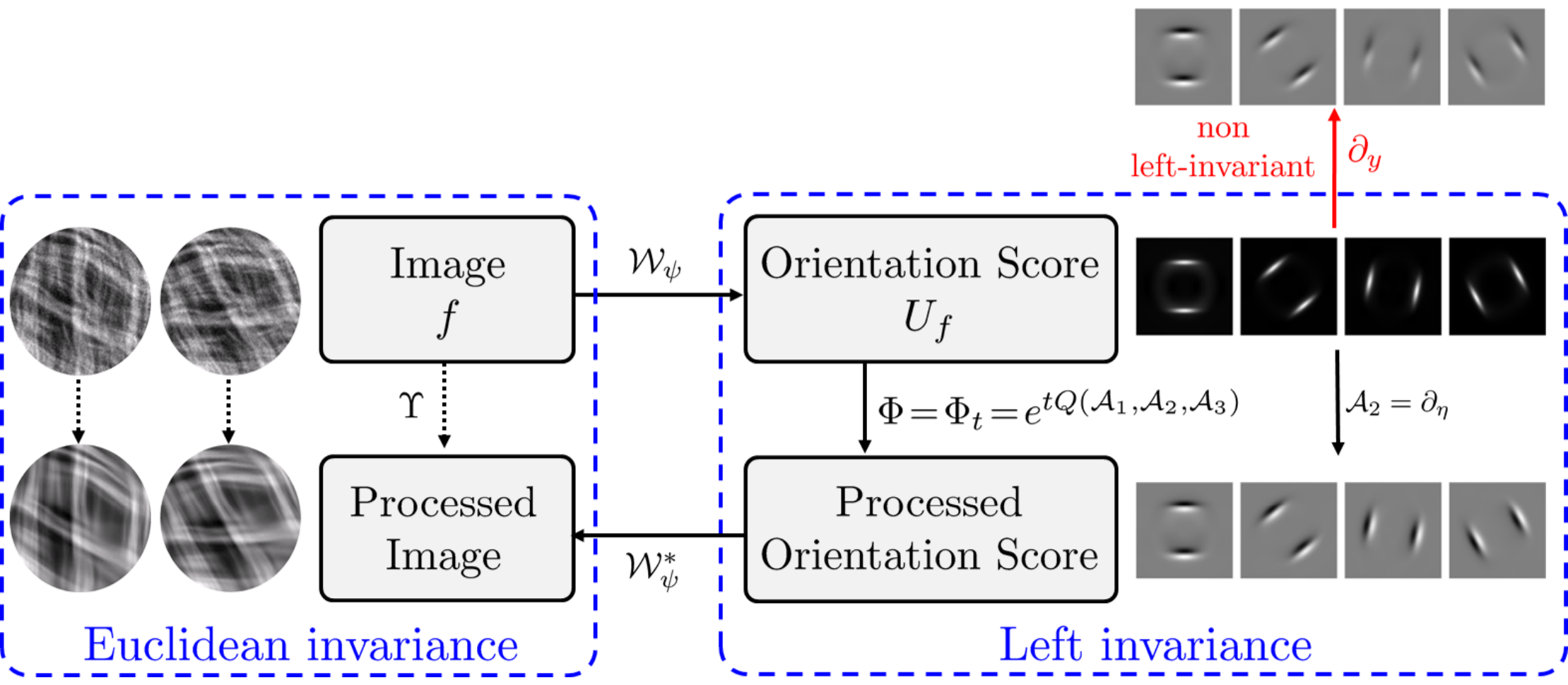}
		\caption{Image processing via invertible orientation scores. Operators $\Phi$ on the invertible orientation score robustly relate to operators $\Upsilon$ on the image domain.  Euclidean-invariance of $\Upsilon$ is obtained by left-invariance of $\Phi$. Thus, we consider left-invariant (convection)-diffusion operators $\Phi=\Phi_t$ with evolution time $t$, which are generated by a quadratic form $Q=Q^{\ul{D},\ul{a}}(\mathcall{A}_1,\mathcall{A}_2,\mathcall{A}_3)$ (
			cf.~\!Eq.~\!(\ref{diffusionconvectiongenerator})) on the left-invariant vector fields $\{\mathcall{A}_i\}$, cf.~\!Eq.~\!(\ref{leftInvariantDerivatives}). We show the relevance of left-invariance of $\mathcall{A}_2$ acting on an image of a circle (as in Figure \ref{fig:OSIntro}) compared to action of the  non-left-invariant derivative $\partial_y$ on the same image. }
		\label{fig:ImageProcessingViaOS}
	\end{figure}

	\subsection{Left-invariant Tangent Vectors and Vector Fields}\label{section:Left-invariant Vector Fields}
	
	The Euclidean motion group $SE(2)$ is a Lie group. Its tangent space at the unity element $T_e(SE(2))$ is the corresponding Lie algebra and it is spanned by the basis $\{\ul{e}_x,\ul{e}_y,\ul{e}_\theta\}$. Next we derive the left-invariant derivatives associated to $\ul{e}_x,\ul{e}_y,\ul{e}_\theta$, respectively.
	A tangent vector $X_e \in T_e(SE(2))$ is tangent to a curve $\gamma$ at unity element $e=(0,0,0)$. Left-multiplication of the curve $\gamma$ with $g \in SE(2)$ associates to each $X_{e} \in T_{e}(SE(2))$ a corresponding tangent vector $X_{g}=(L_{g})_{*}X_{e} \in T_{g}(SE(2))$:
	\begin{equation}
	\begin{aligned}
	\{\ul{e}_\xi(g),\ul{e}_\eta(g),\ul{e}_\theta(g)\} &=\{(L_g)_{*} \ul{e}_x,(L_g)_{*} \ul{e}_y,(L_g)_{*} \ul{e}_\theta\} \\ &=\{\cos\theta\ul{e}_x\!+\!\sin\theta\ul{e}_y,-\sin\theta\ul{e}_x\!+\!\cos\theta\ul{e}_y,\ul{e}_\theta\},
	\end{aligned}
	\end{equation}
	where $(L_g)_*$ denotes the pushforward of left-multiplication $L_gh = gh$, and where we introduce the local coordinates $\xi:= x \cos \theta + y \sin \theta$ and $\eta:= -x \sin \theta + y \cos \theta$.
	As the vector fields can also be considered as differential operators on locally defined smooth functions \cite{aubin2001diffgeo}, we replace $\ul{e}_i$ by using $\partial_i$, $i=\xi,\eta,\theta$, yielding the general form  for a left-invariant vectorfield
	\begin{equation}
	\begin{aligned}
	&X_g(U)=(c^\xi(\cos\theta\partial_x+\sin\theta\partial_y)
	+c^\eta(-\sin\theta\partial_x+\cos\theta\partial_y)+c^\theta\partial_\theta)U,
	\textit{ for all } c^\xi, c^\eta, c^\theta \in \R.
	\end{aligned}
	\end{equation}
	Throughout this article, we shall rely on the following notation for left-invariant vector fields
	\begin{equation} \label{leftInvariantDerivatives}
	\{\mathcall{A}_1,\mathcall{A}_2,\mathcall{A}_3\}:=\{\partial_\xi,\partial_\eta,\partial_\theta\}=\{\cos\theta\partial_x+\sin\theta\partial_y,-\sin\theta\partial_x+\cos\theta\partial_y,\partial_\theta\},
	\end{equation}
	which is the frame of left-invariant derivatives acting on $SE(2)$, the domain of the orientation scores.
	
	\section{The PDE's of Interest} \label{section:The PDE's of Interest}
	
	\subsection{Diffusions and Convection-Diffusions on $SE(2)$ }\label{section:TimedDiffusion}
	A diffusion process on $\mathbb{R}^n$ with a square integrable input image $f:\mathbb{R}^n \longmapsto \mathbb{R}$ is given by
	\begin{align} 
	\left\{\begin{aligned}
	&\partial_t u(\ul{x},t)=\triangledown \cdot \ul{D}\triangledown u(\ul{x},t) \qquad \ul{x}\in\mathbb{R}^n,t \geq 0, \\
	&u(\ul{x},0)=f(\ul{x}).\\
	\end{aligned} \right.
	\end{align}
	Here, the $\triangledown$ operator is defined based on the spatial coordinates with $\triangledown=(\partial_{x_1},...,\partial_{x_n})$, and the constant diffusion tensor $\ul{D}$ is a positive definite matrix of size $n \times n$. Similarly, the left-invariant diffusion equation on $SE(2)$ is given by:
	\begin{align} \label{ExactDiffusionConvectionEquation}
	\left\{\begin{aligned}
	\partial_t W(g,t)&=\left( \begin{array}{ccc}
	\partial_\xi & \partial_\eta & \partial_\theta \end{array} \right)
	\left( \begin{array}{ccc}
	D_{\xi\xi} & D_{\xi\eta} & D_{\xi\theta} \\
	D_{\eta\xi} & D_{\eta\eta} & D_{\eta\theta} \\
	D_{\theta\xi} & D_{\theta\eta} & D_{\theta\theta}\\
	\end{array} \right)
	\left( \begin{array}{ccc}
	\partial_\xi\\
	\partial_\eta\\
	\partial_\theta \end{array} \right)W(g,t),\\
	W(g,t=0)&=U^{0}(g),\\
	\end{aligned} \right.
	\end{align}
	where as a default the initial condition is usually chosen as the orientation score of image $f \in \mathbb{L}_{2}(\R^{2})$, $U^{0}=U_{f}=\mathcall{W}_\psi f$. From the general theory for left-invariant scale spaces \cite{DuitsSSVM2007}, the quadratic form of the convection-diffusion generator is given by
	\begin{equation} \label{diffusionconvectiongenerator}
	\begin{aligned}
	&Q^{\ul{D},\ul{a}}(\mathcall{A}_1,\mathcall{A}_2,\mathcall{A}_3)=\sum_{i=1}^3\left(-a_i\mathcall{A}_i+\sum_{j=1}^3 D_{ij}\mathcall{A}_i \mathcall{A}_j  \right),\\ &a_i,D_{ij} \in \mathbb{R}, \ul{D}:=[D_{ij}] \geq 0, \ul{D}^T=\ul{D},
	\end{aligned}
	\end{equation}
	where the first order part takes care of the convection process, moving along the exponential curves $t \longmapsto g \cdot exp(t(\sum_{i=1}^3 a_iA_i))$ over time with $g \in SE(2)$, and the second order part specifies the diffusion in the following left-invariant evolutions
	\begin{align} \label{diffusionconvection}
	\left\{ \begin{aligned}
	&\partial_t W=Q^{\ul{D},\ul{a}}(\mathcall{A}_1,\mathcall{A}_2,\mathcall{A}_3)W,\\
	&W(\cdot,t=0)=U^{0}(\cdot).\\
	\end{aligned} \right.
	\end{align}
	In case of linear diffusion, the positive definite diffusion matrix $\ul{D}$ is constant. Then we obtain the solution of the left-invariant diffusion equation via a $SE(2)$-convolution with the Green's function $K_t^{\ul{D},\ul{a}}: SE(2)\rightarrow \R^+$ and the
	initial condition $U^{0}:SE(2) \to \mathbb{C}$:
	\begin{equation} \label{SE(2)ConvolutionOnDiffusion}
	\begin{aligned}
	W(g,t) =(K_t^{\ul{D},\ul{a}} \ast_{SE(2)}U^{0})(g) &=\int \limits_{SE(2)}K_t^{\ul{D},\ul{a}}(h^{-1}g)U^{0}(h)\, {\rm d}h \\ &=\int \limits_{\R^2}\int \limits_{-\pi}^{\pi}K_t^{\ul{D},\ul{a}}(\ul{R}_{\theta'}^{-1}(\ul{x}-\ul{x}'), \theta-\theta')U^{0}(\ul{x}',\theta')\, {\rm d}\theta'{\rm d}\ul{x}',
	\end{aligned}
	\end{equation}
	for all $g=(\ul{x},\theta)\in SE(2)$.
	This can symbolically be written as $W(\cdot,t)=e^{tQ^{\ul{D},\ul{a}}(\mathcall{A}_1,\mathcall{A}_2,\mathcall{A}_3)}U^{0}(\cdot)$.
	In this time dependent diffusion we have to set a fixed time $t>0$. In the subsequent sections we consider time integration while imposing a negatively exponential distribution $T \sim NE(\alpha)$, i.e. $P(T=t)=\alpha e^{-\alpha t}$. We choose this distribution since it is the only continuous memoryless distribution, and in order to ensure that the underlying stochastic process is Markovian, traveling time must be memoryless.
	
	There are two specific cases of interest:
	\begin{compactitem}
		\item Contour enhancement, where $\ul{a}=\ul{0}$ and $\ul{D}$ is symmetric positive semi-definite such that the H\"{o}rmander condition is satisfied. This includes both elliptic diffusion $\ul{D}>0$ and hypo-elliptic diffusion in which case we have $\ul{D} \geq 0$ in such a way that H\"{o}rmander's condition \cite{Hoermander} is still satisfied. In the linear case we shall be mainly concerned with the hypo-elliptic case $\ul{D}=\textrm{diag}\{D_{11},0,D_{33}\}$,
		\item Contour completion, where $\ul{a}=(1,0,0)$ and $\ul{D}=\textrm{diag}\{0,0,D_{33}\}$ with $D_{33}>0$.
	\end{compactitem}
	Several new exact representations for the (resolvent) Green's functions in $SE(2)$ were derived by Duits et al. \cite{DuitsAMS1,DuitsAlmsick2008,DuitsCASA2005,DuitsCASA2007,MarkusThesis}  in the spatial Fourier domain, as explicit formulas were still missing, see e.g.~\cite{Mumford}.
	This includes the Fourier series representations, studied independently in \cite{Boscain3}, but also includes a series of rapidly decaying terms and explicit representations obtained by computing the latter series representations explicitly via the Floquet theorem, producing explicit formulas involving only 4 Mathieu functions. The works in \cite{DuitsAMS1,DuitsAlmsick2008} relied to a large extend on distribution theory to derive these explicit formulas. Here we deal with the general case with $D\geq 0$ and $\ul{a} \in \R^{3}$ (as long as H\"{o}rmander's condition
	\cite{Hoermander} is satisfied) and we stress the analogy between the contour completion and contour enhancement case in
	the appropriate general setting (for the resolvent PDE, for the (convection)-diffusion PDE, and for fundamental solution PDE).
	Instead of relying on distribution theory \cite{DuitsAlmsick2008,DuitsAMS1}, we obtain the general solutions more directly via Sturm-Liouville theory.
	
	Furthermore, in Section \ref{section:Experimental results} we include, for the first time, numerical comparisons of various numerical approaches to the exact solutions. The outcome of which, is underpinned by a strong convergence theorem that we will present in Theorem~\ref{th:RelationofFourierBasedWithExactSolution}.
	
	On top of this, in Appendix~\ref{app:A}, we shall present new asymptotic expansions around the origin that allow us to analyze the order of the singularity at the origin of the resolvent kernels. From these asymptotic expansions we deduce that the singularities in the resolvent kernels
	(and fundamental solutions) are quite severely. In fact, the better the equations are numerically approximated, the weaker the completion and enhancement properties of the kernels.
	
	To overcome this severe discrepancy between the mathematical PDE theory and the practical requirements, we propose time-integration via Gamma distributions (beyond the negative exponential distribution case).
	Mathematically, as we will prove in Subsection~\ref{IterationResolventOperators}, this newly proposed time integration both reduces the singularities, and maintains the formal PDE theory. In fact using a Gamma distribution coincides with iteration the resolvents, with an iteration depth $k$ equal to the squared mean divided by the variance of the Gamma distribution.
	
	We will also show practical experiments that demonstrate the advantage of using the Gamma-distributions: we can control and amplify the infilling property ("the spread of ink") of the PDE's.
	
	\subsection{The Resolvent Equation}\label{section:ResolventEquation}
	Traveling time of a memoryless random walker in $SE(2)$ is negatively exponential distributed, i.e.
	\begin{align} \label{exponentialdistribution}
	p(T=t)=\alpha e^{-\alpha t}, t\geq0,
	\end{align}
	with the expected life time $E(T)=\frac{1}{\alpha}$. Then the resolvent kernel is obtained by integrating Green's function $K_t^{\ul{D},\ul{a}}:SE(2)\rightarrow \R^+$ over the time distribution, i.e.
	\[\label{ResolventKernel}
	\begin{aligned}
	R_{\alpha}^{\ul{D},\ul{a}}&=\alpha\int_0^\infty K_t^{\ul{D},\ul{a}}e^{-\alpha t}dt=\alpha\int_0^\infty e^{tQ}\delta_ee^{-\alpha t}dt=-\alpha(Q-\alpha I)^{-1}\delta_e,
	\end{aligned}
	\]
	where we use short notation $Q=Q^{\ul{D},\ul{a}}(\mathcall{A}_1,\mathcall{A}_2,\mathcall{A}_3)$.
	Via this resolvent kernel, one gets the probability density $P_{\alpha}(g)$ of finding a random walker at location
	$g \in SE(2)$ regardless its traveling time, given that it has departed from distribution $U:SE(2) \to \R^{+}$:
	\begin{equation} \label{Resolvent}
	\begin{aligned}
	P_\alpha(g)=(R_{\alpha}^{\ul{D},\ul{a}} \ast_{SE(2)}U)(g)=-\alpha(Q^{\ul{D},\ul{a}}(\mathcall{A}_1,\mathcall{A}_2,\mathcall{A}_3)-\alpha I)^{-1}U(g),
	\end{aligned}
	\end{equation}
	which is the same as taking the Laplace transform of the left-invariant evolution equations ~(\ref{diffusionconvection}) over time. The resolvent equation can be written as
	\[
	\begin{aligned}
	P_\alpha(g)=\alpha\int_0^\infty e^{-\alpha t}(e^{tQ}U^{0})(g)dt=\alpha((\alpha I-Q)^{-1}U)(g).
	\end{aligned}
	\]
	However, we do not want to go into the details of semigroup theory \cite{Yosida} and just included where $(e^{tQ}U^0)$ in short notation for the solution of Eq.~(\ref{diffusionconvection}).
	Resolvents can be used in completion fields\cite{Zweck,DuitsAMS1,August}. Some resolvent kernels of the contour enhancement and completion process are given in Figure~\ref{fig:ResolventCompletionEnhancementKernels}.
	\begin{figure}[t]
		\centering
		\includegraphics[width=0.7\hsize]{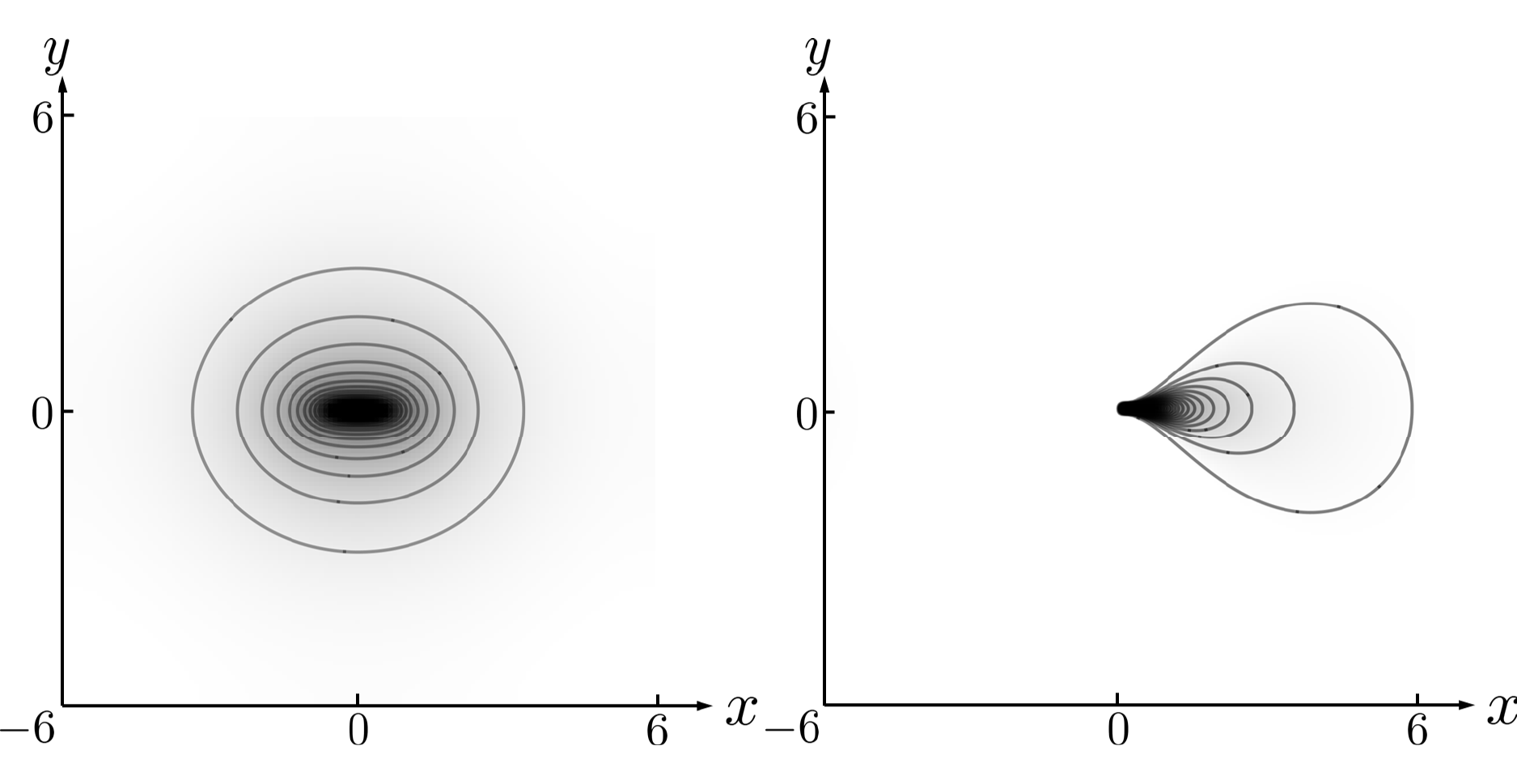}
		\caption{Left: the $xy$-marginal of the contour enhancement kernel $R_{\alpha}^{\ul{D}}:=R_{\alpha}^{\ul{D},\ul{0}}$ with parameters $\alpha=\frac{1}{100}$, $\ul{D}=\{1,0,0.08\}$, numbers of orientations $N_o = 48$ and spatial dimensions $N_s = 128$. Right: the $xy$-marginal of the contour completion kernel $R_{\alpha}^{\ul{D},\ul{a}}$ with parameters $\alpha=\frac{1}{100}$, $\ul{a}=(1,0,0)$, $\ul{D}=\{0,0,0.08\}$, $N_o = 72$ and $N_s = 192$.}
		\label{fig:ResolventCompletionEnhancementKernels}
	\end{figure}
	
	\subsection{Improved Kernels via Iteration of Resolvent Operators \label{IterationResolventOperators}}
	
	The kernels of the resolvent operators suffer from singularities at the origin. Especially for contour completion, this is cumbersome from the application point of view, since here the better one approximates Mumford's direction process and its inherent singularity in the Green's function, the less ``ink'' is spread in the areas with missing and interrupted contours. To overcome this problem we extend the temporal negatively exponential distribution in our line enhancement/completion models by a 1-parameter family of Gamma-distributions.
	
	As a sum $T=T_{1} + \ldots + T_{k}$ of linearly independent negatively exponential time variables is Gamma distributed $P(T=t)= \frac{\alpha^{k} t^{k-1}}{(k-1)!} e^{-\alpha t}$, the time integrated process is now obtained by a $k$-fold resolvent operator. While keeping the expectation of the Gamma distribution fixed by $E(T)=k/\alpha $, increasing of $k$ will induce more mass transport away from $t=0$ towards the void areas of interrupted contours. For $k\geq 3$
	the corresponding Green's function of the $k$-step approach even no longer suffers from a singularity at the origin. This procedure is summarized in the following theorem and applied in Figure~\ref{fig:Gamma}.
	\begin{theorem}\label{th:prob}
		A concatenation of $k$ subsequent, independent time-integrated memoryless stochastic process for contour enhancement(/completion) with expected traveling time $\alpha^{-1}$,
		corresponds to a time-integrated contour enhancement(/completion) process with a Gamma distributed traveling time $T=T_{1}+ \ldots +T_{k}$ with
		\begin{equation}\label{GammaDistributionIntegration}
		\begin{array}{l}
		P(T_{i}=t)=\alpha e^{-\alpha t}, \textrm{ for } i=1,\ldots,k, \\
		P(T=t)=\Gamma(t; k,\alpha):=\frac{\alpha^{k} t^{k-1}}{\Gamma(k)} e^{-\alpha t}.
		\end{array}
		\end{equation}
		The probability density kernel of this stochastic process is given by
		\begin{equation}\label{ProbabilityDensityKernel}
		R_{\alpha,k}^{\ul{D},\ul{a}}=R_{\alpha}^{\ul{D},\ul{a}} *^{(k-1)}_{SE(2)}R_{\alpha}^{\ul{D},\ul{a}}= \alpha^{k} (Q^{\ul{D},\ul{a}}(\underline{\mathcall{A}})-\alpha I)^{-k} \delta_{e},
		\end{equation}
		For the linear, hypo-elliptic, contour enhancement case (i.e. $\ul{D}=\textrm{diag}\{D_{11},0,D_{33}\}$ and $\ul{a}=\ul{0}$) the kernels admit the following asymptotical formula for $|g| << 1:$
		\begin{equation}\label{enhass}
		\begin{array}{ll}
		R_{\alpha,k}(g) &= \int \limits_{0}^{\infty} \frac{\alpha^{k} t^{k-1}e^{-\alpha t}}{(k-1)!}
		\frac{e^{-C^2\frac{|g|^2}{4t}}}{4\pi D_{11}D_{33}t^2} {\rm d}t=
		\frac{\alpha^k}{(k-1)! 4\pi D_{11}D_{33}}  \int \limits_{0}^{\infty}
		t^{k-3}e^{-C^2\frac{|g|^2}{4t}-\alpha t}\,{\rm d}t \\
		&= \frac{2^{1-k}}{\pi D_{11}D_{33} (k-1)!}\alpha^{k}
		||g|C|^{k-2} \; \mathcall{K}(2-k,|g|C\sqrt{\alpha}),
		\end{array}
		\end{equation}
		where $\mathcall{K}(n,z)$ denotes the modified Bessel function of the 2nd kind, and
		with $C \in [2^{-1},\sqrt[4]{2}]$ and with
		\begin{equation} \label{logmodulus}
		|g|=\left|e^{c^{1}\mathcall{A}_{1}+c^{2}\mathcall{A}_{2}+c^{3}\mathcall{A}_{3}}\right|=
		\sqrt{\left(\frac{|c^{1}|^2}{D_{11}}+\frac{|c^{3}|^2}{D_{33}}\right)^2 +\frac{|c^{2}|^2}{D_{11}D_{33}}}
		\end{equation}
		with $c^{1}=\frac{\theta(y-\eta)}{2(1-\cos \theta)}$, $c^{2}=\frac{\theta(\xi-x)}{2(1-\cos \theta)}$, $c^{3}=\theta$ if $\theta \neq 0$ and $(c^{1},c^{2},c^{3})=(x,y,0)$ if $\theta=0$.
	\end{theorem}
	\textbf{Proof }
	We consider a random traveling time $T=\sum_{i=1}^{n} T_{i}$ in an
	$n$-step approach random process $G_{T}=\sum_{i=1}^{N}G_{T_i}$ on $SE(2)$,
	with $G_{T_i}$ independent random random walks whose Fokker-Planck equations are given by (\ref{diffusionconvection}), and with independent traveling times $T_{i} \sim NE(\alpha)$ (i.e. $P(T_{i}=t)=f(t):=\alpha e^{-\alpha t}$).
	Then for $k \geq 2$ we have $T \sim f *_{\R^{+}}^{k-1} f=\Gamma(\cdot; k,\alpha)$, (with $f*_{\R^+}g(t)=\int_{0}^{t} f(\tau)g(t-\tau)\,{\rm d}\tau$),
	which follows by consideration of the characteristic function and application of Laplace transform $\mathcall{L}$.
	
	We have $\alpha^{k}(Q-\alpha I)^{-k}=(\alpha (Q-\alpha I)^{-1})^k$, and for $k=2$ we have
	identity
	\[
	\begin{array}{l}
	R_{\alpha,k=2}^{\ul{D},\ul{a}}(\ul{x},\theta)
	=\int \limits_{0}^{\infty} p(G_{T}=(\ul{x},\theta) | T=t, G_{0}=e)\cdot p(T=t)\, {\rm d}t \\
	=\int \limits_{0}^{\infty} p(G_{T}=(\ul{x},\theta) \; |\; T=T_{1}+T_{2}=t, G_{0}=e)\cdot  p(T_{1}+T_{2}=t) \, {\rm d}t \\
	=\int \limits_{0}^{\infty} \int \limits_{0}^{t} p(G_{T_{1}+T_2}=(\ul{x},\theta) \; |\; T_{1}=t-s, T_{2}=s, G_{0}=e)\cdot
	p(T_{1}=t-s)\; p(T_{2}=s) \, {\rm d}s {\rm d}t \\
	=\alpha^2 \, \mathcall{L}\left(t \mapsto \int \limits_{0}^{t} (K_{t-s}^{\ul{D},\ul{a}}*_{SE(2)}K_{s}^{\ul{D},\ul{a}} *_{SE(2)} \delta_{e})(\ul{x},\theta) {\rm d}s\right)(\alpha)\\
	= \alpha^2 \, \mathcall{L}\left(t \mapsto \int \limits_{0}^{t} (K_{t-s}^{\ul{D},\ul{a}}*_{SE(2)} K_{s}^{\ul{D},\ul{a}} )(\ul{x},\theta) {\rm d}s\right)(\alpha) \\
	= \alpha^2 \, \left(\mathcall{L}\left(t \mapsto K_{t}^{\ul{D},\ul{a}}(\cdot)\right)(\alpha) *_{SE(2)}\mathcall{L}\left(t \mapsto K_{t}^{\ul{D},\ul{a}}(\cdot)\right)(\alpha)\right)(\ul{x},\theta)
	= (R_{\alpha,k=1}^{\ul{D},\ul{a}}*_{SE(2)}R_{\alpha,k=1}^{\ul{D},\ul{a}})(\ul{x},\theta).
	\end{array}
	\]
	Thereby main result Eq.~\!(\ref{ProbabilityDensityKernel}) follows by induction.
	
	Result (\ref{enhass}) follows by direct computation and application of the theory of weighted
	sub-coercive operators on Lie groups \cite{TerElst} to the $SE(2)$ case. We have filtration $\gothic{g}_0:=
	\textrm{span}\{\mathcall{A}_{1},\mathcall{A}_{3}\}$,
	and $\gothic{g}_{1}=[\gothic{g}_0,\gothic{g}_0]=\textrm{span}\{\mathcall{A}_{1},\mathcall{A}_{2},\mathcall{A}_{3}\}=\mathcall{L}(SE(2))$,
	so $w_1=1$, $w_3=1$ and $w_{2}=2$ and computation of the logarithmic map on $SE(2)$,
	$g=e^{\sum_{i=1}^{3}c^{i} A_{i}} \desda \sum_{i=1}^{3}c^{i} A_{i} = \log g$, yields a non-smooth logarithmic squared modulus
	locally equivalent to smooth $|g|^2$ given by (\ref{logmodulus}), see \cite[ch:5.4,eq.5.28]{DuitsAMS1}.
	$\hfill \Box$ \\
	\\
	We have the following asymptotical formula for $\mathcall{K}(n,z)$:
	\[
	\mathcall{K}(n,z)
	\approx
	\left\{
	\begin{array}{ll}
	- \log(z/2) -\gamma_{EUL} & \textrm{if }n=0 \\
	\frac{1}{2}(|n|-1)! \left( \frac{z}{2}\right)^{-|n|}
	\end{array}
	\right.
	\textrm{ for }0 < z <\!<\! 1,
	\]
	with Euler's constant $\gamma_{EUL}$,
	and thereby Eq.~(\ref{enhass}) implies the following result:
	\begin{corollary}\label{corr:X}
		If $k=1$ then $R_{\alpha,k}^{\ul{D}}(g)\equiv O(|g|^{-2})$. If $k=2$ then $R_{\alpha,k}^{\ul{D}}(g)\equiv O(\log |g|^{-1})$.
		If $k\geq 3$ then $R_{\alpha,k}^{\ul{D}}(g)\equiv O(1)$ and the kernel has no singularity.
	\end{corollary}
	\begin{remark}
		As this approach also naturally extends to positive (non-integer) fractional powers $k \in \mathbb{Q}, k\geq 0$ of the resolvent operator we wrote $\Gamma(k)$ instead of $(k-1)!$ in
		Eq.~\!(\ref{GammaDistributionIntegration}). The recursion depth $k$ equals $(E(T))^2/Var(T)$, since the variance of $T$ equals $Var(T)= k/\alpha^2$.
	\end{remark}
	In Figure~\ref{fig:Gamma}, we show that increase of $k$ (while fixing $E(T)=k/\alpha$) allows for better propagation of ink towards the completion areas. The same concept applies to the contour enhancement process. Here we change time integration (using the stochastic approach outlined in Section~\ref{section:MonteCarloStochasticImplementation}) in Eq.~\!(\ref{GammaDistributionIntegration}) rather than iterating the resolvents in Eq.~\!(\ref{ProbabilityDensityKernel}) for better accuracy.
	\begin{figure}
		\centering
		\includegraphics[width=0.85\hsize]{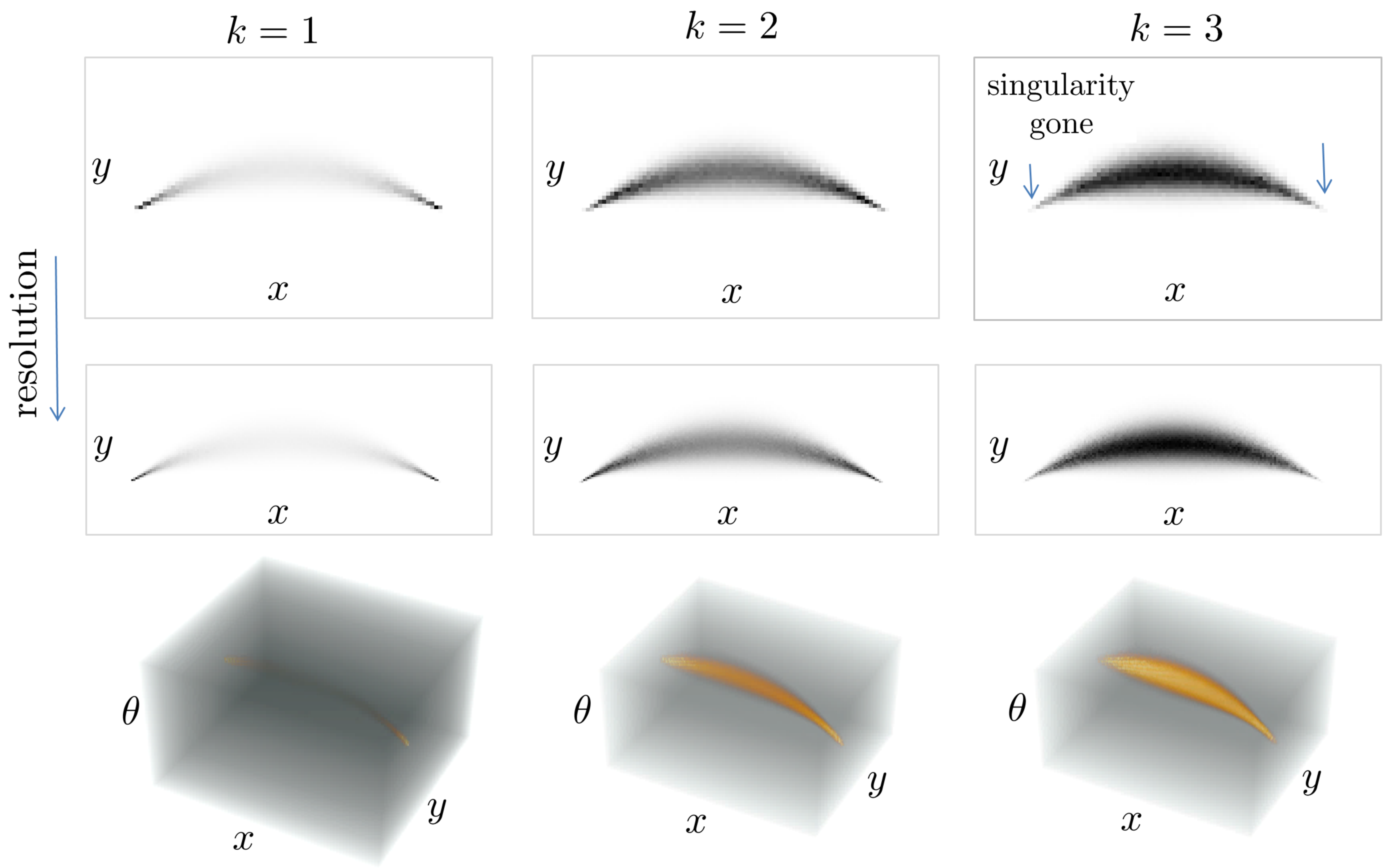}
		\caption{Illustration of Theorem~\ref{th:prob} and Corollary~\ref{corr:X}, via the stochastic implementation for the $k$-step contour completion process ($T=\sum_{i=1}^k T_{i}$) explained in Subsection~\ref{section:MonteCarloStochasticImplementation}. We have depicted the (2D marginals) of 3D completion fields \cite{Zweck} now generalized to
			$\mathcall{C}(x,y,\theta):=((Q-(\alpha k) I)^{-k}\delta_{g_{0}})(x,y,\theta) \cdot ((Q^{*}-(\alpha k) I)^{-k}\delta_{g_{1}})(x,y,\theta)$, with $Q=-\mathcall{A}_{1}+ D_{33} \mathcall{A}_{3}^2$ and with
			$g_0=(\ul{x}_0, \frac{\pi}{6})$ and $g_{1}=(\ul{x}_1, -\frac{\pi}{6})$, $\alpha=0.1$, $D_{33}=0.1$, via a rough resolution
			(on a $200\times 200 \times 32$-grid) and a finer resolution (on a $400\times 400 \times 64$-grid).
			Image intensities have been scaled to full range.
			The resolvent process $k=1$ suffers from: "the better the approximation, the less relative infilling in the completion" (cf.~left column). The singularities at $g_0$
			and $g_{1}$ vanish at $k=3$. A reasonable compromise is found at $k=2$ where infilling is stronger, and where the modes (i.e. curves $\gamma$ with $\mathcall{A}_{2}\mathcall{C}(\gamma)=\mathcall{A}_{3}C(\gamma)=0$, cf.~\cite[App.~A]{BekkersJMIV},\cite{DuitsAlmsick2008}) are easy to detect.  \label{fig:Gamma}}
	\end{figure}
	
	\subsection{Fundamental Solutions\label{section:FundamentalSolutions}}
	
	The fundamental solution $S^{\ul{D},\ul{a}}:SE(2) \to \R^{+}$ associated to generator
	$Q^{\ul{D},\ul{a}}(\mathcall{A}_1,\mathcall{A}_2,\mathcall{A}_3)$
	solves
	\begin{equation} \label{fundsolPDE}
	Q^{\ul{D},\ul{a}}(\mathcall{A}_1,\mathcall{A}_2,\mathcall{A}_3) \; S^{\ul{D},\ul{a}} =-\delta_{e}\ ,
	\end{equation}
	and is given by
	\begin{equation}\label{FundamentalSolution}
	\begin{aligned}
	S^{\ul{D},\ul{a}}(x, y, \theta) &= \int \limits_{0}^{\infty}K_{t}^{\ul{D},\ul{a}}(x,y,\theta)\, {\rm d}t =
	\left(-(Q^{\ul{D},\ul{a}}(\mathcall{A}_1,\mathcall{A}_2,\mathcall{A}_3))^{-1}\delta_e \right)(x,y,\theta) \\
	&=\lim_{\alpha \downarrow 0}\left(\frac{-\alpha(Q^{\ul{D},\ul{a}}(\mathcall{A}_1,\mathcall{A}_2,\mathcall{A}_3)-\alpha I)^{-1}}{\alpha}\delta_e \right)(x,y,\theta)
	=\lim_{\alpha \downarrow 0}\frac{R_{\alpha}^{\ul{D},\ul{a}}(x, y, \theta)}{\alpha}.\\
	\end{aligned}
	\end{equation}
	There exist many intriguing relations \cite{DuitsAMS2,Boscain2} between fundamental solutions of hypo-elliptic diffusions
	and left-invariant metrics on $SE(2)$, which make these solutions interesting. Furthermore, fundamental solutions on the nilpotent approximation $(SE(2))_{0}$ take a relatively simple explicit form \cite{Gaveau,DuitsAMS1}.
	However, by Eq.~(\ref{FundamentalSolution}) these fundamental solutions suffer from some practical drawbacks: they are not probability kernels, in fact they are not even $\mathbb{L}_1$-normalizable, and they suffer from poles in both spatial and Fourier domain. Nevertheless, they are interesting to study for the limiting case $\alpha \downarrow 0$ and they have been suggested in cortical modeling \cite{Barbieri2012,BarbieriArxiv2013}. \\
	\\
	
	\subsection{The Underlying Probability Theory}
	In this section we provide an overview of the underlying probability theory
	belonging to our PDE's of interest, given by Eq.~(\ref{diffusionconvection}), (\ref{Resolvent}) and (\ref{fundsolPDE}).
	
	We obtain the contour enhancement case by setting $\ul{D}=\textrm{diag}\{D_{11},0,D_{33}\}$ and $\ul{a}=\ul{0}$. Then, by application of Eq.~(\ref{leftInvariantDerivatives}), Eq.~(\ref{diffusionconvection}) becomes the forward Kolmogorov equation
	\begin{equation} \label{StochasticEnhancementEvolution}
	\left\{
	\begin{aligned}
	&\partial_t W(g,t)=(D_{11}\partial_\xi^2+D_{33}\partial_\theta^2)W(g,t),\\
	&W(g,t=0)=U(g)\\
	\end{aligned} \right.
	\end{equation}
	of the following stochastic process for contour enhancement:
	\begin{equation} \label{StochasticEnhancementProcess}
	\left\{\begin{aligned}
	&\ul{X}(t)=\ul{X}(0)+\sqrt{2D_{11}}\varepsilon_\xi\int^t_0(\cos\Theta(\tau)\ul{e}_x+\sin\Theta(\tau)\ul{e}_y)\frac{1}{2\sqrt{\tau}}\,{\rm d}\tau\\
	&\Theta(t)=\Theta(0)+\sqrt{t}\sqrt{2D_{33}}\varepsilon_\theta,\qquad\varepsilon_\xi,\varepsilon_\theta\thicksim\mathcall{N}(0,1)\\
	\end{aligned} \right.
	\end{equation}
	
	For contour completion, we must set the diffusion matrix $\ul{D}=\textrm{diag}\{0,0,D_{33}\}$ and convection vector $\ul{a}=(1,0,0)$. In this case Eq.~(\ref{diffusionconvection}) takes the form
	\begin{equation} \label{StochasticCompletionEvolution}
	\left\{
	\begin{aligned}
	&\partial_t W(g,t)=(\partial_\xi+D_{33}\partial_\theta^2)W(g,t),\qquad g\in SE(2), t>0,\\
	&W(g,t=0)=U(g).\\
	\end{aligned} \right.
	\end{equation}
	This is the Kolmogorov equation of Mumford's direction process \cite{Mumford}
	\begin{equation} \label{eq:MumfordDirectionProcess}
	\left\{\begin{aligned}
	&\ul{X}(t)=X(t)\ul{e}_x+Y(t)\ul{e}_y=\ul{X}(0)+\int^t_0 \cos\Theta(\tau)\ul{e}_x+\sin\Theta(\tau)\ul{e}_y\,{\rm d}\tau\\
	&\Theta(t)=\Theta(0)+\sqrt{t}\sqrt{2D_{33}}\varepsilon_\theta,\qquad\varepsilon_\theta\thicksim\mathcall{N}(0,1)\\
	\end{aligned} \right.
	\end{equation}
	
	\begin{remark}
		As contour completion processes aim to reconstruct the missing parts of interrupted contours based on the contextual information of the data, a positive direction $\ul{e}_{\xi}=\cos(\theta)\ul{e}_x+\sin(\theta)\ul{e}_y$ in the spatial plane is given to a random walker.
		On the contrary, in contour enhancement processes a bi-directional movement of a random walker along $\pm\ul{e}_{\xi}$ is included for noise removal by anisotropic diffusion.
	\end{remark}
	The general stochastic process on $SE(2)$ underlying Eq.~(\ref{diffusionconvection}) is :
	{\small
		\begin{equation} \label{eq:form}
		\left\{
		\begin{array}{l}
		G_{n+1}:=(X_{n+1},\Theta_{n+1})=G_n + \Delta t \sum \limits_{i \in I} a_{i} \left.\ul{e}_{i}\right|_{G_n}  +\sqrt{\Delta t}\sum \limits_{i \in I} \epsilon_{i, n+1}\,\sum \limits_{j \in I} \sigma_{ji}\,
		\left. \ul{e}_{j}\right|_{G_n},  \\
		G_{0}=(X_{0},\Theta_{0}),
		\end{array}
		\right.
		\end{equation}
	}
	with $I = \{1,2,3\} $ in the elliptic case and $I = \{1,3\}$ in the hypo-elliptic case and where $n =1,\ldots, N-1$, $N \in \mathbb{N}$ denotes the number of steps with stepsize $\Delta t >0$, $\sigma=\sqrt{2D}$ is the unique symmetric positive definite matrix such that $\sigma^2=2D$, $\{\epsilon_{i, n+1}\}_{i \in I, n =1,\ldots, N-1 }$ are independent normally distributed \mbox{$\epsilon_{i, n+1} \sim \mathcall{N}(0,1)$} and {\small $\left. \ul{e}_{1} \right|_{G_{n}}=(\cos \Theta_{n},\sin \Theta_{n},0)$, $\left. \ul{e}_{2} \right|_{G_{n}}=(-\sin \Theta_{n},\cos \Theta_{n},0)$, and $\left. \ul{e}_{3} \right|_{G_{n}}=(0,0,1)$}. In case $I = \{1,2,3\}$, Eq.~(\ref{eq:form}) boils down to:
	{
		\begin{equation}
		\begin{array}{l}
		\begin{array}{l}
		\left(
		\begin{array}{c}
		X_{n+1} \\
		Y_{n+1} \\
		\Theta_{n+1}
		\end{array}
		\right)=
		\left(
		\begin{array}{c}
		X_{n} \\
		Y_{n} \\
		\Theta_{n}
		\end{array}
		\right)+
		\Delta t \,
		{\rm R}_{\Theta_n}
		\left(
		\begin{array}{c}
		a_{1} \\
		a_{2} \\
		a_{3}
		\end{array}
		\right)
		+
		\sqrt{\Delta t}\,
		({\rm R}_{\Theta_n})^{T} \,
		\sigma \,
		\rm{ R}_{\Theta_n}
		\left(
		\begin{array}{c}
		\epsilon_{1,n+1} \\
		\epsilon_{2,n+1} \\
		\epsilon_{3,n+1}
		\end{array}
		\right),\\
		\textrm{ with }{\rm R}_{\theta}=
		\left(
		\begin{array}{ccc}
		\cos \theta & -\sin \theta & 0 \\
		\sin \theta & \cos \theta & 0 \\
		0 & 0 & 1
		\end{array}
		\right).
		\end{array}
		\end{array}
		\end{equation}
	}
	See Figure~\ref{figure:StochasticRandomWalkerCompletionEnhancementResult} for random walks of the Brownian motion and the direction process in $SE(2)$.
	\begin{figure}[!htbp]
		\centering
		\includegraphics[width=0.7\hsize]{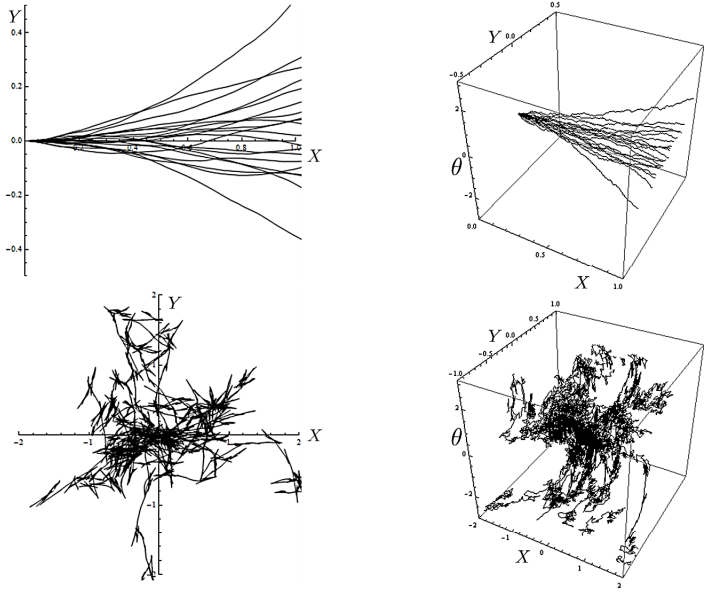}
		\caption{From left to right: Up row: 20 random walks of the direction process for contour completion in $SE(2)=\R^2 \rtimes S^1$ by Mumford \cite{Mumford} with $\ul{a}=(1,0,0)$, $D_{33}=0.3$, time step $\triangle t$=0.005 and 1000 steps. Bottom row: 20 random walks of the linear left-invariant stochastic processes for contour enhancement within $SE(2)$ with parameter settings $D_{11}=D_{33}=0.5$ and $D_{22}=0$, time step $\triangle t$=0.05 and 1000 steps.}
		\label{figure:StochasticRandomWalkerCompletionEnhancementResult}
	\end{figure}
	
	\section{Implementation} \label{section:Implementation}

	\subsection{Left-invariant Differences} \label{section:Left-invariantDifferences}
	
	\subsubsection{Left-invariant Finite Differences with B-Spline Interpolation} \label{section: Left-invariant Finite Differences with B-spline Interpolation}
	As explained in Section \ref{section:The Euclidean Motion Group $SE(2)$ and Group Representations}, our diffusions must be left-invariant. Therefore, a new grid template based on the left-invariant frame $\{\ul{e}_\xi,\ul{e}_\eta,\ul{e}_\theta\}$, instead of the fixed frame $\{\ul{e}_x,\ul{e}_y,\ul{e}_\theta\}$, need to be used in the finite difference methods.
	\begin{figure}[!htbp]
		\centering
		\subfloat{\includegraphics[width=0.9\hsize]{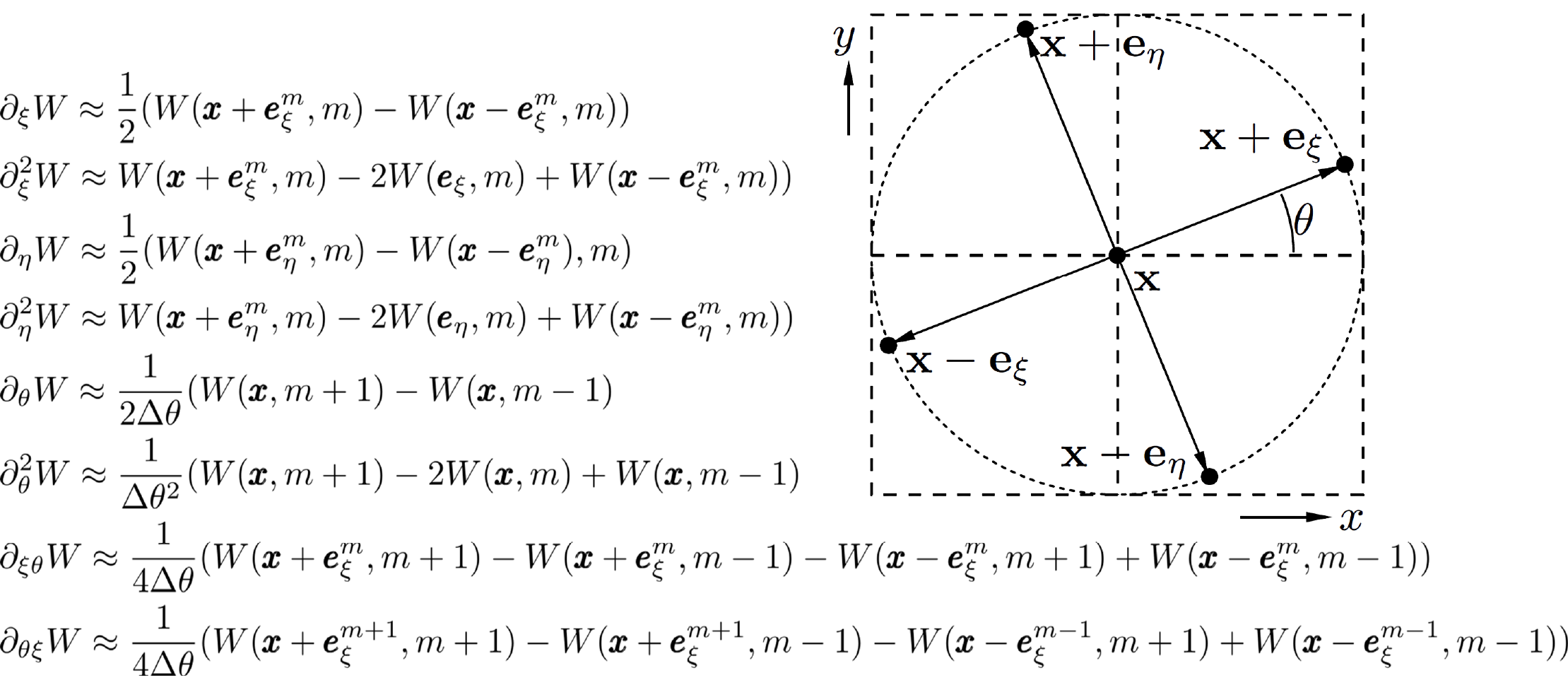}}
		\caption{Illustration of the spatial part of the stencil of the numerical scheme. The horizontal and vertical dashed lines indicate the sampling grid, which is aligned with $\{\ul{e}_x,\ul{e}_y\}$. The black dots, which are aligned with the rotated left-invariant coordinate system $\{\ul{e}_\xi,\ul{e}_\eta\}$ with $\theta=m \cdot \Delta\theta$, where $m \in \{0,1,...,N_o-1\}$ denotes the sampled orientation equidistantly sampled with distance $\Delta \theta = \frac{2\pi}{N_o}$.}
		\label{fig:finiteDifferenceScheme}
	\end{figure}
	To understand how left-invariant finite differences are implemented, see Figure~\ref{fig:finiteDifferenceScheme}, where 2nd order B-spline interpolation \cite{Unser1993} is used to approximate off-grid samples.
	The main advantage of this left-invariant finite difference scheme is the improved rotation invariance compared to finite differences applied after expressing the PDE's in fixed $(x,y,\theta)$-coordinates, such as in \cite{Boscain2,FrankenPhDThesis,Zweck}. This advantage is clearly demonstrated in \cite[Fig.~10]{Franken2009IJCV}. The drawback, however, is the low computational speed and a small amount of additional blurring caused by the interpolation scheme \cite{FrankenPhDThesis}.
	
	\subsection{Left-invariant Finite Difference Approaches for Contour Enhancement and Completion}
	\label{section:Left-invariant Finite Difference Approaches for Contour Enhancement}
	Eq.~(\ref{StochasticEnhancementEvolution}) of the contour enhancement process and Eq.~(\ref{StochasticCompletionEvolution}) of the contour completion process show us respectively the Brownian motion and direction process of oriented particles moving in $SE(2)\equiv \R^2 \rtimes S^1$. Next we will provide and analyze finite difference schemes for both processes.
	
	\subsubsection{Explicit Scheme for Linear Contour Enhancement and Completion}\label{section:ExplicitSchemeforLinearContourEnhancementCompletion}
	We can represent the explicit numerical approximations of the contour enhancement process and contour completion process by using the generator
	$Q^{\ul{D},\ul{a}}(\mathcall{A}_1,\mathcall{A}_2,\mathcall{A}_3)$ in a general form, i.e. $Q^{\ul{D},\ul{a}}(\mathcall{A}_1,\mathcall{A}_2,\mathcall{A}_3) = (D_{11}\mathcall{A}_1^2+D_{33}\mathcall{A}_3^2) = (D_{11}\partial_\xi^2+D_{33}\partial_\theta^2)$ for the diffusion process and $Q^{\ul{D},\ul{a}}(\mathcall{A}_1,\mathcall{A}_2,\mathcall{A}_3)=(\partial_\xi+D_{33}\partial_\theta^2)$ for
	the convection-diffusion process, which yield the following forward Euler discretization:
	\begin{align} \label{forwardEuler}
	\left\{ \begin{aligned}
	&W(g,t+\Delta t)=W(g,t)+\Delta t \, Q^{\ul{D},\ul{a}}(\mathcall{A}_1,\mathcall{A}_2,\mathcall{A}_3) \, W(g,t),\\
	&W(g,0)=U_f(g).\\
	\end{aligned} \right.
	\end{align}
	We take the centered 2nd order finite difference scheme with B-spline interpolation as shown in Figure~\ref{fig:finiteDifferenceScheme} to numerically approximate the diffusion terms $(D_{11}\partial_\xi^2+D_{33}\partial_\theta^2)$, and use upwind finite differences for $\partial_\xi$. In the forward Euler discretization, the time step $\Delta t$ is critical for the stability of the algorithm. Typically, the convection process and the diffusion process have different properties on the step size $\Delta t$. The convection requires time steps equal to the spatial grid size ($\Delta t=\Delta x$) to prevent the additional blurring due to interpolation, while the diffusion process requires sufficiently small $\Delta t$ for stability, as we show next. In this combined case, we simulate the diffusion process and convection process alternately with different step size $\Delta t$ according to the splitting scheme in \cite{Creusen2013}, where half of the diffusion steps are carried out before one step convection, and half after the convection.
	
	The resolvent of the (convection-)diffusion process can be obtained by integrating and weighting each evolution step with the negative exponential distribution in Eq.~(\ref{exponentialdistribution}). We set the parameters $\ul{a}=(1,0,0)$ and $\ul{D}=\textrm{diag} \{1,0,D_{33}\}$ with $D_{33}=\frac{D_{33}}{D_{11}}\approx0.01$ to avoid too much blurring on $S^{1}$.
	\begin{remark}
		Referring to the stability analysis of Franken et al.\cite{Franken2009IJCV} in the general gauge frame setting, we similarly obtain: $\Delta t \leq \frac{1}{2(1+\sqrt{2}+\frac{1}{q^2})}$ in our case of normal left-invariant derivatives.
		For a typical value of $q=\frac{\Delta\theta}{\beta}=\frac{(\pi/24)}{0.1}$ in our convention with $\beta^2:=\frac{D_{33}}{D_{11}} = 0.01$, in which $D_{33} = 0.01$ and $D_{11} = 1$, cf.~\cite{DuitsJMIV2014b}, we obtain stability bound $\Delta t \leq 0.16$ in the case of contour enhancement Eq.~(\ref{StochasticEnhancementEvolution}).
	\end{remark}
	
	\subsubsection{Implicit Scheme for Linear Contour Enhancement and Completion}
	The implicit scheme of the contour enhancement and contour completion is given by:
	\begin{align} \label{ImplicitScheme}
	\left\{ \begin{aligned}
	&W(g,t+\Delta t)=W(g,t)+\Delta t \, Q^{\ul{D},\ul{a}}(\mathcall{A}_1,\mathcall{A}_2,\mathcall{A}_3) \, W(g,t+\Delta t),\\
	&W(g,0)=U_f(g).\\
	\end{aligned} \right.
	\end{align}
	Then, the equivalent discretization form of the Euler equation can be written as:
	\begin{align} \label{DiscretizationImplicitScheme}
	\left\{ \begin{aligned}
	&\ul{w}^{s+1}=\ul{w}^s+\hat{\ul{Q}}\ul{w}^{s+1},\\
	&\ul{w}^1=\ul{u},\\
	\end{aligned} \right.
	\end{align}
	in which $\hat{\ul{Q}} \equiv \Delta t (Q^{\ul{D},\ul{a}}(\mathcall{A}_1,\mathcall{A}_2,\mathcall{A}_3))$, and $\ul{w}^s$ is the solution of the PDE at $t=(s-1)\Delta t, s \in \{1,2,...\}$, with the initial state $\ul{w}^1=\ul{u}$. According to the conjugate gradient method as shown in \cite{Creusen2013}, we can approximate the obtained linear system $(\ul{I}-\hat{\ul{Q}})\ul{w}^{s+1}=\ul{w}^s$ iteratively without evaluating matrix $\hat{\ul{Q}}$ explicitly. The advantage of an implicit method is that it is unconditionally stable, even for large step sizes.
	
	\subsection{Numerical Fourier Approaches \label{section:Duitsmatrixalgorithm}}
	
	The following numerical scheme is a generalization of the numerical scheme proposed by Jonas August for the direction process \cite{August}.
	An advantage of this scheme over others, such as the algorithm by Zweck et al. \cite{Zweck} or other finite difference schemes \cite{Franken2009IJCV}, is that (as we will show later in Theorem \ref{th:RelationofFourierBasedWithExactSolution}) it is directly related to the exact analytic solutions (approach 1) presented in Section~\ref{3GeneralFormsExactSolutions}.
	
	The goal is to obtain a numerical approximation of the exact solution of
	\begin{equation} \label{theeqn}
	\alpha(\alpha I-Q^{\ul{D},\ul{a}}(\underline{\mathcall{A}}))^{-1}U=P, \,   U \in \mathbb{L}_{2}(G), \quad \textit{with} \quad \underline{\mathcall{A}}=(\mathcall{A}_1,\mathcall{A}_2,\mathcall{A}_3),
	\end{equation}
	where the generator $Q^{\ul{D},\ul{a}}(\underline{\mathcall{A}})$ is given in the general form Eq.~(\ref{diffusionconvectiongenerator})
	without further assumptions on the parameters $a_{i}>0$, $D_{ii}>0$. Recall that its solution is given by $SE(2)$-convolution with the corresponding kernel. First we write
	\begin{equation} \label{ansatz}
	\begin{array}{l}
	\mathcall{F}[P(\cdot,e^{i\theta})](\www)=\hat{P}(\www,e^{i\theta})= \sum \limits_{l=-\infty}^{\infty} \hat{P}^{l}(\www) e^{i l \theta}, \\
	\mathcall{F}[U(\cdot,e^{i\theta})](\www)=\hat{U}(\www,e^{i\theta})= \sum \limits_{l=-\infty}^{\infty} \hat{U}^{l}(\www) e^{i l \theta}. \\
	\end{array}
	\end{equation}
	Then by substituting (\ref{ansatz}) into (\ref{theeqn}) we obtain the following 4-fold recursion
	{\small
		\begin{equation} \label{5recursion}
		\begin{array}{l}
		(\alpha \!+\!l^2 D_{33}\!+\! i\, a_{3} l+\frac{\rho^2}{2}(D_{11}+D_{22}))\hat{P}^{l}(\www) + \frac{ a_1(i\, \omega_x
			\!+\! \omega_{y})\!+\!a_2(i \, \omega_y \!-\!\omega_{x})}{2} \hat{P}^{l-1}(\www)\\+\frac{ a_1(i\, \omega_x\!-\! \omega_{y})+a_2(i \, \omega_y \!+\!\omega_{x})}{2} \hat{P}^{l+1}(\www)
		-
		\frac{ D_{11}(i\, \omega_x\!+\! \omega_{y})^2\!+\!D_{22}(i \, \omega_y \!-\!\omega_{x})^2}{4}
		\hat{P}^{l-2}(\www) \\
		-
		\frac{ D_{11}(i\, \omega_x\!-\! \omega_{y})^2+D_{22}(i \, \omega_y \!+\! \omega_{x})^2}{4}
		\hat{P}^{l+2}(\www) = \alpha \, \hat{U}^{l}(\www),
		\end{array}
		\end{equation}
	}
	which can be rewritten in polar coordinates
	\begin{equation} \label{recurs}
	\begin{array}{l}
	(\alpha + i l a_3 +D_{33}l^2+ \frac{\rho^2}{2}(D_{11}+D_{22})) \, \tilde{P}^{l}(\rho)+ \frac{\rho}{2}(i a_{1}-a_2)\, \tilde{P}^{l-1}(\rho)+ \\
	\frac{\rho}{2}(i a_{1}+a_2) \, \tilde{P}^{l+1}(\rho) + \frac{\rho^2}{4}(D_{11}-D_{22})\, (\tilde{P}^{l+2}(\rho)+\tilde{P}^{l-2}(\rho))=
	\alpha \, \tilde{U}^{l}(\rho)
	\end{array}
	\end{equation}
	for all $l=0,1,2,\ldots$ with $\tilde{P}^{l}(\rho) = e^{il\varphi} \hat{P}^{l}(\www)$ and $\tilde{U}^{l}(\rho) = e^{il\varphi} \hat{U}^{l}(\www)$, with
	$\www=(\rho \cos \varphi, \rho \sin \varphi)$.
	Equation (\ref{recurs}) can be written in matrix-form, where a 5-band matrix must be inverted. For explicit representation
	of this 5-band matrix where the spatial Fourier transform in (\ref{ansatz}) is replaced by the $\textbf{DFT}$ we refer to \cite[p.230]{DuitsPhDThesis}. Here we stick to a Fourier series on $\mathbb{T}$, $\textbf{CFT}$ on $\R^2$ and  truncation of the series at $N \in \mathbb{N}$ which yields the
	$(2N+1) \times (2N+1)$ matrix equation:
	\begin{equation} \label{MatrixInverse}
	{\tiny \left(
		\begin{array}{ccccccc}
		p_{-N} & q+t & r & 0 & 0 & 0 & 0 \\
		q-t & p_{-N+1} & q+t & r & 0 & 0 & 0  \\
		r & \ddots & \ddots & \ddots &  r & 0 & 0 \\
		0 & \ddots & q-t & p_{0} & q+t & r & 0 \\
		0 & 0 & r  & \ddots & \ddots & \ddots & r \\
		0 & 0 & 0  & r & q-t & p_{N-1} & q+t \\
		0 & 0 & 0 & 0  & r & q-t & p_{N}
		\end{array}
		\right)
		\left(
		\begin{array}{c}
		\tilde{P}^{-N}(\rho) \\
		\tilde{P}^{-N+1}(\rho) \\
		\vdots \\
		\tilde{P}^{0}(\rho) \\
		\vdots \\
		\tilde{P}^{N-1}(\rho)
		\\
		\tilde{P}^{N}(\rho)
		\end{array}
		\right)=
		\frac{4 \alpha}{ D_{11}} \!
		\left(
		\begin{array}{c}
		\tilde{U}^{-N}(\rho) \\
		\tilde{U}^{-N+1}(\rho) \\
		\vdots \\
		\tilde{U}^{0}(\rho) \\
		\vdots \\
		\tilde{U}^{N-1}(\rho)
		\\
		\tilde{U}^{N}(\rho)
		\end{array}
		\right)
	}
	\end{equation}
	where $p_{l}= (2l)^2 + \frac{4 \alpha + 2 \rho^2(D_{11}+D_{22})+4 i a_{3} l}{D_{33}}$, $r=\frac{\rho^2(D_{11}-D_{22})}{D_{33}}$, $q= \frac{2 \rho a_{1}i}{D_{33}}$ and $t= \frac{2 a_2 \rho}{D_{33}}.$
	\begin{remark}\label{rem:41}
		The four-fold recursion Eq.~(\ref{recurs}) is uniquely determined by $\tilde{P}_{-N-1}=0, \tilde{P}_{-N-2}=0$, $\tilde{P}_{N+1}=0, \tilde{P}_{N+2}=0$, which is applied in Eq.~(\ref{MatrixInverse}).
	\end{remark}
	\begin{remark}\label{rem:42}
		When applying the Fourier transform on $SE(2)$ to the PDE's of interest, as done in \cite{DuitsAlmsick2008,Boscain3,Boscain2}, one obtains a fully isomorphic 5-band matrix system as pointed out in \cite[App.A, Lemma A.1, Thm A.2]{DuitsAlmsick2008}, the basic underlying coordinate transition to be applied is given by
		\[
		(p,\phi)= (\rho,\varphi - \theta)
		\]
		where $p$ indexes the irreducible representations of $SE(2)$ and $\phi$
		denotes the angular argument of the $p$-th irreducible function subspace $\mathbb{L}_{2}(S^{1})$ on which
		the $p$-th irreducible representation acts. For further details see \cite[App.A]{DuitsAlmsick2008} and \cite{Chirikjian}.
	\end{remark}
	In \cite{DuitsAlmsick2008}, we showed the relation between spectral decomposition of this matrix (for $N \to \infty$) and the exact solutions of contour completion. In this paper we do the same for the contour enhancement case in Section \ref{section:FourierBasedForEnhancement}.
	
	\subsection{Stochastic Implementation}\label{section:MonteCarloStochasticImplementation}
	In a Monte-Carlo simulation as proposed in \cite{Gonzalo,BarbieriArxiv2013}, we sample the stochastic process (Eq.~\!(\ref{eq:form})) such that we obtain the kernels for our linear left-invariant diffusions. In particular the kernel of the contour enhancement process, and the kernel for the contour completion process. Figure~\ref{figure:MentoCarloSimulation} shows the xy-Marginal of the enhancement and the completion kernel, which were obtained by counting the number of paths crossing each voxel in the orientation score domain. In addition, the length of each path follows a negative exponential distribution.
	Within Figure~\ref{figure:MentoCarloSimulation} we see, for practically reasonable parameter settings, that increasing the number of sample paths to 50000 already provides a reasonable approximation of the exact kernels.
	In addition, each path was weighted using the negative exponential distribution with respect to time in Eq.~\!(\ref{exponentialdistribution}), in order to obtain the resolvent kernels.
	\begin{figure}[!htb]
		\centering
		{\includegraphics[width=\textwidth]{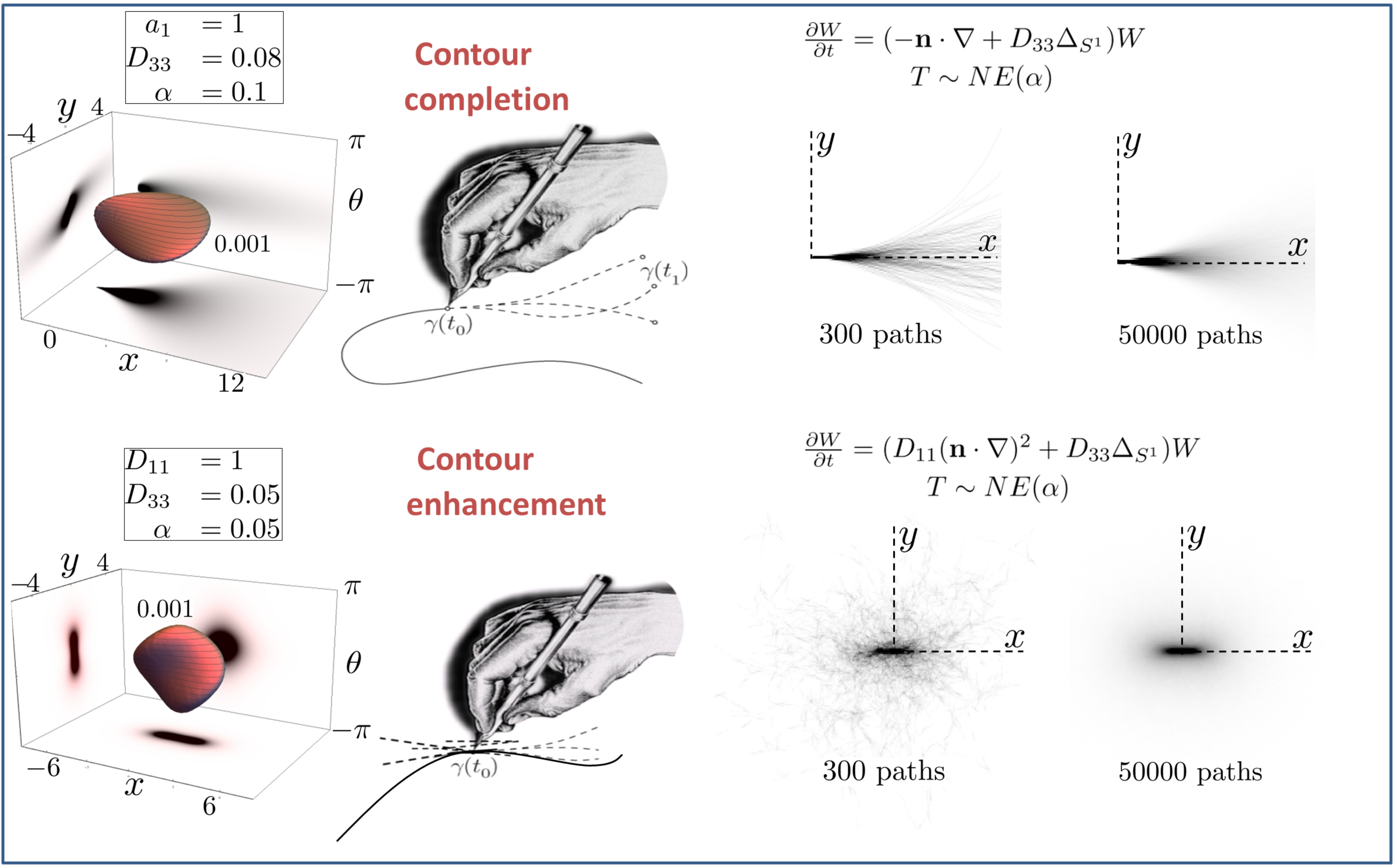}}
		\caption{Stochastic random process for the contour enhancement kernel (top) and stochastic random process for the contour completion raw kernel (bottom). Both processes are obtained via Monte Carlo simulation of random process
			(\ref{eq:form}). In contour completion, we set step size $\Delta t=0.05, \alpha=10, D_{11}=D_{33}=0.5$, and $D_{22}=0$. In contour completion, we set step size $\Delta t=0.005, \alpha=5, D_{33}=1$, and $\ul{a}=(1,0,0)$.}
		\label{figure:MentoCarloSimulation}
	\end{figure}
	The implementation of the $k$-fold resolvent kernels is obtained by application of Theorem~\ref{th:prob}, i.e. by imposing a Gamma distribution instead of a negatively exponential distribution. Here stochastic implementations become slower as one can no longer rely on the memoryless property of the negatively exponential distribution, which means one should only take the end-condition of each sample path $G_T$ after a sampling of random traveling time $T\sim\Gamma(t;k,\alpha)$. Still such stochastic implementations are favorable (in view of the singularity) over the concatenation of $SE(2)$-convolutions of the resolvent kernels with themselves.
	
	\section{Implementation of the Exact Solution in the Fourier and the Spatial Domain and their Relation to Numerical Methods}\label{section:Comparison}
	In previous works by Duits and van~Almsick \cite{DuitsCASA2005,DuitsCASA2007,DuitsAlmsick2008}, three methods were applied producing three different exact representations for the kernels (or "Green's functions") of the forward Kolmogorov equations of the contour completion process:
	\begin{enumerate}
		\item The first method involves a spectral decomposition of the bi-orthogonal generator in the $\theta$-direction for each fixed spatial frequency $(\omega_{x},\omega_y)=(\rho \cos\varphi, \rho \sin\varphi) \in \R^{2}$ which is an unbounded Mathieu operator, producing a (for reasonably small times $t>0$) \emph{slowly converging} Fourier series representation. Disadvantages include the Gibbs phenomenon. Nevertheless, the Fourier series representation in terms of \emph{periodic} Mathieu functions directly relates to the numerical algorithm proposed by August in \cite{August}, as shown in \cite[ch:5]{DuitsAlmsick2008}. Indeed the Gibbs phenomenon appears in this algorithm as the method requires some smoothness of data: running the algorithm on a sharp discrete delta-spike provides Gibbs-oscillations. The same holds for Fourier transform on $SE(2)$ methods \cite{DuitsAlmsick2008,Boscain3,Boscain2}, recall Remark \ref{rem:42}.
		\item The second method unwraps for each spatial frequency the circle $S^{1}$ to the real line $\R$, to solve the Green's function with absorbing boundary conditions at infinity which results in a quickly converging series in rapidly decaying terms expressed in \emph{non-periodic} Mathieu functions. There is a nice probabilistic interpretation: The $k$-th number in the series reflects the contribution of sample-paths in a Monte-Carlo simulation, carrying homotopy number $k \in \mathbb{Z}$, see Figure~\ref{fig:K0K1K2}.
		\item The third method applies the Floquet theorem on the resulting series of the second method and application of the geometric series produces a formula involving only 4 Mathieu functions \cite{DuitsAlmsick2008,MarkusThesis}.
	\end{enumerate}
	We briefly summarize these results in the general case and then we provide the end-results of the three approaches for respectively the contour enhancement case and the contour completion case in the theorems below.
	In Figure~\ref{fig:EnhancementKernel}, we show an illustration of an exact resolvent enhancement kernel and an exact fundamental solution and their marginals.
	\begin{figure}[!htb]
		\centerline{
			\includegraphics[width=0.9\hsize]{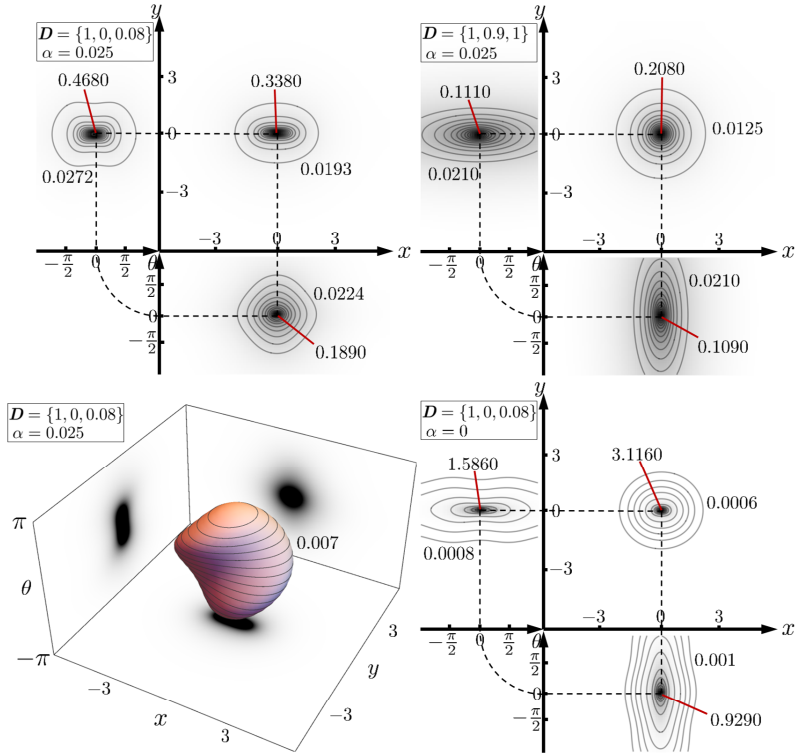}
		}
		\caption{Top row, left: The three marginals of the exact Green's function $R_{\alpha}^{\ul{D}}$ of the resolvent process where $\ul{D} = \textrm{diag}\{D_{11},0,D_{33}\}$ with parameter settings {\small $\alpha=0.025$ and $\ul{D}=\{1,0,0.08\}$}.
			right: The isotropic case of the exact Green's function $R_{\alpha}^{\ul{D}}$ of the resolvent process with {\small $\alpha=0.025$, $\ul{D}=\{1,0.9,1\}$}.
			Bottom row: The fundamental solution $S^{\ul{D}}$ of the resolvent process with {\small $\ul{D}=\{1,0,0.08\}$}. The iso-contour values are indicated in the Figure.
		}\label{fig:EnhancementKernel}
	\end{figure}
	
	Furthermore, we investigate the distribution of the stochastic line propagation process with periodic boundaries at $-\pi-2k\pi$ to $\pi+2k\pi$ of the exact kernel. The probability density distribution of the kernel shows us that most of the random walks only move within $k=2$ loops, i.e. from $-3\pi$ to $3\pi$. See Figure~\ref{fig:K0K1K2}, where it can be seen
	that the series of rapidly decaying terms of method 2 for reasonable parameter settings already be truncated at $N=1$ or $N=2$.
	\begin{figure}[!htb]
		\centerline{
			\includegraphics[width=0.9\hsize]{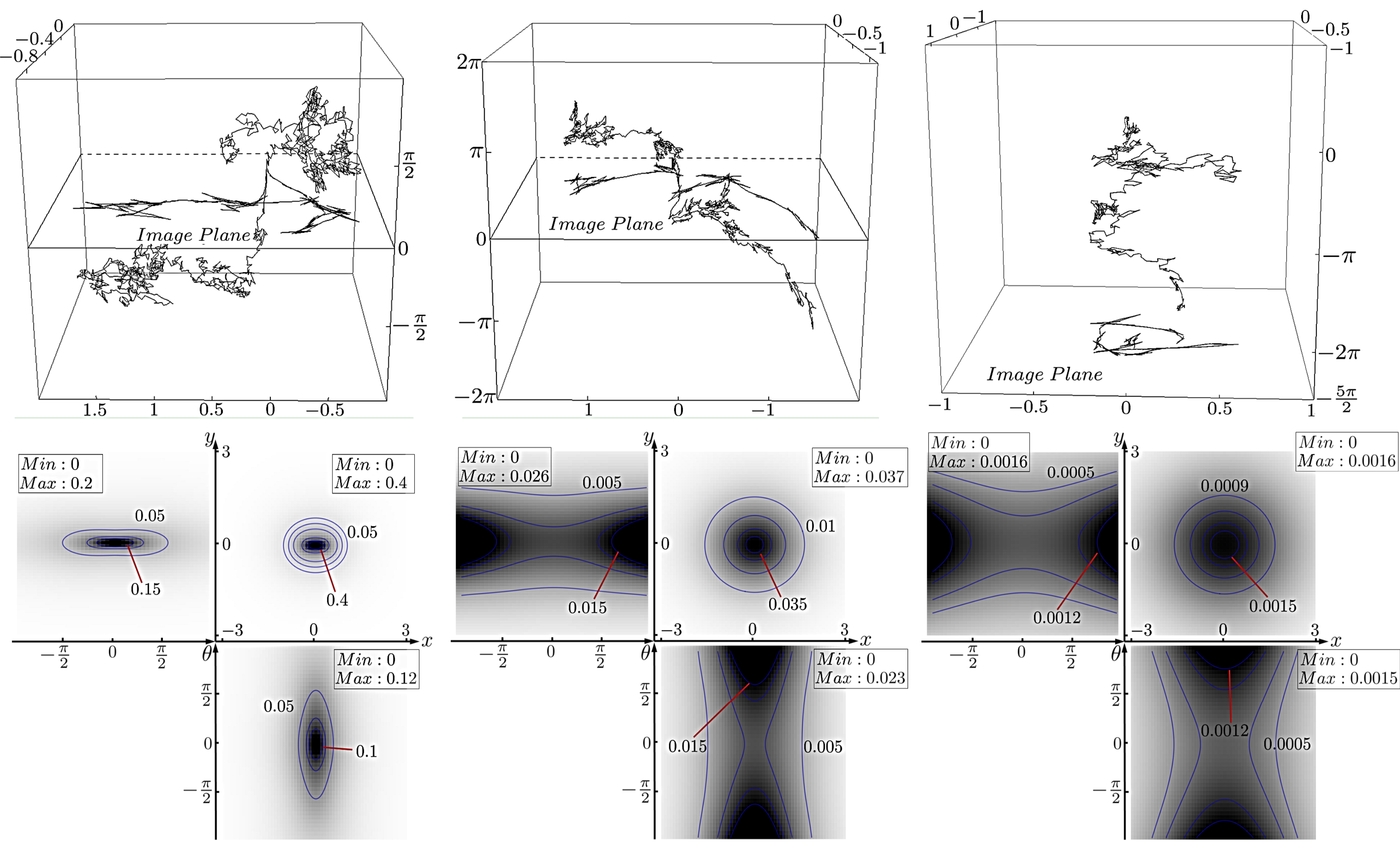}
		}
		\caption{Top row, left to right: Two random walks in $SE(2)=\R^2 \rtimes S^1$ (and their projection on $\R^2$) of the direction processes for $k=0, 1, 2$ cases (where $k$ denotes the amount of loops) of contour enhancement with $\mathbf{D}=\{0.5,0.,0.19\}$ (800 steps, step-size $\Delta t = 0.005$). Bottom row, left to right: the intensity projection of the exact enhancement kernels corresponding to the three cases in the top row, i.e. $\theta$ range from $-\pi$ to $\pi$ for $k=0$ case, from $-3\pi$ to $-\pi$ and $\pi$ to $3\pi$ for $k=1$ case, from  $-5\pi$ to $-3\pi$ and $3\pi$ to $5\pi$ for $k=2$ case, with {\small $\alpha=\frac{1}{40}$, $\mathbf{D}=\{0.5,0.,0.19\}$}.}
		\label{fig:K0K1K2}
	\end{figure}
	
	In Appendix~\ref{app:A} we analyze the asympotical behavior of the spatial Fourier transform of the kernels at the origin and at infinity. It turns out that the fundamental solutions (the case $\alpha \downarrow 0$) are the only kernels with a pole at the origin. This reflects that fundamental solutions are not $\mathbb{L}_{1}$-normalizable, in contrast to resolvent kernels and temporal kernels. Furthermore, the Fourier transform of any kernel restricted to a fixed $\theta$-layer has a rapidly decaying direction $\omega_{\eta}$ and a slowly decaying direction $\omega_{\xi}$. Therefore we analyze the decaying behavior of the spatially Fourier transformed kernels along these axes at infinity and we deduce that all resolvent kernels and fundamental solutions have a singularity at the origin, whereas the time-dependent kernels do not suffer from such a singularity.
	
	\subsection{Spectral Decomposition and the 3 General Forms of Exact Solutions}\label{3GeneralFormsExactSolutions}
	
	In this section, we will derive 3 general forms of the exact solutions. To this end we note that analysis of strongly continuous semigroups \cite{Yosida} and their resolvents start with analysis of the generator $Q^{\ul{D},\ul{a}}(\underline{\mathcall{A}})$. Symmetries of the solutions
	directly follow from the symmetries of the generator. Furthermore, spectral analysis of the generator $Q^{\ul{D},\ul{a}}(\underline{\mathcall{A}})$ as an unbounded operator on $\mathbb{L}_{2}(SE(2))$ provides spectral decomposition and explicit formulas for the time-dependent kernels, their resolvents and fundamental solutions as we will see next.
	
	First of all, the domain of the self-adjoint operator $Q^{\ul{D},\ul{a}}(\underline{\mathcall{A}})$ equals
	\[
	\begin{array}{l}
	\mathcall{D}(Q^{\ul{D},\ul{a}}(\underline{\mathcall{A}}))=\mathbb{H}_{2}(\R^{2}) \otimes \mathbb{H}_{2}(S^{1}), \textrm{ with second order Sobolev space} \\
	\mathbb{H}_{2}(S^{1})\equiv \{\phi \in \mathbb{H}_{2}([0,2\pi])\;|\; \phi(0)=\phi(2\pi) \textrm{ and } {\rm d}\phi(0)={\rm d}\phi(2\pi)\},
	\end{array}
	\]
	where ${\rm d}\phi \in \mathbb{H}_{1}(S^{1})$ is the weak derivative of $\phi$ and where both Sobolev spaces $\mathbb{H}_{2}(S^{1})$ are $\mathbb{H}_{2}(\R^{2})$ are endowed with the $\mathbb{L}_{2}$-norm. Operator $Q^{\ul{D},\ul{a}}(\underline{\mathcall{A}})$ is equivalent to the corresponding operator
	\[
	\mathcall{B}^{\ul{D},\ul{a}}:=(\mathcall{F}_{\R^{2}} \otimes \textrm{id}_{\mathbb{L}_{2}(S^{1})}) \circ Q^{\ul{D},\ul{a}}(\underline{\mathcall{A}}) \circ (\mathcall{F}_{\R^{2}}^{-1} \otimes \textrm{id}_{\mathbb{H}_{2}(S^{1})}),
	\]
	where $\otimes$ denotes the tensor product in distributional sense, $\mathcall{F}_{\R^{2}}$ denotes the unitary Fourier transform operator on $\mathbb{L}_{2}(\R^{2})$ almost everywhere given by
	\[
	\mathcall{F}_{\R^{2}}f(\www)=\hat{f}(\www):= \frac{1}{2\pi} \int_{\R^{2}} f(\ul{x}) e^{-i\, \www \cdot \ul{x}}\, {\rm d}\ul{x},
	\]
	and where $\textrm{id}_{\mathbb{H}_{2}(S^{1})}$ denotes the identity map on $\mathbb{H}_{2}(S^{1})$.
	This operator $\mathcall{B}^{\ul{D},\ul{a}}$ is given by
	\[
	(\mathcall{B}^{\ul{D},\ul{a}}\hat{U})(\www,\theta)= (\mathcall{B}^{\ul{D},\ul{a}}_{\www}\hat{U}(\www,\cdot))(\theta),
	\]
	where for each fixed spatial frequency $\www=(\rho \cos \varphi, \rho \sin \varphi) \in \R^{2}$ operator $\mathcall{B}^{\ul{D},\ul{a}}_{\www}: \mathbb{H}_{2}(S^{1}) \to \mathbb{L}_{2}(S^{1})$ is a mixture of multiplier operators and
	weak derivative operators $d=\partial_{\theta}$:
	\begin{equation}
	\mathcall{B}^{\ul{D},\ul{a}}_{\www}= -\sum \limits_{j=1}^{2} a_{j} m_j + \sum \limits_{k,j=1}^{2} D_{kj} m_k m_j
	+(-a_{3} + 2 D_{j3} m_j) d + D_{33} d^2,
	\end{equation}
	with multipliers $m_{1}=i \rho \cos (\varphi - \theta)$ and $m_{2}= -i \rho \sin(\varphi - \theta)$  corresponding to respectively $\partial_{\xi}=\cos \theta \partial_{x} +\sin \theta \partial_{y}$ and $\partial_{\eta}=-\sin \theta \partial_{x} +\cos \theta \partial_{y}$.
	By straightforward goniometric relations it follows that for each $\www \in \R^{2}$ operator
	$\mathcall{B}^{\ul{D},\ul{a}}_{\www}$ boils down to a 2nd order Mathieu-type operator (i.e. an operator of the type $\frac{d^2}{dz^2}-2q\cos(2z)+a $).
	In case of the contour enhancement we have
	\[\label{ContourEnhancementMathieuOperator}
	\begin{array}{l}
	\left(\ul{a}=\ul{0} \textrm{ and }\ul{D}=\textrm{diag}\{D_{11},D_{22},D_{33}\}\textrm{ and } D_{11},D_{22} \geq 0, D_{33}> 0 \right)  \Rightarrow \\
	\mathcall{B}^{\ul{D},\ul{a}}_{\www}= - D_{11} \rho^2 \cos^{2}(\varphi - \theta)  - D_{22}  \rho^{2}\sin^{2}(\varphi - \theta)+
	D_{33} \partial_{\theta}^2.
	\end{array}
	\]
	In case of the contour completion we have
	\[
	(\ul{a}=(1,0,0) \textrm{ and }D_{33}>0 ) \Rightarrow
	\mathcall{B}^{\ul{D},\ul{a}}_{\www}= - i\rho \cos(\varphi - \theta)  + D_{33} \partial_{\theta}^2.
	\]
	Operator $\mathcall{B}^{\ul{D},\ul{a}}_{\www}$ satisfies
	\[
	(\mathcall{B}^{\ul{D},\ul{a}}_{\www})^* \Theta = \overline{\mathcall{B}^{\ul{D},\ul{a}}_{\www}\overline{\Theta}},
	\]
	and moreover it admits a right-inverse kernel operator
	$K:\mathbb{L}_{2}(S^{1}) \to \mathbb{H}_{2}(S^{1})$ given by
	\begin{equation}\label{relconj}
	Kf(\theta) = \int \limits_{S^{1}} k(\theta,\nu) f(\nu) {\rm d}\nu,
	\end{equation}
	with a kernel satisfying $k(\theta,\nu)=k(\nu,\theta)$ (without conjugation). This kernel $k$
	relates to the fundamental solution of operator $\mathcall{B}^{\ul{D},\ul{a}}_{\www}$:
	\[
	\mathcall{B}^{\ul{D},\ul{a}}_{\www} \hat{S}^{\ul{D},\ul{a}}(\www,\cdot) =\delta^{\theta}_{0},
	\textrm{    for all }\www=(\rho \cos\varphi, \rho \sin \varphi) \in \R^{2},
	\]
	with $\hat{S}^{\ul{D},\ul{a}} :SE(2)\setminus \{e\} \to \R$, infinitely differentiable. By left-invariance of our generator $Q^{\ul{D},\ul{a}}(\underline{\mathcall{A}})$, we
	have
	\[
	k(\theta,\nu)= \hat{S}^{\ul{D},\ul{a}}(\rho \cos(\varphi-\theta),  \rho \sin (\varphi-\theta),\nu-\theta),
	\]
	where $\hat{S}^{\ul{D},\ul{a}}(\www,\theta)$ denotes the spatial Fourier transform of the fundamental solution $S^{\ul{D},\ul{a}}:SE(2) \setminus \{e\} \to \R^{+}$. Now that we have analyzed the generator of our PDE evolutions, we summarize 3 exact approaches describing the kernels of the PDE's of interest.
	
	\subsubsection*{Exact Approach 1}
	Kernel operator $K$ given by Eq.~(\ref{relconj}) is compact and its kernel satisfies $k(\theta,\nu) = k(\nu,\theta)$ and thereby it has a complete bi-orthonormal basis of eigenfunctions $\{\Theta_{n}\}_{n \in \mathbb{Z}}$:
	\[
	\mathcall{B}^{\ul{D},\ul{a}}_{\www} \Theta_{n}^{\www}= \lambda_{n} \Theta_{n}^{\www} \textrm{ and } K \Theta_{n}^{\www} =\lambda_{n}^{-1} \Theta_{n}^{\www}, \textrm{ with }0\geq \lambda_{n} \to \infty,
	\]
	As operator $\mathcall{B}^{\ul{D},\ul{a}}_{\www}$ is a Mathieu type of operator these eigenfunctions $\Theta_{n}$ can be expressed in periodic Mathieu functions, and the corresponding
	eigenvalues can be expressed in Mathieu characteristics as we will explicitly see in the subsequent subsections for both the contour-enhancement and contour-completion cases.
	The resulting solutions of our first approach are
	\begin{equation} \label{sols1}
	\boxed{
		\begin{array}{l}
		W(x,y,\theta,s)= [\mathcall{F}^{-1}_{\R^{2}}\hat{W}(\cdot,\theta,s)](x,y) \textrm{ with }
		\hat{W}(\www,\theta,s)= \sum \limits_{n \in \mathbb{Z}} e^{s \lambda_{n}} (\hat{U}(\www,\cdot),\overline{\Theta_{n}^{\www}})\Theta_{n}^{\www}(\theta), \\
		P_{\alpha}(x,y,\theta)= [\mathcall{F}^{-1}_{\R^{2}}\hat{P}_{\alpha}(\cdot,\theta)](x,y) \textrm{ with }
		\hat{P}_{\alpha}(\www,\theta)= \alpha \sum \limits_{n \in \mathbb{Z}} \frac{1}{\alpha -\lambda_n} (\hat{U}(\www,\cdot),\overline{\Theta_{n}^{\www}}) \Theta_{n}^{\www}(\theta), \\
		\hat{R}_{\alpha}^{\ul{D},\ul{a}}(\ul{\www},\theta)= \frac{\alpha}{2\pi} \sum \limits_{n \in \mathbb{Z}} \frac{1}{\alpha-\lambda_n} \overline{\Theta_{n}^{\www}(\theta)}\, \Theta_{n}^{\www}(0),\\
		S^{\ul{D},\ul{a}}(x,y,\theta)= [\mathcall{F}^{-1}_{\R^{2}}\hat{S}^{\ul{D},\ul{a}}(\cdot,\theta)](x,y) \textrm{ with }
		\hat{S}^{\ul{D},\ul{a}}(\www,\theta)=-\frac{1}{2\pi} \sum \limits_{n \in \mathbb{N}} \frac{1}{\lambda_n} \overline{\Theta_{n}^{\www}(\theta)}\, \Theta_{n}^{\www}(0).
		\end{array}
	}
	\end{equation}
	\begin{remark}
		If $\ul{a}=\ul{0}$ then $(\mathcall{B}^{\ul{D},\ul{a}}_{\www})^*=(\mathcall{B}^{\ul{D},\ul{a}}_{\www})$ and $\overline{\Theta_{n}^{\www}}=\Theta_{n}^{\www}$ and the $\{\Theta_{n}^{\www}\}$ form an orthonormal basis for $\mathbb{L}_{2}(S^{1})$ for each fixed $\www \in \R^{2}$.
	\end{remark}
	\subsubsection*{Exact Approach 2}
	The problem with the solutions (\ref{sols1}) is that the Fourier series representations (\ref{sols1}) do not converge quickly for $s>0$ small.
	Therefore, in the second approach we unfold the circle and for the moment we replace the $2\pi$-periodic boundary condition in $\theta$ by absorbing boundary conditions at infinity
	and we consider the auxiliary problem of finding $\hat{R}^{\ul{D},\ul{a},\infty}_{\alpha}: \R^{2} \times \R \setminus \{e\} \to \R^{+}$, such that
	\begin{equation} \label{unfoldeqs}
	\begin{array}{l}
	(-Q^{\ul{D},\ul{a}}+\alpha I) R^{\ul{D},\ul{a},\infty}_{\alpha} =\alpha \delta_{0}^{x} \otimes \delta_{0}^{y} \otimes \delta_{0}^{\theta}, \\
	R^{\ul{D},\ul{a},\infty}_{\alpha}(\cdot, \theta) \to 0 \textrm{ as }|\theta| \to \infty.
	\end{array}
	\desda \forall_{\www \in \R^{2}}\;:\:
	\left\{
	\begin{array}{l}
	(-\mathcall{B}^{\ul{D},\ul{a}}_{\www}+\alpha I) \hat{R}^{\ul{D},\ul{a},\infty}_{\alpha}(\www, \cdot) =\alpha \, \frac{1}{2\pi}\, \delta_{0}^{\theta}, \\
	\hat{R}^{\ul{D},\ul{a},\infty}_{\alpha}(\www, \theta) \to 0 \textrm{ as }|\theta| \to \infty.
	\end{array}
	\right.
	\end{equation}
	The spatial Fourier transform of the corresponding fundamental solution again follows by taking the limit $\alpha \downarrow 0$: $\hat{S}^{\infty}:=\lim \limits_{\alpha \downarrow 0}\alpha^{-1}\hat{R}^{\ul{D},\ul{a},\infty}_{\alpha}$. Now the solution of (\ref{unfoldeqs}) is given by
	\begin{equation}\label{genform}
	\boxed{
		\hat{R}^{\ul{D},\ul{a},\infty}_{\alpha}(\www, \theta)=
		\frac{\alpha}{ 2\pi D_{33}\, W_{\rho}}
		\left\{
		\begin{array}{l}
		G_{\rho}(\varphi)F_{\rho}(\varphi-\theta), \textrm{ for }\theta \geq 0, \\
		F_{\rho}(\varphi)G_{\rho}(\varphi-\theta), \textrm{ for }\theta \leq 0,
		\end{array}
		\right.
		\textrm{ for all }
		\www=(\rho \cos \varphi, \rho \sin\varphi)}
	\end{equation}
	where $\theta \mapsto F_{\rho}(\varphi-\theta)$ is the unique solution in the nullspace of operator $-\mathcall{B}^{\ul{D},\ul{a}}_{\www}+\alpha I$ satisfying $F_{\rho}(\theta) \to 0$ for $\theta \to +\infty$,
	and where $G_{\rho}$ is the unique solution in the nullspace of operator $-\mathcall{B}^{\ul{D},\ul{a}}_{\www}+\alpha I$ satisfying $G_{\rho}(\theta) \to 0$ for $\theta \to -\infty$, and
	The Wronskian of $F_{\rho}$ and $G_{\rho}$ is given by
	\begin{equation}\label{WronskianComputation}
	W_{\rho}=F_{\rho}G_{\rho}'-G_{\rho}F_{\rho}'=
	F_{\rho}(0)G_{\rho}'(0)-G_{\rho}
	(0)F_{\rho}'(0).
	\end{equation}
	See Figure~\ref{fig:ContinuousFit}.
	We conclude with the periodized solutions
	\begin{equation} \label{periodized}
	\boxed{
		\begin{array}{l}
		R_{\alpha}^{\ul{D},\ul{a}}(x,y,\theta)= [\mathcall{F}^{-1}_{\R^{2}}\hat{R}_{\alpha}^{\ul{D},\ul{a}}(\cdot,\theta)](x,y) \textrm{ with }
		\hat{R}_{\alpha}^{\ul{D},\ul{a}}(\www,\theta)
		= \sum \limits_{n \in \mathbb{Z}} \hat{R}_{\alpha}^{\ul{D},\ul{a},\infty}(\www,\theta+2n \pi) , \\
		S^{\ul{D},\ul{a}}(x,y,\theta)= [\mathcall{F}^{-1}_{\R^{2}}\hat{S}^{\ul{D},\ul{a}}(\cdot,\theta)](x,y) \textrm{ with }
		\hat{S}^{\ul{D},\ul{a}}(\www,\theta)=\sum \limits_{n \in \mathbb{Z}}
		\hat{S}^{\ul{D},\ul{a},\infty}(\www,\theta+2n \pi).
		\end{array}
	}
	\end{equation}
	\begin{figure}[!htb]
		\centering
		\includegraphics[width=0.9\hsize]{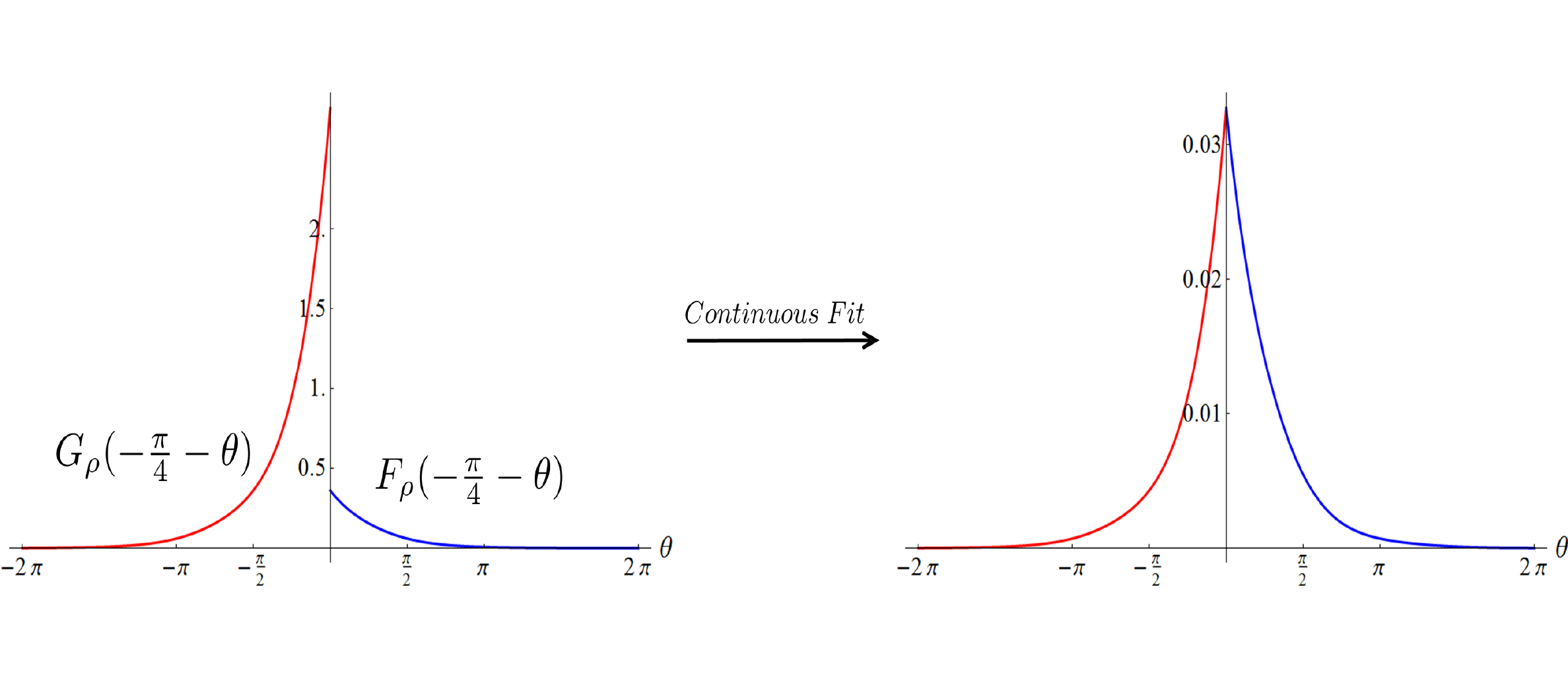}
		\caption{Illustration of the continuous fit of $\theta \mapsto \hat{R}_{\alpha}^{\ul{D},\ul{0},\infty}(\www,\theta)$ in Eq.~(\ref{genform}) for contour enhancement with parameter settings
			$D_{11}=1, D_{22}=0, D_{33}=0.05$ and $\alpha=\frac{1}{20}$, at $(\omega_x, \omega_y)=(\frac{\pi}{20},\frac{\pi}{20})$.}
		\label{fig:ContinuousFit}
	\end{figure}
	For further details see \cite{DuitsAlmsick2008,DuitsCASA2005,DuitsCASA2007,DuitsAMS1,MarkusThesis}. Here we omit the details on these explicit solutions for the general case as the proof is fully equivalent to \cite[Lemma 4.4\&Thm 4.5]{DuitsAlmsick2008}, and moreover the techniques are directly generalizable from standard Sturm-Liouville theory.
	
	\subsubsection*{Exact Approach 3}
	In the third approach, where for simplicity we restrict ourselves to both cases of the contour enhancement and the contour completion, we apply the well-known Floquet theorem to the second order ODE
	\begin{equation}\label{Mathieu}
	(-\mathcall{B}^{\ul{D},\ul{a}}_{\www}+\alpha I)F(\theta)=0 \desda
	F''(\theta) -2 q_{\rho} \cos((\varphi - \theta)\mu) F(\theta) = -a_{\rho}\, F(\theta),
	\end{equation}
	with $\mu \in \{1,2\}$. For the precise settings/formulas of $a_{\rho}$, $q_{\rho}$ and $\mu$, in the case of contour enhancement and contour completion we refer to the next subsections.
	Note that in both the case of contour enhancement and completion we have the Mathieu functions (following the conventions of \cite{AbraMathieu,Schaefke} ) with
	\begin{equation} \label{MatheiuEllipticFunctions}
	\boxed{
		\begin{array}{l}
		\textrm{me}_\nu(z;q_\rho)=\textrm{ce}_\nu(z;q_\rho)+i\textrm{se}_\nu(z;q_\rho)\\
		\textrm{me}_{-\nu}(z;q_\rho)=\textrm{ce}_\nu(z;q_\rho)-i\textrm{se}_\nu(z;q_\rho)
		\end{array},
	}
	\end{equation}
	where $z= \varphi-\theta, \nu=\nu(a_{\rho},q_{\rho})$, $\textrm{ce}_\nu(z;q_\rho)$ denotes the cosine-elliptic functions and $\textrm{se}_\nu(z;q_\rho)$ denotes the sine-elliptic functions, given by
	\begin{equation*} \label{CosineSineEllipticFunctions}
	\begin{array}{l}
	\textrm{ce}_\nu(z;q_\rho)=\sum \limits_{r=-\infty}^{\infty} \textrm{c}_{2r}^\nu(q_\rho)\cos{(\nu+2r)}z\; \textrm{with}\; \textrm{ce}_\nu(z;0)=\cos{\nu z}\\
	\textrm{se}_{\nu}(z;q_\rho)=\sum \limits_{r=-\infty}^{\infty} \textrm{c}_{2r}^\nu(q_\rho)\sin{(\nu+2r)}z\; \textrm{with}\; \textrm{se}_\nu(z;0)=\sin{\nu z}
	\end{array},
	\end{equation*}
	For details see \cite{Schaefke}.
	Then, we have \[
	F_{\rho}(z)=\textrm{me}_{-\nu}(z/\mu ,q_{\rho}),\;
	G_{\rho}(z)=\textrm{me}_{\nu}(z/\mu ,q_{\rho}),
	\]
	with $\mu=1$ in the contour enhancement case and $\mu=2$ in the contour completion case. Furthermore $a_{\rho}$ denotes the Mathieu characteristic and $q_{\rho}$
	denotes the Mathieu coefficient and $\nu=\nu(a_{\rho},q_{\rho})$ denotes the purely imaginary Floquet exponent (with $i \nu <0$)
	with respect to the Mathieu ODE-equation (\ref{Mathieu}), whose general form is
	\[
	y''(z)- 2q \cos(2z) y(z)= -a y(z).
	\]
	Application of this theorem to the solutions $F_{\rho}$ and $G_{\rho}$ in Eq.~(\ref{periodized}) yields
	\begin{equation} \label{geom}
	F_{\rho}\left(z -2n \pi\right)=e^{\frac{2n \pi \,i\, \nu}{\mu}} F_{\rho}\left(z\right) \textrm{ and }G_{\rho}\left(z -2n \pi\right)=e^{-\frac{2n \pi\, i \,\nu}{\mu}} F_{\rho}\left(z\right), \qquad z=\varphi-\theta.
	\end{equation}
	Substitution of (\ref{geom}) into (\ref{periodized}) together with the geometric series
	\[
	\sum \limits_{n=0}^{\infty} \left(e^{2  \nu \pi i/\mu} \right)^{n}=\frac{1}{1-e^{2 i\nu \pi/\mu}} \textrm{ and } \frac{1+e^{2i\nu \pi/\mu}}{1-
		e^{2i \nu \pi/\mu}}= -\textrm{coth}\,(i \nu \pi/\mu)= i\cot(\nu \pi/\mu),
	\]
	with Floquet exponent $\nu=\nu(a_{\rho},q_{\rho}), \  \textrm{Im}(\nu)>0$,
	yields
	the following closed form solution expressed in 4 Mathieu functions:
	{\small
		\begin{equation} \label{sols3}
		\boxed{
			\begin{array}{l}
			[\hat{R}_{\alpha}^{\ul{D},\ul{a}}(\cdot,\theta)](\www)= \frac{\alpha}{D_{33}\, i \, W_{\rho}}
			\left\{ \right. \\
			\left.\left(-\cot(\frac{\nu \pi}{\mu})\left(\textrm{ce}_{\nu}(\frac{\varphi}{\mu},q_{\rho})\, \textrm{ce}_{\nu}(\frac{\varphi-\theta}{\mu},  q_{\rho})+
			\textrm{se}_{\nu}(\frac{\varphi}{\mu},q_{\rho})\, \textrm{se}_{\nu}(\frac{\varphi-\theta}{\mu},q_{\rho})\right)+\right. \right.\\
			\left. \left. \qquad
			\textrm{ce}_{\nu}(\frac{\varphi}{\mu},q_{\rho})\, \textrm{se}_{\nu}(\frac{\varphi-\theta}{\mu},q_{\rho})-
			\textrm{se}_{\nu}(\frac{\varphi}{\mu},q_{\rho})\, \textrm{ce}_{\nu}(\frac{\varphi-\theta}{\mu},q_{\rho})\right){\rm u}(\theta)\;+ \qquad \right. \\
			\left.\left(-\cot(\frac{\nu \pi}{\mu})\left(\textrm{ce}_{\nu}(\frac{\varphi}{\mu},q_{\rho})\, \textrm{ce}_{\nu}(\frac{\varphi-\theta}{\mu},q_{\rho})-
			\textrm{se}_{\nu}(\frac{\varphi}{\mu},q_{\rho})\, \textrm{se}_{\nu}(\frac{\varphi-\theta}{\mu},q_{\rho}\right) + \right. \right.\\
			\left. \qquad
			\textrm{ce}_{\nu}(\frac{\varphi}{\mu},q_{\rho})\, \textrm{se}_{\nu}(\frac{\varphi-\theta}{\mu},q_{\rho})+
			\textrm{se}_{\nu}(\frac{\varphi}{\mu},q_{\rho})\, \textrm{ce}_{\nu}(\frac{\varphi-\theta}{\mu},q_{\rho}\right){\rm u}(-\theta)
			\end{array}
		}
		\end{equation}
	}
	with Floquet exponent $\nu=\nu(a_{\rho},q_{\rho})$ and where $\theta \mapsto {\rm u}(\theta)$ denotes the unit step function.
	
	Next we will summarize the main results, before we consider the special cases of the contour enhancement and the contour completion.
	\begin{theorem}\label{th:exact}
		The exact solutions of all linear left-invariant (convection)-diffusions on $SE(2)$, their resolvents, and their fundamental solutions given by
		\[
		W(g,t)=(K_{t}^{\ul{D},\ul{a}} *_{SE(2)} U)(g), \ \ P_{\alpha}(g)= (R_{\alpha}^{\ul{D},\ul{a}} *_{SE(2)} U)(g), \ \
		S^{\ul{D},\ul{a}}= (Q^{\ul{D},\ul{a}}(\underline{\mathcall{A}}))^{-1} \delta_{e},
		\]
		admit three types of exact representations for the solutions. The first type is a series expressed involving periodic Mathieu functions given by Eq.~(\ref{sols1}).
		The second type is a rapidly decaying series involving non-periodic Mathieu functions given by Eq.~(\ref{genform}) together with Eq.~(\ref{periodized}),
		and the third one involves only four non-periodic Mathieu functions and is given by Eq.~(\ref{sols3}).
	\end{theorem}
	
	\subsubsection{The Contour Enhancement Case}
	
	In case $\ul{D}=\textrm{diag}\{D_{11},D_{22}, D_{33}\}$ with $D_{11},D_{33}>0$ and $D_{22}\geq 0$ and $\ul{a}=\ul{0}$,
	the settings in the solution formula of the first approach Eq.\!~(\ref{sols1}) are
	\begin{equation} \label{efenh}
	\begin{array}{l}
	\Theta_{n}(\theta)= \frac{\textrm{me}_{n}(\varphi-\theta,q_{\rho})}{\sqrt{2\pi}},\;
	q_{\rho}=\frac{\rho^2 (D_{11}-D_{22})}{4 D_{33}},\;
	\lambda_{n}=-a_{n}(q_{\rho}) D_{33} - \frac{\rho^{2}(D_{11}+D_{22})}{2},
	\end{array}
	\end{equation}
	where $\textrm{me}_{n}(z,q)$ denotes the periodic Mathieu function with parameter $q$ and $a_{n}(q)$ the corresponding Mathieu characteristic,
	and with Floquet exponent $\nu = n \in \mathbb{Z}$.
	
	The settings of the solution formula of the second approach
	Eq.\!~(\ref{Mathieu}) together with Eq.\!~(\ref{periodized}) are
	\begin{equation} \label{aqenh}
	\begin{array}{l}
	a_{\rho}=\frac{-\alpha -\frac{\rho^2}{2}(D_{11}+D_{22})}{D_{33}}, \
	q_{\rho}=\frac{\rho^2 (D_{11}-D_{22})}{4 D_{33}}, \
	\mu = 1, \ W_{\rho}=-2i \, \textrm{se}_{\nu}'(0,q_{\rho})\textrm{ce}_{\nu}(0,q_{\rho}),
	\end{array}
	\end{equation}
	where $\textrm{se}_\nu'{(0,q_\rho)}=\frac{d}{dz}\textrm{se}_\nu(z,q)|_{z=0}$.
	The third approach Eq.\!~(\ref{sols3}) yields for $D_{11}>D_{22}$ the result
	in \cite[Thm 5.3]{DuitsAMS1}.
	\begin{remark}
		As the generator $Q^{\ul{D},\ul{0}}(\underline{\mathcall{A}})=D_{11}\mathcall{A}_{1}^{2} + D_{33}\mathcall{A}_{3}^{2}$ is invariant under the reflection
		$\mathcall{A}_{3} \mapsto -\mathcall{A}_{3}$ we have that our real-valued kernels satisfy $K(x,y,\theta)=K(-x,-y,\theta)$. As a result the spatially Fourier transformed enhancement kernels given by
		$\hat{K}_{t}^{\ul{D}}(\www,\theta)$, $\hat{R}_{\alpha}^{\ul{D}}(\www,\theta)$, $\hat{S}^{\ul{D}}(\www,\theta)$ are real-valued.
		This is indeed the case in e.g. Eq.\!~(\ref{genform}), Eq.\!~(\ref{sols3}), as for $q,z \in \R$ and $\overline{\nu}=-\nu$, we have
		$\overline{\textrm{me}_{\nu}(z,q)}=
		\textrm{me}_{\overline{\nu}}(-\overline{z},\overline{q})=\textrm{me}_{\nu}(z,q)$,
		so that
		$\overline{\textrm{se}_{\nu}(z,q)} \in i\mathbb{R}$ and $\overline{\textrm{ce}_{\nu}(z,q)} \in \mathbb{R}$.
	\end{remark}
	
	\subsubsection{The Contour Completion Case}
	In case $\ul{D}=\textrm{diag}\{0,0, D_{33}\}$ with $D_{33}>0$ and  $\ul{a}=(1,0,0)$,
	the settings in the solution formula of the first approach Eq.\!~(\ref{sols1}) are
	\begin{equation} \label{efcom}
	\begin{array}{l}
	\Theta_{n}(\theta)= \frac{\textrm{ce}_{n}\left(\frac{\varphi-\theta}{2},q_{\rho}\right)}{\sqrt{\pi}} , n \in \mathbb{N}\cup \{0\}, \qquad
	\lambda_{n}=-\frac{a_{n}(q_{\rho})D_{11}}{4}, \ \ q_{\rho}=\frac{2\rho i}{D_{33}},
	\end{array}
	\end{equation}
	where $\textrm{ce}_{n}$ denotes the even periodic Mathieu-function
	with Floquet exponent $n$.
	
	The settings of the solution formula of the second approach
	Eq.\!~(\ref{Mathieu}) together with Eq.\!~(\ref{periodized}) are
	\begin{equation} \label{aqcom}
	\begin{array}{l}
	a_{\rho}= -\frac{4\alpha}{D_{33}}, \
	q_{\rho}= \frac{2\rho i}{D_{33}}, \
	\mu = 2, \ W_{\rho}= - i \, \textrm{se}_{\nu}'(0,q_{\rho})\textrm{ce}_{\nu}(0,q_{\rho}).
	\end{array}
	\end{equation}
	See Figure~\ref{fig:ExactCompletionKernel} for plots of completion kernels.
	\begin{figure}[!htb]
		\centerline{
			\includegraphics[width=0.9\hsize]{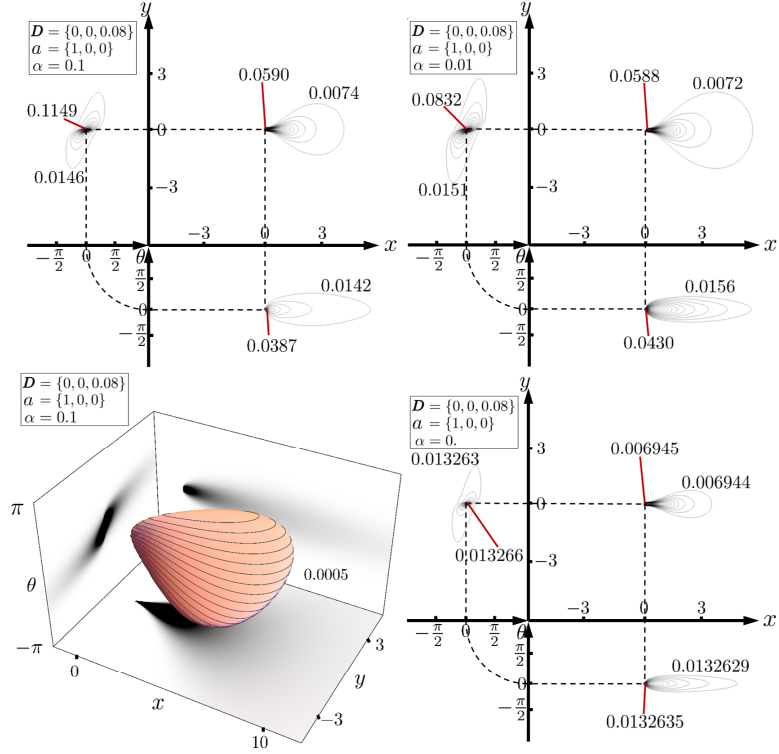}
		}
		\caption{The marginals of the exact Green's functions for contour completion. All the figures have the same settings: $\sigma=0.4$, $\ul{D}=\{0,0,0.08\}$ and $\ul{a}=(1,0,0)$. Top row, left: The resolvent process with {\small $\alpha=0.1$},
			right: The resolvent process with {\small $\alpha=0.01$}.
			Bottom row: The fundamental solution of the resolvent process with {\small $\alpha=0.$} The iso-contour values are indicated in the Figure.
		}\label{fig:ExactCompletionKernel}
	\end{figure}

	\subsubsection{Overview of the Relation of Exact Solutions to Numerical Implementation Schemes}
	
	Theorem~\ref{th:exact} provides three type of exact solutions for our PDE's of interest, and the question rises how these exact solutions
	relate to the common numerical approaches to these PDE's.
	
	The solutions of the first type relate to $SE(2)$-Fourier and finite element type (but then using a in Fourier basis) of techniques, as we will show for the general case in Section~\ref{section:Duitsmatrixalgorithm}. The general idea is that if the dimension of the square band matrices (where the bandsize is atmost $5$) tends to infinity, the exact solutions arise in the spectral decomposition of the numerical matrices.
	
	To compare the solutions of the second/third type of exact solutions to the numerics we must sample the solutions involving non-periodic Mathieu functions
	in the Fourier domain. Unfortunately, as also reported by Boscain et al. \cite{Boscain2}, well-tested and complete publically available packages for Mathieu-function evaluations
	are not easy to find. The routines for Mathieu function evaluation in \emph{Mathematica 7,8,9}, at least show proper results for specific parameter settings. However,
	in case of contour enhancement their evaluations numerically break down for the interesting cases $D_{11}=1$ and $D_{33}<0.2$, see Figure~\ref{fig:MathieuImplementationComparison} in Appendix \ref{app:B}. Therefore, in Appendix~\ref{app:B}, we provide
	our own algorithm for Mathieu-function evaluation relying on standard theory of continued fractions \cite{ContinuedFractions}.
	This allows us to sample the exact solutions in the Fourier domain for comparisons. Still there are two issues left that we address in the next section:
	1. One needs to analyze errors that arise by replacing $\textbf{CFT}^{-1}$ (Inverse of the Continuous Fourier
	Transform) by the $\textbf{DFT}^{-1}$ (Inverse of the Discrete Fourier Transform), 2. One needs to deal with singularities at the origin.
	\subsubsection{The Direct Relation of Fourier Based Techniques to the Exact Solutions}\label{section:FourierBasedForEnhancement}
	In \cite{DuitsAlmsick2008} we have related matrix-inversion in Eq.~(\ref{MatrixInverse}) to the exact solutions for the contour completion case. Next we follow a similar approach for the contour enhancement case with ($D_{22}=0$, i.e. hypo-elliptic diffusion), where again we relate diagonalization of the five-band matrix to the exact solutions.
	\begin{theorem}\label{th:RelationofFourierBasedWithExactSolution}
		Let $\pmb{\omega}=(\rho\cos\varphi, \rho\sin\varphi) \in \R^2$ be fixed. In case of contour enhancement with $\ul{D}=\textrm{diag}\{D_{11},0,D_{33}\}$ and $\ul{a}= \mathbf{0}$, the solution of the matrix system (\ref{5recursion}), for $N \rightarrow \infty$, can be written as\\
		{\small
			\begin{equation}
			\boxed{
				\hat{P}=S \Lambda^{-1} S^T \hat{\mathbf{u}}}
			\end{equation}
		}
		with
		\begin{equation}\label{recursionParameters}
		\begin{array}{l}
		\hat{P}=\{\tilde{P}^\ell(\rho)\}_{\ell \in \mathbb{Z}},\quad \hat{\mathbf{u}}=\{\tilde{u}^\ell(\rho)\}_{\ell \in \mathbb{Z}}, \quad S=[S_n^\ell]=[c_\ell^n(q_\rho)],\\
		\Lambda=\textrm{diag}\{\alpha-\lambda_n({\rho})\},\quad \lambda_n(\rho)=-a_{2n}(q_\rho)D_{33}-\frac{\rho^2 D_{11}}{2}, \quad q_\rho=\frac{\rho^2D_{11}}{4 D_{33}},\\
		\end{array}
		\end{equation}
		and where
		\begin{align*}
		c_\ell^n=
		\left\{
		\begin{array}{l}
		\textit{Mathieu Coefficient }\; c_\ell^n, \quad if \; \ell \; is \; even\\
		0, \quad if \;\ell\; is\; odd.
		\end{array}
		\right.
		\end{align*}
		
		In fact Eq.~(\ref{5recursion}), for $N \rightarrow \infty,$ boils down to a steerable $SE(2)$ convolution\cite{FrankenPhDThesis} with the corresponding exact kernel $R_\alpha^{\ul{D},\ul{a}}:SE(2)\rightarrow \R^+$
	\end{theorem}
	\begin{proof}
		Both $\{\frac{1}{\sqrt{2\pi}}e^{i\ell(\varphi-\theta)}| \ell \in \mathbb{Z}\}$
		and $\{\frac{1}{\sqrt{2\pi}}\Theta_n^{\pmb{\omega}}(\theta):=\frac{me_n(\varphi-\theta,q_\rho)}{\sqrt{2\pi}}|n \in \mathbb{Z}\}$
		form an orthonormal basis of $\mathbb{L}_2({S^1})$. The corresponding basis transformation is given by $S$. As this basis transformation is unitary, we have $S^{-1}=S^\dagger=\bar{S}^T$. As a result we have
		\begin{equation}
		\begin{aligned}
		\tilde{P}^\ell(\rho) = \sum_{m,n,p \in \mathbb{Z}} S_m^\ell \delta_n^m \frac{1}{\alpha-\lambda_n(\rho)} (S^\dagger)_p^n\tilde{u}^p(\rho) =\sum_{n \in \mathbb{Z}}\sum_{p \in \mathbb{Z}} \frac{c_\ell^n(q_\rho)c_p^n(q_\rho)\tilde{u}^p(\rho)}{\alpha-\lambda_n(\rho)}.
		\end{aligned}
		\end{equation}
		Thereby, as $me_n(z)=\sum_{\ell \in \mathbb{Z}}c_\ell^n(q_\rho)e^{i \ell z}$, we have:
		\begin{equation}
		\begin{aligned}
		\hat{P}_\alpha(\pmb{\omega},\theta)
		=\alpha \sum_{\ell \in \mathbb{Z}} \tilde{P}^\ell(\rho)e^{i \ell (\varphi - \theta)}
		=\alpha \sum_{n \in \mathbb{Z}}\sum_{p \in \mathbb{Z}}\frac{me_n(\varphi - \theta, q_\rho)c_p^n(q_\rho)e^{ip\varphi} \hat{u}^p(\rho)}{\alpha-\lambda_n(\rho)},
		\end{aligned}
		\end{equation}
		where we recall $\tilde{u}^p=e^{ip\varphi}\hat{u}^p$.
		Now by setting $u=\delta_e \Leftrightarrow \hat{u}(\pmb{\omega},\theta)=\frac{1}{2\pi}\delta_0^\theta \Leftrightarrow \forall_{p \in \mathbb{Z}},\; \hat{u}^p=\frac{1}{2\pi}$.\\
		We obtain the exact kernel \\
		\begin{equation}
		R_\alpha^{\ul{D},\ul{a}}(\pmb{\omega},\theta)=\frac{\alpha}{2\pi}\sum_{n \in \mathbb{Z}}\frac{\Theta_n^{\pmb{\omega}}(\theta)\Theta_n^{\pmb{\omega}}(0)}{\alpha-\lambda_n(\rho)}.
		\end{equation}
		From which the result follows. $\hfill \Box$
	\end{proof}
	\textbf{Conclusion:} This theorem supports our numerical findings that will follow in Section \ref{section:Experimental results}. The small relative error are due to rapid convergence $\frac{1}{(\alpha-\lambda_n(\rho))} \rightarrow 0\quad (n\rightarrow \infty)$, so that truncation of the 5-band matrix produces very small $\textit{uniform}$ errors compared to the exact solutions. It is therefore not surprising that the Fourier based techniques outperform the finite difference solutions in terms of numerical approximation (see experiments Section~\ref{section:Experimental results}).
	
	\subsection{Comparison to The Exact Solutions in the Fourier Domain\label{ch:comparison}}
	In the previous section we have derived the Green's function of the exact solutions of the system
	\begin{align} \label{ResolventEquations2}
	\left\{
	\begin{aligned}
	&(\alpha I-Q^{\ul{D},\ul{a}}) R_{\alpha}^{\ul{D},\ul{a}}=\alpha \delta_e\\
	&R_{\alpha}^{\mathbf{D},\mathbf{a}}(x,y,-\pi)
	= 
	R_{\alpha}^{\mathbf{D},\mathbf{a}}(x,y,\pi)
	\end{aligned}
	\right. \end{align}
	in the continuous Fourier domain. However, we still need to produce nearly exact solutions $R_{\alpha}^{\ul{D},\ul{a}}(x,y,\theta_r)$ in the spatial domain, given by
	\begin{equation}\label{ContinuousExactSolutions}
	\begin{aligned}
	R_{\alpha}^{\ul{D},\ul{a}}(x,y,\theta_r)&=\left(\frac{1}{2\pi}\right)^2\int_{-\infty}^{\infty}\int_{-\infty}^{\infty}
	\hat{R}_{\alpha}^{\ul{D},\ul{a}}(\pmb{\omega},\theta_r)e^{i\pmb{\omega}\pmb{x}}d\pmb{\omega}\\
	&=\left(\frac{1}{2\pi}\right)^2\int_{-\varsigma\pi}^{\varsigma\pi}\int_{-\varsigma\pi}^{\varsigma\pi}\hat{R}_{\alpha}^{\ul{D},\ul{a}}(\pmb{\omega},\theta_r)e^{i\pmb{\omega}\pmb{x}}d\pmb{\omega}+I_\varsigma(\ul{x},r),
	\end{aligned}
	\end{equation}
	where $\pmb{x}=(x,y) \in \R^2$, $\pmb{\omega}=(\omega_x,\omega_y)\in \R^2$, $\theta_r=(\frac{2\pi}{2R+1} \cdot r) \in [-\pi,\pi]$ are the discrete angles and $r \in \{-R,-(R-1),...,0,...,R-1,R\}$, $\varsigma$ is an oversampling factor and $I_\varsigma(\ul{x},r)$ represent the tails of the exact solutions due to their support outside the range $[-\varsigma\pi,\varsigma\pi]$ in the Fourier domain, given by
	\begin{equation}
	I_{\zeta}(\ul{x},r)=\left(\frac{1}{2\pi}\right)^2 \int_{\R^2 \setminus [-\varsigma\pi,\varsigma\pi]^2}e^{-s|\pmb{\omega}|^2}\hat{R}_{\alpha}^{\ul{D}}(\pmb{\omega},\theta_r)e^{i\pmb{\omega}\ul{x}}d\pmb{\omega}.
	\end{equation}
	However in practice we sample the exact solutions in the Fourier domain and then obtain the spatial kernel by directly applying the $\textbf{DFT}^{-1}$. Here errors will emerge by using the $\textbf{DFT}^{-1}$ instead of the $\textbf{CFT}^{-1}$. More precisely, we shall rely on the $\textbf{CDFT}^{-1}$ (Inverse of the Centered Discrete Fourier Transform). Next we analyze and estimate the errors via Riemann sum approximations \cite{RiemannSum}. The nearly exact solutions of the spatial kernel in Eq.~(\ref{ContinuousExactSolutions}) can be written as
	\begin{equation}\label{DiscreteExactSolutions}
	\begin{array}{l}
	R^{\ul{D},\ul{a}}_\alpha(x,y,\theta_r)=\left(\frac{1}{2\pi}\right)^2 \sum\limits_{p'=-\varsigma P}^{\varsigma P}\sum\limits_{q'=-\varsigma Q}^{\varsigma Q}\hat{R}^{\ul{D},\ul{a}}_\alpha (\omega_{p'}^1,\omega_{q'}^2,\theta_r)e^{i(\omega_{p'}^1 x+\omega_{q'}^2 y)}\Delta\omega^1\Delta\omega^2+\\
	\qquad\qquad\qquad \left.
	I_\varsigma(\ul{x},r)+O\left(\frac{1}{2P+1}\right)+O\left(\frac{1}{2Q+1}\right)\right.\\
	\qquad  \qquad \left.
	=\frac{1}{2P+1}\frac{1}{2Q+1} \sum\limits_{p'=-\varsigma P}^{\varsigma P}\sum\limits_{q'=-\varsigma Q}^{\varsigma Q}\hat{R}^{\ul{D},\ul{a}}_\alpha (\omega_{p'}^1,\omega_{q'}^2,\theta_r)e^{i(\omega_{p'}^1 x+\omega_{q'}^2 y)}+\right.\\
	\qquad\qquad\qquad \left.
	I_\varsigma(\ul{x},r)+O\left(\frac{1}{2P+1}\right)+O\left(\frac{1}{2Q+1}\right)\right.
	\end{array}
	\end{equation}
	where
	$\Delta\omega^1=\frac{2\pi}{2P+1}=\frac{2\pi}{x_{dim}},
	\Delta\omega^2=\frac{2\pi}{2Q+1}=\frac{2\pi}{y_{dim}}$ and $P, \, Q \in \mathbb{N}$ determine the number of samples in the spatial domain, with discrete frequencies and angles given by
	\begin{equation}\label{discretefrequencies}
	\omega_{p'}^1=\frac{2\pi}{2P+1} \cdot p' \in [-\varsigma\pi,\varsigma\pi],\;
	\omega_{q'}^2=\frac{2\pi}{2Q+1} \cdot q' \in [-\varsigma\pi,\varsigma\pi],\;
	\theta_r=\frac{2\pi}{2R+1} \cdot r \in [-\pi,\pi]
	\end{equation}
	There are three approximation terms in Eq.~(\ref{DiscreteExactSolutions}), and two of them, i.e. $O\left(\frac{1}{2P+1}\right)$ and $O\left(\frac{1}{2Q+1}\right)$ are standard due to Riemann sum approximation. However, $I_\varsigma(\ul{x},r)$ is harder to control and estimate.
	This is one of the reasons why we include a spatial Gaussian blurring with small scale $0<s \ll 1$. This means that instead of solving
	$R_\alpha^{\ul{D},\ul{a}}=
	\alpha(\alpha I-Q^{\ul{D},\ul{a}}(\mathcall{A}_1,\mathcall{A}_2,\mathcall{A}_3))^{-1}\delta_e$,
	we compute
	\begin{equation}\label{ResolventWithGaussian}
	R_\alpha^{\ul{D},\ul{a},s}=e^{s\Delta}\alpha(\alpha I-Q^{\ul{D},\ul{a}}(\mathcall{A}_1,\mathcall{A}_2,\mathcall{A}_3))^{-1}\delta_e
	=\alpha(\alpha I - Q^{\ul{D},\ul{a}}(\mathcall{A}_1,\mathcall{A}_2,\mathcall{A}_3))^{-1} e^{s\Delta} \delta_e.
	\end{equation}
	So instead of computing the impulse response of a resolvent diffusion we compute the response of a spatially blurred spike $G_s \otimes \delta_0^\theta$ with Gaussian kernel $G_s(x)=\frac{e^{-\frac{||x||^2}{4s}}}{4\pi s}$. Another reason for including a linear isotropic diffusion is that the kernels $R_\alpha^{\ul{D},\ul{a},s}$ are not singular at the origin. The singularity at the origin $(0,0,0)$ of $R_\alpha^{\ul{D},\ul{a}}$ reproduces the original data, whereas the tails of $R_\alpha^{\ul{D},\ul{a}}$ take care of the external actual visual enhancement. Therefore, reducing the singularity at the origin by slight increase of $s>0$, amplifies the enhancement properties of the kernel in practice.
	However, $s>0$ should not be too large as we do not want the isotropic diffusion to dominate the anisotropic diffusion.
	\begin{theorem}
		The exact solutions of $R_\alpha^{\ul{D},\ul{a},s}:SE(2) \rightarrow \R^+$ are given by
		\bqn \label{ExactSolutionsFourier}
		\left(\mathcall{F}_{\R^2}R_\alpha^{\ul{D},\ul{a},s}(\cdot,\theta)\right)(\pmb{\omega})
		=
		\left(\mathcall{F}_{\R^2}R_\alpha^{\ul{D},\ul{a}}(\cdot,\theta)\right)(\pmb{\omega})e^{-s|\pmb{\omega}|^2},
		\eqn
		where analytic expressions for $\hat{R}_\alpha^{\ul{D},\ul{a}}(\pmb{\omega,\theta})=\left[\mathcall{F}_{\R^2}(R_\alpha^{\ul{D},\ul{a}}(\cdot,\theta))\right](\pmb{\omega})$ in terms of Mathieu functions are provided in Theorem \ref{th:exact}. For the spatial distribution, we have the following error estimation:
		\begin{equation}\label{ExactSolutionRiemannSumApproximation}
		R_\alpha^{\ul{D},\ul{a},s}(\ul{x},\theta_r)=\left(\left[\mathbf{CDFT}\right]^{-1}(\hat{R}_{\alpha}^{\mathbf{D},\mathbf{a},s}(\pmb{\omega}_\cdot^1,\pmb{\omega}_\cdot^2,\theta_r))\right)(\ul{x})+I_\varsigma^s(\ul{x},r)+O\left(\frac{1}{2P+1}\right)+O\left(\frac{1}{2Q+1}\right),
		\end{equation}
		for all $\ul{x}=(x,y) \in \mathbb{Z}_P \times \mathbb{Z}_Q$, with discretization in Eq.~(\ref{discretefrequencies}), $\varsigma \in \mathbb{N}$ denotes the oversampling factor in the Fourier domain and $s=\frac{1}{2}\sigma^2$ is the spatial Gaussian blurring scale with $\sigma \approx 1, 2$ pixel length, and \\
		\begin{equation}
		I_\varsigma^s(\ul{x},r)=\int_{\R^2 \setminus [-\varsigma\pi,\varsigma\pi]^2}e^{-s|\pmb{\omega}|^2}\hat{R}_{\alpha}^{\ul{D},\ul{a}}(\pmb{\omega},\theta_r)e^{i\pmb{\omega}\cdot\ul{x}}d\pmb{\omega}.
		\end{equation}
		\label{thm:DiscreteExactSolutionsErrorEstimation}
	\end{theorem}
	First of all we recall Eq.~(\ref{ResolventWithGaussian}), from which Eq.~(\ref{ExactSolutionsFourier}) follows. Eq.~(\ref{ExactSolutionRiemannSumApproximation}) follows by standard Riemann-sum approximation akin to Eq.~(\ref{DiscreteExactSolutions}).
	Finally, we note that due to H\"{o}rmander theory \cite{Hoermander} the kernel $R_\alpha^{\ul{D},\ul{a}}$ is smooth on $SE(2)\setminus\{e\}=(0,0,0)$. Now, thanks to the isotropic diffusion, $R_\alpha^{\ul{D},\ul{a},s}$ is well-defined and smooth on the whole group $SE(2)$.
	\begin{remark}
		In the isotropic case $D_{11}=D_{22}$ we have the asympotic formula (for $\rho \gg 0$ fixed):
		\begin{equation}
		\begin{array}{l}
		\cr(D_{11}\rho^2+D_{33}\rho_\theta^2+\alpha I)\hat{R}_{\alpha}^{\ul{D},\ul{a}}(\pmb\omega,\theta)=1
		\Longrightarrow  \hat{R}_{\alpha}^{\ul{D},\ul{a}}(\pmb\omega,\rho_\theta)=\frac{1}{D_{11}\rho^2+D_{33}\rho_\theta^2+\alpha} \approx O(\frac{1}{\rho^2})
		\end{array}
		\end{equation}
	\end{remark}
	Now for
	\begin{equation}
	\begin{array}{l}
	|I_\varsigma^s(\ul{x},r)|=|\int_{\R^2 \setminus [-\varsigma\pi,\varsigma\pi]^2}e^{-s|\pmb{\omega}|^2}\hat{R}_{\alpha}^{\ul{D},\ul{a}}(\pmb{\omega},\theta_r)e^{i\pmb{\omega}\ul{x}}d\pmb{\omega}|
	\leq 2\pi\int_{\varsigma\pi}^\infty e^{-s\rho^2}\frac{C}{\rho}d\rho = \pi C \; \Gamma(0,\pi^2 s \varsigma^2),
	\end{array}
	\end{equation}
	for fixed $\ul{a}$, $C \approx \frac{1}{D_{11}}$ (for $D_{33}$ small), and where $\Gamma(a,z)$ denotes the incomplete Gamma distribution. We have $s=\frac{1}{2}\sigma^2$. For typical parameter settings in the contour enhancement case, $\sigma=1$ pixel length, $D_{11}=1,D_{33}=0.05$, we have
	\begin{align} \label{GammaDistribution}
	|I_\varsigma^s(\ul{x},r)|\leq\left\{
	\begin{aligned}
	&(0.00124)\pi C, &\varsigma=1 \\
	&(10^{-10})\pi C, &\varsigma=2
	\end{aligned}
	\right. \end{align}
	which is sufficiently small for $\varsigma \geq 2$.
	
	\subsubsection{Scale Selection of the Gaussian Mask and Inner-scale}\label{ScaleSelectionGaussianMask}
	In the previous section, we proposed to use a narrow spatial isotropic Gaussian window to control errors caused by using the $\textbf{DFT}^{-1}$. In $\R$, we have $\sqrt{4\pi s} \mathcall{F}G_{s}=G_{\frac{1}{4s}}$, i.e.
	\begin{equation}\label{GaussianFunction}
	\begin{aligned}
	(\mathcall{F}G_{s})(\omega)&=e^{-s||\omega||^2}, \qquad
	G_s(x)=\frac{1}{\sqrt{4\pi s}}e^{\frac{-||x||^2}{4s}},
	\qquad
	\sigma_s \cdot \sigma_f=1.
	\end{aligned}
	\end{equation}
	where $\sigma_f$ denotes the standard deviation of the Fourier window, and $\sigma_s$ denotes the standard deviation of the spatial window.
	In our convention, we always take $\Delta x=\frac{l}{N_s}$ as the spatial pixel unit length, where $l$ gives the spatial physical length and $N_s$ denotes the number of samples.
	
	The size of the fourier Gaussian window can be represented as: $2\sigma_f=\nu\cdot\varsigma\pi$,
	where $\nu\in[\frac{1}{2},1]$ is the factor that specifies the percentage of the maximum frequency we are going to sample in the fourier domain and $\varsigma$ is the oversampling factor. Then, we can represent the size of the continuous and discrete spatial Gaussian window $\sigma_{s}$ and $\sigma_{s}^{Discrete}$ as:
	\begin{equation}
	\sigma_{s}=\frac{2}{\nu \varsigma \pi},\qquad
	\sigma_{s}^{Discrete}=\sigma_s \cdot \frac{l}{N_s}=\frac{2}{\nu \varsigma \pi}\left(\frac{l}{N_s}\right).
	\end{equation}
	From Figure~\ref{fig:GaussianScale}, we can see that a Fourier Gaussian window with $\nu<1$ corresponds to a spatial Gaussian blurring of slightly more than 1 pixel unit.
	\begin{figure}[htbp]
		\centering
		\subfloat{\includegraphics[scale=0.6]{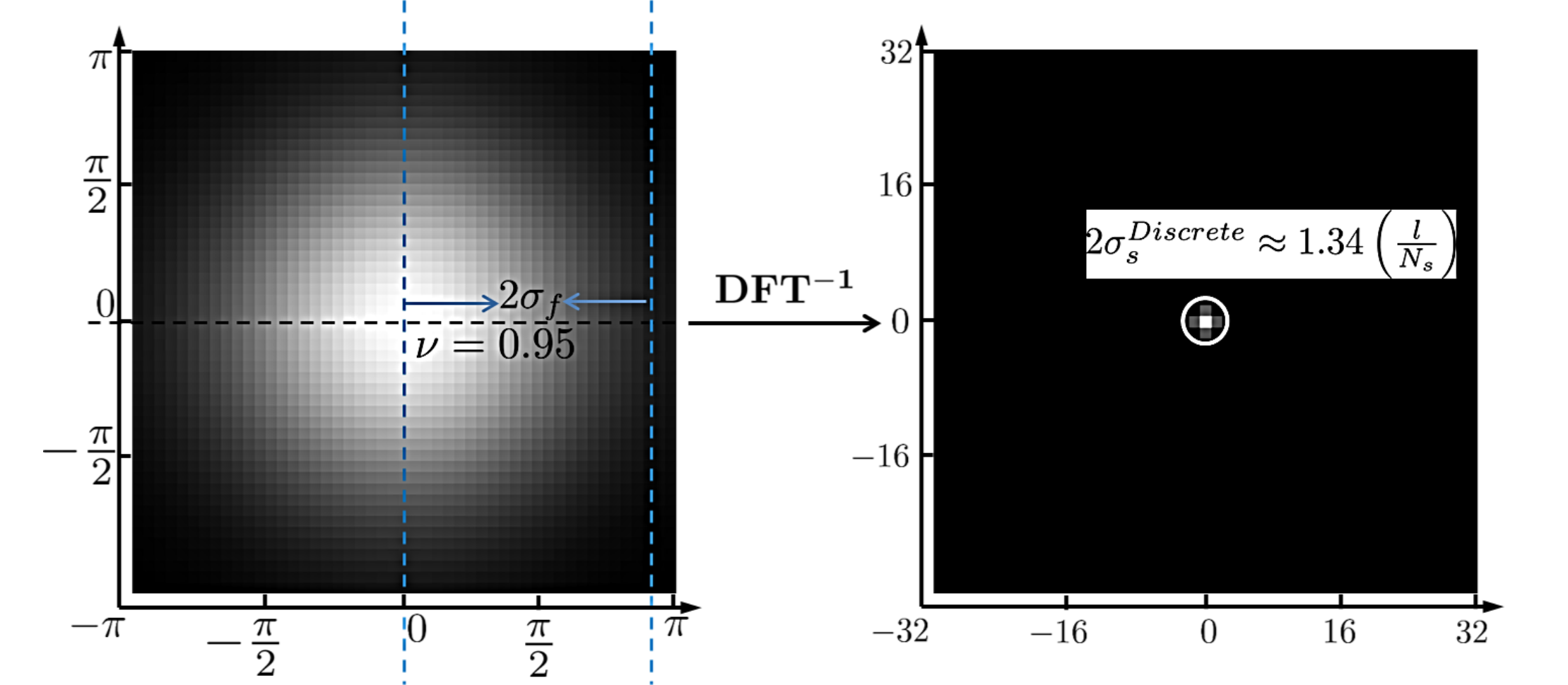}}
		\caption{Illustration of the scales between a Fourier Gaussian window and the corresponding spatial Gaussian window. Here we define the number of samples $N_s=65$.}
		\label{fig:GaussianScale}
	\end{figure}
	If we set the oversampling factor $\varsigma=1$, one has $2\sigma_{s}^{Discrete}\in \left[\Delta x, 2\Delta x\right]$. Then, the scale of the spatial Gaussian window $s_s=\frac{1}{2}(\sigma_s^{Discrete})^2\leq\frac{1}{2}(\Delta x)^2$, in which $\frac{1}{2}(\Delta x)^2$ is called \emph{inner-scale} \cite{FlorackInnerScale1992}, which is by definition the minimum reasonable Gaussian scale due to the sampling distance.
	
	\subsubsection{Comparison by Relative $\ell_K-$errors in the Spatial and Fourier Domain}\label{section:ComparisonRelativeErrorFormula}
	Firstly, we explain how to make comparisons in the Fourier domain. Before the comparison, we apply a normalization such that all the DC components in the discrete Fourier domain add up to 1, i.e.
	\begin{align*}
	\sum_{r=-R}^R\sum_{x=-P}^P\sum_{y=-Q}^Q R_\alpha^{\ul{D},\ul{a}}(x,y,\theta_r)\Delta x \Delta y \Delta \theta=\sum_{r=-R}^R \left(\left[\mathbf{CDFT}\right]R_\alpha^{\ul{D},\ul{a}}(\cdot,\cdot,\theta_r)\right)(0,0)\cdot \Delta\theta=1,
	\end{align*}
	where the $\textbf{CDFT}$ and its inverse are given by
	\begin{equation}
	\begin{array}{l}
	\cr\left[\mathbf{CDFT}\left(R_\alpha^{\ul{D},\ul{a}}(\cdot,\cdot,\theta_r)\right)\right][p',q']:=\sum\limits_{p=-P}^P\sum\limits_{q=-Q}^Q R_\alpha^{\ul{D},\ul{a}}(p,q,\theta_r)e^{\frac{-2\pi i p p'}{2P+1}}e^{\frac{-2\pi i q q'}{2Q+1}},\\
	\cr\left[\mathbf{CDFT^{-1}}\left([p',q']\rightarrow \hat{R}_\alpha^{\ul{D},\ul{a}}(\omega_{p'}^1,\omega_{q'}^2,\theta_r)\right)\right][p,q]\cr\qquad\qquad\qquad:=\left(\frac{1}{2P+1}\frac{1}{2Q+1}\right)\sum\limits_{p'=-P}^P\sum\limits_{q'=-Q}^Q R_\alpha^{\ul{D},\ul{a}}(\omega_{p'}^1,\omega_{q'}^2,\theta_r)e^{\frac{2\pi i p p'}{2P+1}}e^{\frac{2\pi i q q'}{2Q+1}},
	\end{array}
	\end{equation}
	in order to be consistent with the normalization in the continuous domain:
	\begin{align*}
	\int_{-\pi}^\pi\hat{R}_\alpha^{\ul{D},\ul{a}}(0,0,\theta){\rm d}\theta=\int_{-\pi}^\pi\int_\R\int_\R R_\alpha^{\ul{D},\ul{a}}(x,y,\theta){\rm d}x {\rm d}y {\rm d}\theta=1.
	\end{align*}
	
	The procedures of calculating the relative errors $\epsilon_R^f$ in the Fourier domain are given as follows:
	\begin{equation}\label{RelativeError}
	\epsilon_R^f=\frac{\textbf{|}\hat{R}_\alpha^{\ul{D},\ul{a},exact}(\omega_{\cdot}^1,\omega_{\cdot}^2,\theta_\cdot)-\hat{R}_\alpha^{\ul{D},\ul{a},approx}(\omega_{\cdot}^1,\omega_{\cdot}^2,\theta_\cdot)\textbf{|}_{\ell_K(\mathbb{Z}_P \times \mathbb{Z}_Q \times \mathbb{Z}_R )}}{{\textbf{|}\hat{R}_\alpha^{\ul{D},\ul{a},exact}(\omega_{\cdot}^1,\omega_{\cdot}^2,\theta_\cdot)\textbf{|}_{\ell_K(\mathbb{Z}_P \times \mathbb{Z}_Q \times \mathbb{Z}_R )}}
	},
	\end{equation}
	where $K \in \mathbb{N}$ indexes the $\ell_K$ norm on the discrete domain $\mathbb{Z}_P \times \mathbb{Z}_Q \times \mathbb{Z}_R$. Akin to comparisons in the Fourier domain, we compute relative errors $\epsilon_R^s$ in the spatial domain as follows:
	\begin{equation}\label{RelativeError}
	\epsilon_R^s=\frac{\textbf{|}R_\alpha^{\ul{D},\ul{a},exact}(x_\cdot,y_\cdot,\theta_\cdot)-R_\alpha^{\ul{D},\ul{a},approx}(x_\cdot,y_\cdot,\theta_\cdot)\textbf{|}_{\ell_K(\mathbb{Z}_P \times \mathbb{Z}_Q \times \mathbb{Z}_R )}}{{\textbf{|}R_\alpha^{\ul{D},\ul{a},exact}(x_\cdot,y_\cdot,\theta_\cdot)\textbf{|}_{\ell_K(\mathbb{Z}_P \times \mathbb{Z}_Q \times \mathbb{Z}_R )}}
	},
	\end{equation}
	where we firstly normalize the approximation kernel with respect to the $\ell_1(\mathbb{Z}_P \times \mathbb{Z}_Q \times \mathbb{Z}_R )$ norm.
	\section{Experimental Results}\label{section:Experimental results}
	To compare the performance of different numerical approaches with the exact solution, Fourier and spatial kernels with special parameter settings are produced from different approaches in both enhancement and completion cases. The evolution of all our numerical schemes starts with a spatially blurred orientation score spike, i.e. $(G_{\sigma_s}*\delta_0^{\R^2})\otimes\delta_0^{S^1}$, which corresponds to the Fourier Gaussian window mentioned in Section~\ref{ch:comparison} for the error control of the exact kernel in Theorem \ref{thm:DiscreteExactSolutionsErrorEstimation}. We vary $\sigma_s>0$ in our comparisons. We analyze the relative errors of both spatial and Fourier kernels with changing standard deviation $\sigma_s$ of Gaussian blurring in the finite difference and the Fourier based approaches for contour enhancement, see Figure~\ref{fig:RelativeErrorWithChangingSigma}.
	
	All the kernels in our experiments are $\ell_1-$ normalized before comparisons are done. In the contour completion experiments, we construct all the kernels with the number of orientations $N_o = 72$ and spatial dimensions $N_s = 192$, while in the contour enhancement experiments we set $N_o = 48$ and $N_s = 128$. Our experiments are not aiming for speed of convergence in terms of $N_o$ and $N_s$, as this can be derived theoretically from Theorem \ref{th:RelationofFourierBasedWithExactSolution}, we rather stick to reasonable sampling settings to compare our methods, and to analyze a reasonable choice of $\sigma_s > 0$.
	\begin{figure}[!htbp]
		\centering
		\subfloat{\includegraphics[width=.6\textwidth]{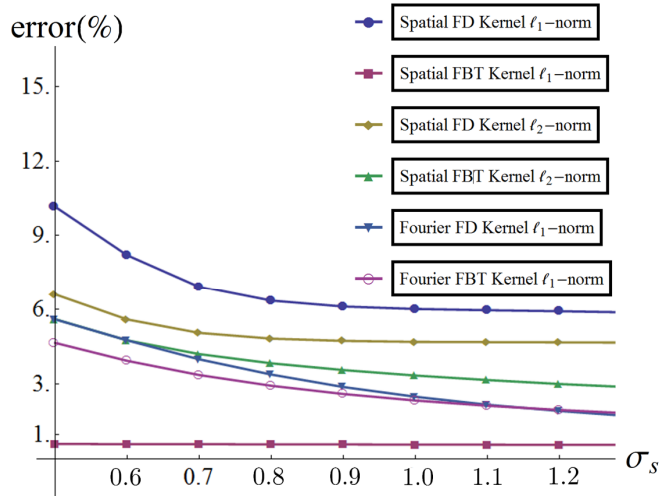}}
		\caption{The relative errors, Eq.~(\ref{RelativeError}), of the finite difference (FD), and Fourier based techniques (FBT) with respect to the exact methods (Exact) for contour enhancement. Both $\ell_1$ and $\ell_2$ normalized spatial and Fourier kernels are calculated based on different standard deviation $\sigma_s$ ranging from 0.5 to 1.7 pixels, with parameter settings $\ul{D}=\{1.,0.,0.03\}, \alpha=0.05$ and time step size $\Delta t=0.005$ in the FD explicit approach.
		}
		\label{fig:RelativeErrorWithChangingSigma}
	\end{figure}
	
	From Figure~\ref{fig:RelativeErrorWithChangingSigma} we deduce that the relative errors of the $\ell_1$ and $\ell_2$ normalized finite difference (FD) spatial kernels converge to an offset of approximately $5\%$, which is understood by additional numerical blurring due to B-spline approximation in Section \ref{section: Left-invariant Finite Differences with B-spline Interpolation}, which is needed for rotation covariance in discrete implementations \cite[Figure~\!10]{Franken2009IJCV}, but which does affect the actual diffusion parameters. The relative errors of the Fourier based techniques (FBT) are very slowly decaying from $0.61\%$ along the axis $\sigma_s$. We conclude that an appropriate stable choice of $\sigma_s$ for fair comparison of our methods is $\sigma_s=1$, recall also Section~ \ref{ScaleSelectionGaussianMask}.

	\begin{table*}[!ht]
		\centering
		\caption{Enhancement kernel comparison of the exact analytic solution with the numerical Fourier based techniques, the stochastic methods and the finite difference schemes.}
		\begin{tabular}{r|l|l|l}
			\toprule
			Relative Error & $\ul{D}=\{1.,0.,0.05\}$ & $\ul{D}=\{1.,0.,0.05\}$ & $\ul{D}=\{1.,0.9,1.\}$ \\
			(\ul{\%})  &    $\alpha=0.01$      &       $\alpha=0.05$      &     $\alpha=0.05$         \\
			\midrule
			$\ell_1$-norm        & Spatial \quad Fourier  & Spatial \quad Fourier & Spatial \quad Fourier  \\
			\midrule
			\textit{Exact-FBT} & \textbf{0.12} \qquad \textbf{1.30}  & \textbf{0.35} \qquad \textbf{1.92} & \textbf{2.27} \qquad \textbf{0.60} \\
			\textit{Exact-Stochastic}  & 2.18 \qquad 3.94 & 1.74 \qquad 3.82  & 2.66 \qquad 2.54  \\
			\textit{Exact-FDExplicit}  & 5.07 \qquad 1.82 & 5.70 \qquad 2.34  & 2.99 \qquad 3.56 \\
			\textit{Exact-FDImplicit}  & 5.08 \qquad 2.29 & 5.70 \qquad 3.03  & 3.00 \qquad 5.59  \\
			\midrule
			$\ell_2$-norm        & Spatial \quad Fourier  & Spatial \quad Fourier & Spatial \quad Fourier  \\
			\midrule
			\textit{Exact-FBT} & \textbf{1.40} \qquad \textbf{1.37}& \textbf{2.39} \qquad \textbf{2.30} & \textbf{2.24} \qquad \textbf{1.23}  \\
			\textit{Exact-Stochastic}  & 2.26 \qquad 2.32& 3.50 \qquad 3.16 & 2.93 \qquad 2.65  \\
			\textit{Exact-FDExplicit}  & 4.80 \qquad 1.72& 4.97 \qquad 1.60 & 2.90 \qquad 3.15  \\
			\textit{Exact-FDImplicit}  & 5.17 \qquad 2.11& 5.80 \qquad 2.29 & 5.42 \qquad 5.56  \\
			\bottomrule
		\end{tabular}%
		\caption*{
			\textbf{Measurement method abbreviations}: ($\textit{Exact}$) - Ground truth measurements based on the analytic solution by using Mathieu functions in Section~\ref{3GeneralFormsExactSolutions}, ($\textit{FBT}$) - Fourier based techniques in Section~\ref{section:Duitsmatrixalgorithm} and Section~\ref{section:FourierBasedForEnhancement}, ($\textit{Stochastic}$) - Stochastic method in Section~\ref{section:MonteCarloStochasticImplementation} (with $\Delta t=0.02$ and $10^8$ samples), ($\textit{FDExplicit}$) and ($\textit{FDImplicit}$) - Explicit and implicit left-invariant finite difference approaches with B-Spline interpolation in Section~\ref{section:Left-invariant Finite Difference Approaches for Contour Enhancement}, respectively. The settings of time step size are $\Delta t=0.005$ in the $\textit{FDExplicit}$ scheme, and $\Delta t=0.05$ in the $\textit{FDImplicit}$ scheme. }
		\label{tab:RelativeErrorEnhancementComparison}%
	\end{table*}

	Table~\ref{tab:RelativeErrorEnhancementComparison} shows the validation results of our numerical enhancement kernels, in comparison with the exact solution using the same parameter settings. The first 5 rows and the last 5 rows of the table show the relative errors of the $\ell_1$ and $\ell_2$ normalized kernels separately. In all the three parameter settings, the kernels obtained by using the FBT method provides the best approximation to the exact solutions due to the smallest relative errors in both the spatial and the Fourier domain. 
	Overall, the stochastic approach (a Monte Carlo simulation with $\Delta t=0.02$
	and $10^8$ samples) performs second best. 
	
	Although the finite difference scheme performs less, compared to the more computationally demanding FBT and the stochastic approach, the relative errors of the FD explicit approach are still acceptable, less than $5.7\%$. The $5\%$ offset is understood by the B-spline interpolation to compute on a left-invariant grid. Here we note that finite differences do have the advantage of straightforward extensions to the non-linear diffusion processes \cite{Citti,Creusen2013,FrankenPhDThesis,Franken2009IJCV}, which will also be employed in the subsequent application section. For the FD implicit approach, larger step size can be used than the FD explicit approach in order to achieve a much faster implementation, but still with negligible influence on the relative errors.
	
	Table~\ref{tab:RelativeErrorCompletionComparison} shows the validation results of the numerical completion kernels with three sets of parameters. Again, all the $\ell_1$ and $\ell_2$ normalized FBT kernels show us the best performance (less than $1.2\%$ relative error) in the comparison.
	\begin{table*}[!ht]
		\centering
		\caption{Completion kernel comparison of the exact analytic solution with the numerical Fourier based techniques, the stochastic methods and the finite difference schemes.}
		\begin{tabular}{r|l|l|l}
			\toprule
			Relative Error & $\ul{D}=\{0.,0.,0.08\}$& $\ul{D}=\{0.,0.,0.08\}$  & $\ul{D}=\{0.,0.,0.18\}$ \\
			& $\ul{a}=(1.,0.,0.)$ & $\ul{a}=(1.,0.,0.)$ & $\ul{a}=(1.,0.,0.)$ \\
			(\ul{\%})  &    $\alpha=0.01$  &       $\alpha=0.05$       &     $\alpha=0.05$          \\
			\midrule
			$\ell_1$-norm        & Spatial \quad Fourier  & Spatial \quad Fourier & Spatial \quad Fourier  \\
			\midrule
			\textit{Exact-FBT} & \textbf{0.02} \qquad \textbf{1.06}& \textbf{0.11} \qquad \textbf{1.17} & \textbf{0.05} \qquad \textbf{0.52}  \\
			\textit{Exact-Stochastic} & 2.49 \qquad 3.31 & 2.37 \qquad 5.40  & 1.95 \qquad 4.26\\
			\textit{Exact-FDExplicit}  & 1.91 \qquad 8.36& 4.29  \qquad 8.68  & 4.57 \qquad 9.03 \\
			\midrule
			$\ell_2$-norm        & Spatial \quad Fourier  & Spatial \quad Fourier & Spatial \quad Fourier  \\
			\midrule
			\textit{Exact-FBT} & \textbf{0.94} \qquad \textbf{1.21}& \textbf{1.20} \qquad \textbf{1.50} & \textbf{0.65} \qquad \textbf{0.79}  \\
			\textit{Exact-Stochastic} & 4.96 \qquad 3.40 & 4.84 \qquad 3.25  &  4.39 \qquad 2.45\\
			\textit{Exact-FDExplicit}  & 6.60 \qquad 5.50& 7.92 \qquad 6.56  & 8.46 \qquad 6.48 \\
			\bottomrule
		\end{tabular}%
		\caption*{
			\textbf{Measurement method abbreviations}:
			($\textit{Exact}$) - Ground truth measurements based on the analytic solution by using Mathieu functions in Section~\ref{3GeneralFormsExactSolutions}, ($\textit{FBT}$) - Fourier based techniques in Section~\ref{section:Duitsmatrixalgorithm} and Section~\ref{section:FourierBasedForEnhancement}, ($\textit{Stochastic}$) - Stochastic method in Section~\ref{section:MonteCarloStochasticImplementation}
			(with $\Delta t=0.02$ and $10^8$ samples), ($\textit{FDExplicit}$) - Explicit left-invariant finite difference approaches with B-Spline interpolation in Section~\ref{section:Left-invariant Finite Difference Approaches for Contour Enhancement}. The settings of time step size are $\Delta t=0.005$ in the $\textit{FDExplicit}$ scheme.}
		\label{tab:RelativeErrorCompletionComparison}%
	\end{table*}
	\section{Application of Contour Enhancement to Improve Vascular Tree Detection in Retinal Imaging}\label{section:Applications on Retinal Image}
	In this section, we will show the potential of achieving better vessel tracking results by applying the $SE(2)$ contour enhancement approach on challenging retinal images where the vascular tree (starting from the optic disk) must be detected. The retinal vasculature provides a convenient mean for non-invasive observation of the human circulatory system. A variety of eye-related and systematic diseases such as glaucoma\cite{CassonGlaucoma2012}, age-related macular degeneration, diabetes, hypertension, arteriosclerosis or Alzheimer's disease affect the vasculature and may cause functional or geometric changes\cite{Ikram2013}. Automated quantification of these defects enables
	massive screening for systematic and eye-related vascular diseases on the basis of fast and inexpensive imaging modalities, i.e. retinal photography. To automatically extract and assess the state of the retinal vascular tree, vessels have to be segmented, modeled and analyzed. Bekkers et al.\cite{BekkersJMIV} proposed a fully automatic multi-orientation vessel tracking method (ETOS) that performs excellently in comparison with other state-of-the-art algorithms. However, the ETOS algorithm often suffers from low signal to noise ratios, crossings and bifurcations, or some problematic regions caused by leakages/blobs due to some diseases. See Figure~\ref{fig:ProblematicRetinalVesselTracking}.
	\begin{figure}[htbp]
		\centering
		\subfloat{\includegraphics[width=.75\textwidth]{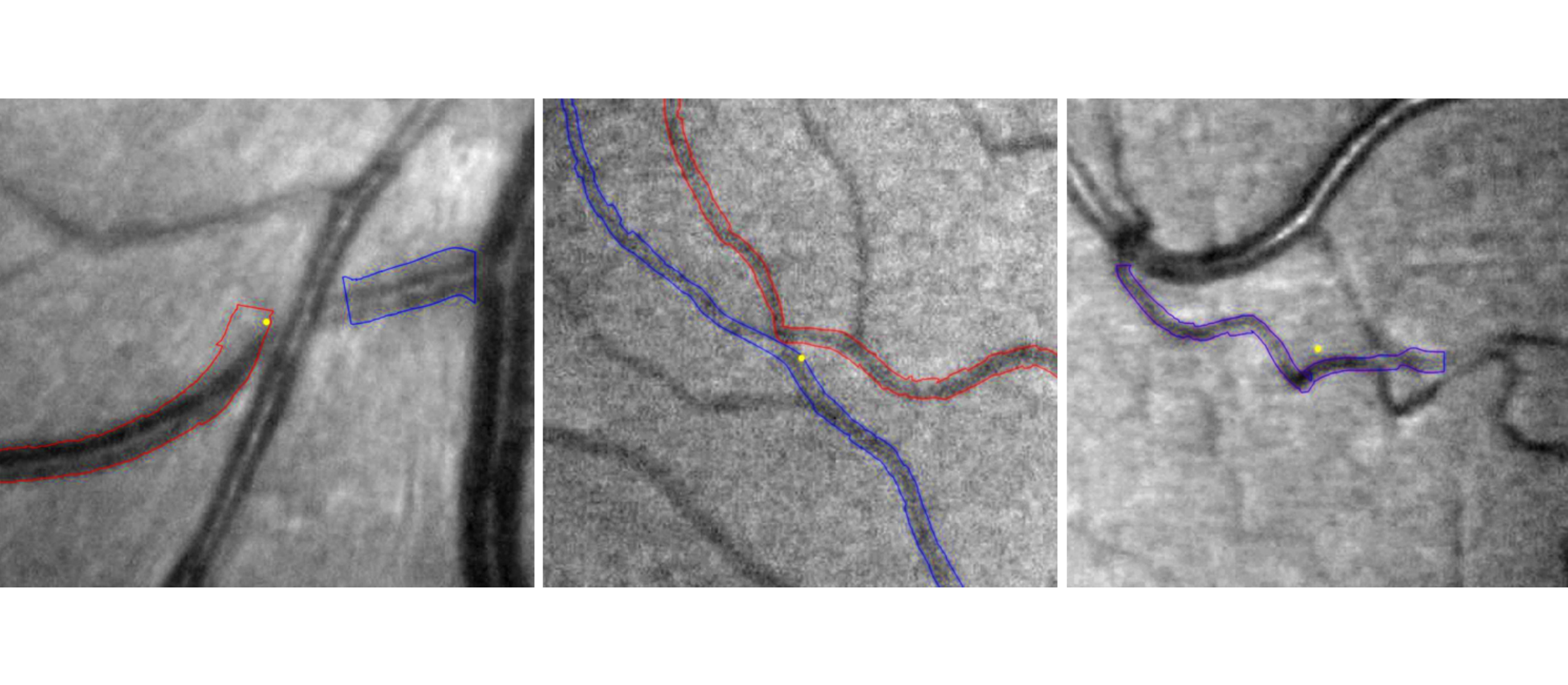}}
		\caption{Three problematical cases in the ETOS tracking algorithm\cite{BekkersJMIV}. From left to right: blurry crossing parts, small vessels with noise and small vessels with high curvature.
		}
		\label{fig:ProblematicRetinalVesselTracking}
	\end{figure}
	
	We aim to solve these problems via left-invariant contour enhancement processes on invertible orientation scores as pre-processing for subsequent tracking\cite{BekkersJMIV}, recall Figure~\ref{fig:OSIntro}. In our enhancements, we rely on non-linear extension \cite{Franken2009IJCV} of finite difference implementations of the contour enhancement process to improve adaptation of our model to the data in the orientation score. Finally, the ETOS tracking algorithm \cite{BekkersJMIV} is performed on the enhanced retinal images with respect to various problematic tracking cases, in order to show the benefit of the left-invariant diffusion on $SE(2)$.
	
	As a proof of concept, we show examples of tracking on left-invariantly diffused invertible orientation scores on cases where standard ETOS-tracking without left-invariant diffusion fails, see~\mbox{Figure~\ref{fig:VesselTrackingonCEDOSRetinalImages}.}
	\begin{figure}[htbp]
		\centering
		\includegraphics[width=.75\textwidth]{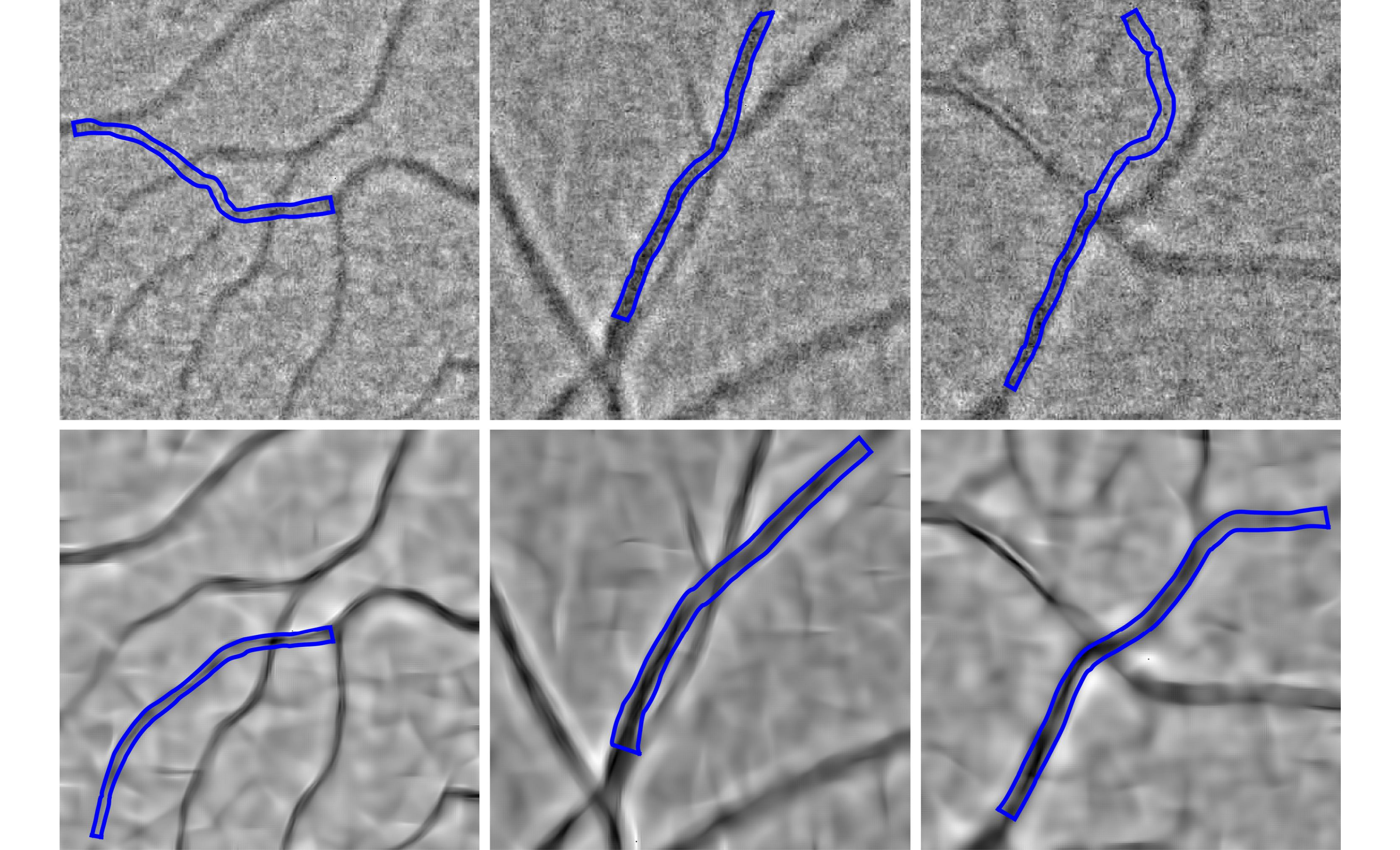}
		\caption{Vessel tracking on retinal images. From up to down: the original retinal images with erroneous ETOS tracking, the enhanced retinal images with accurate tracking after enhancement.}
		\label{fig:VesselTrackingonCEDOSRetinalImages}
	\end{figure}
	
	All the experiments in this section use the same parameters. All the retinal images are selected with the size $400 \times 400$. Parameters used for tracking are the same as the parameters of the ETOS algorithm in \cite{BekkersJMIV}: Number of orientations $N_o$ = 36, wavelets-periodicity = $2\pi$.  The following parameters are used for the non-linear coherence-enhancing diffusion (CED-OS): spatial scale of the Gaussian kernel for isotropic diffusion is $t_s=\frac{1}{2}\sigma_s^2=12$, the scale for computing Gaussian derivatives is $t_s'=0.15$, the metric $\beta=0.058$, the end time $t=20$, and $c=1.2$ for controlling the balance between isotropic diffusion and anisotropic diffusion, for details see \cite{Franken2009IJCV}.
	
	\section{Conclusion}
	
	We analyzed linear left-invariant diffusion, convection-diffusion and their resolvents on invertible orientation scores, following both 3 numerical and 3 exact approaches. In particular, we considered the Fokker-Planck equations of Brownian motion for contour enhancement, and the direction process for contour completion. We have provided 3 exact solution formulas for the generic left-invariant PDE's on $SE(2)$ to place previous exact formulas into context. These formulas involve either infinitely many periodic or non-periodic Mathieu functions, or only 4 non-periodic Mathieu functions.
	
	Furthermore, as resolvent kernels suffer from severe singularities that we analyzed in this article, we propose a new time integration via Gamma distributions, corresponding to iterations of resolvent kernels. We derived new asymptotic formulas
	for the resulting kernels and show benefits towards applications, illustrated via stochastic completion fields in Figure~\ref{fig:Gamma}. 
	
	Numerical techniques can be categorized into 3 approaches: finite difference, Fourier based and stochastic approaches. Regarding the finite difference schemes, rotation and translation covariance on reasonably sized grids requires B-spline interpolation \cite{Franken2009IJCV} (towards a left-invariant grid), including additional numerical blurring. We applied this both to implicit schemes and explicit schemes with explicit stability bound. Regarding Fourier based techniques (which are equivalent to $SE(2)$ Fourier methods, recall Remark~\ref{rem:42}), we have set an explicit connection in Theorem \ref{th:RelationofFourierBasedWithExactSolution} to the exact representations in periodic Mathieu functions from which convergence rates are directly deduced. This is confirmed in the experiments, as they perform best in the numerical comparisons. 
	
	We compared the exact analytic solution kernels to the numerically computed kernels for all schemes. We computed the relative $\ell_1$ and $\ell_2$ errors in both spatial and Fourier domain. We also analyzed errors due to Riemann sum approximations that arise by using the $\textbf{DFT}^{-1}$ instead of using the $\textbf{CFT}^{-1}$. Here, we needed to introduce a spatial Gaussian blurring with small ``inner-scale'' due to finite sampling. This small Gaussian blurring allows us, to control truncation errors, to maintain exact solutions, and to reduce the singularities.
	We implemented all the numerical schemes in $\textit{Mathematica}$, and constructed the exact kernels based on our own implementation of Mathieu functions to avoid the numerical errors and slow speed caused by $\textit{Mathematica}$'s Mathieu functions.
	
	We showed that FBT, stochastic and FD provide reliable numerical schemes. Based on the error analysis we demonstrated that best numerical results were obtained using the FBT with negligible differences. The stochastic approach (via a Monte Carlo simulation) performs second best.
	The errors from the FD method are larger, but still located in an admissible scope, and they do allow non-linear adaptation. Preliminary results in a retinal vessel tracking application show that the PDE's in the orientation score domain preserve the crossing parts and help the ETOS algorithm\cite{BekkersJMIV} to achieve more robust tracking.
	
	\section*{Acknowledgements}
	
	The research leading to the results of this article
	has received funding from the European Research Council under the European Community's 7th Framework Programme (FP7/2007-2014)/ERC grant agreement No. 335555. The China Scholarship Council (CSC) is gratefully acknowledged for the financial support No. 201206300010.
	\begin{figure}[htbp]
		\centering
		\includegraphics[width=.75\textwidth]{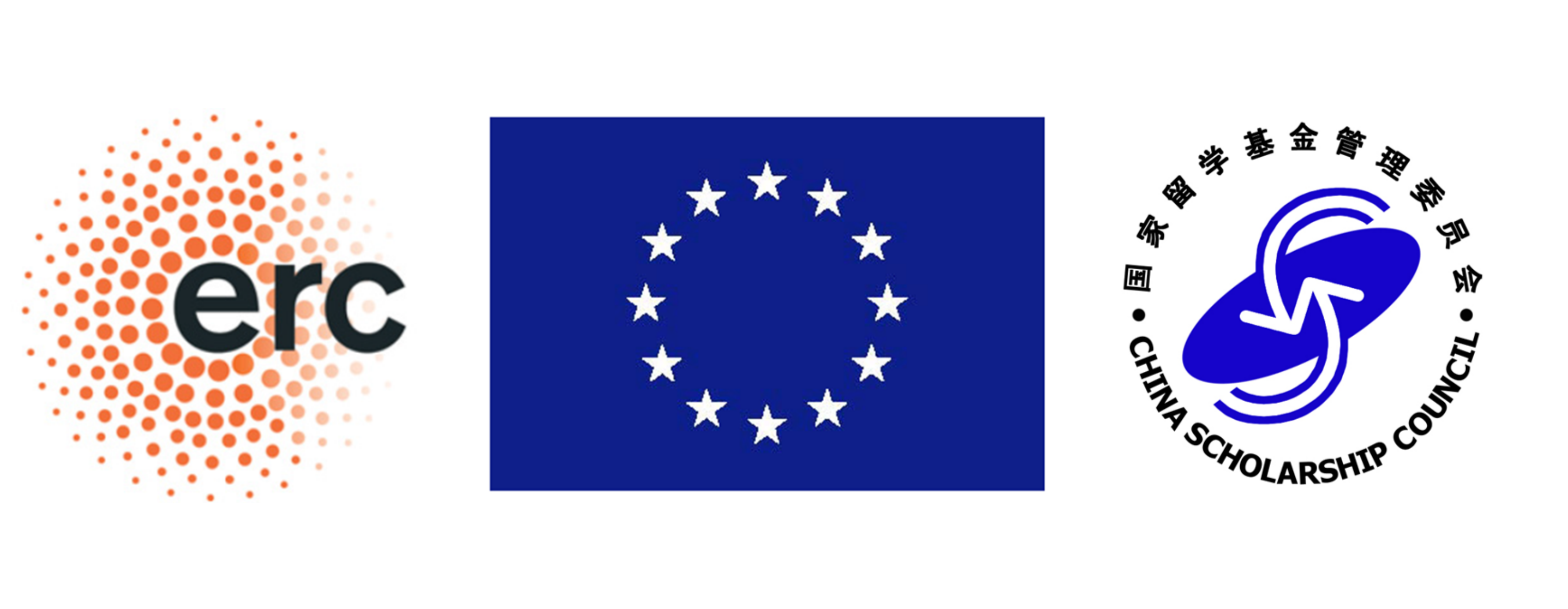}
		\label{fig:logos}
	\end{figure}
	
	\appendix
	
	\section{Invertible Orientation Scores of 2D-images and
		Continuous Wavelet Theory \label{app:new}}
	
	\noindent The continuous wavelet transform constructed by unitary irreducible representations of locally compact groups was first formulated by Grossman et al.~\cite{Grossmann1985}. Given a Hilbert space $H$ and a unitary irreducible representation $g\mapsto \mathcall{U}_g$ of any locally compact group $G$ in $H$, a vector $0\neq\psi \in H$ is called admissible if
	\begin{align}\label{CpsiDef}
	C_{\psi}:= \int_{G}\frac{|(\mathcall{U}_{g}\psi,\psi)|^2}{(\psi,\psi)_{H}}d\mu_{G}(g)<\infty,
	\end{align}
	where $\mu_G$ denotes the left-invariant Haar measure. Given an admissible vector $\psi$ and a unitary representation of a locally compact group $G$ in $H$, the CS transform $\widetilde{\mathcall{W}}_\psi:H\rightarrow \mathbb{L}_{2}(G)$ is given by
	$(\widetilde{\mathcall{W}}_\psi[f])(g)=(\mathcall{U}_{g}\psi,f)_H$. It is well known in mathematical physics~\cite{Alibook}, that $\widetilde{\mathcall{W}}_\psi$ is an isometric transform onto a closed reproducing kernel space $\mathbb{C}_{K_{\psi}}^{G}$ with $K_{\psi}(g,g')=\frac{1}{C_\psi}(\mathcall{U}_g\psi,\mathcall{U}_{g'}\psi)_{H}$ as an $\mathbb{L}_2$-subspace.
	
	Now in our orientation score transform $f \mapsto \mathcall{W}_{\psi}f$, Eq.~(\ref{OrientationScoreConstruction}), we restrict to disk-limited images\footnote{Such a restriction is convenient and reasonable for applications in view of the Nyquist frequency. Nevertheless, it is not strictly necessary for an $\mathbb{L}_2$-isometry, when one extends
		continuous wavelets to distributional wavelet transforms \cite[Thm~1,App.~B]{BekkersJMIV}.}:
	\[
	f\in \mathbb{L}_{2}^{\varrho}(\R^2)=\{g \in \mathbb{L}_{2}(\R^{2}) \;|\;
	\textrm{supp}\mathcall{F}_{\R^2}g \subset B_{\ul{0},\varrho}\},
	\]
	With $B_{\ul{0},\varrho}=\{\www \in \R^{2} \;|\; \|\www\|\leq \varrho\}$, with $\varrho>0$ close to the Nyquist-frequency.
	We set the left-regular representation $g \mapsto \mathcall{U}_{g}$ given by
	$(\mathcall{U}_{g=(\ul{x},\theta)}f)(\ul{y})=f(R_{\theta}^{-1}(\ul{y}-\ul{x}))$ as the unitary representation.
	
	We distinguish between the isometric wavelet transform $\widetilde{\mathcall{W}}_{\psi}:\mathbb{L}_2^{\varrho}(\mathbb{R}^2)\rightarrow\mathbb{L}_2(G)$ and the unitary wavelet transform $\mathcall{W}_{\psi}^{\varrho}:\mathbb{L}_2(\mathbb{R}^2)\rightarrow \mathbb{C}_K^G$,
	as they have different adjoint transforms.
	We drop the formal requirement of $\mathcall{U}$ being square-integrable and $\psi$ being admissible in the sense of \eqref{CpsiDef}, as it is not strictly needed/applicable for lots of cases.
	This includes our case of interest $G=SE(2)$ and its left-regular action on $\mathbb{L}_2(\mathbb{R}^2)$ where $\mathcall{W}_{\psi}f(g)=(\mathcall{U}_{g}\psi,f)_{\mathbb{L}_{2}(\R^{2})}$ gives rise to an orientation score.
	We call $\psi\in\mathbb{L}_2(\mathbb{R}^2) \cap \mathbb{L}_{1}(\mathbb{R}^2)$ an admissible vector if
	\begin{align}
	0<M_\psi(\omega):=(2\pi)
	\int\limits_{-\pi}^{\pi}\left|\mathcall{F}_{\R^2}\psi(R^{-1}_{\theta}\omega)
	\right|^2 \,
	{\rm d}\theta <\infty \text{ for all } \omega\in B_{\ul{0},\varrho}.
	\label{Evo21}
	\end{align}
	Note that $\mathbb{L}_{1}(\R^{2})$ implies that $\mathcall{F}_{\R^2}\psi$ and $M_{\psi}$ are continuous functions
	vanishing at infinity.
	
	From the general theory of reproducing kernels spaces, see e.g. \cite{Alibook}, \cite[Thm.18,Cor.4]{DuitsPhDThesis},
	it follows that $\mathcall{W}_{\psi}:\mathbb{L}_{2}^{\varrho}(\R^2) \mapsto \mathbb{C}_{K}^{SE(2)}$ is unitary,
	where $\mathbb{C}_{K}^{SE(2)}$ denotes the unique \cite{Aronszajn1950} reproducing kernel space consisting of complex-valued
	functions on $SE(2)$ with reproducing kernel
	\[
	K_{(\ul{x},\theta)}(\ul{x}',\theta')=(\mathcall{U}_{(\ul{x},\theta)}\psi, \mathcall{U}_{(\ul{x}',\theta')}\psi)_{\mathbb{L}_{2}(\R^{2})}.
	\]
	Unfortunately, the characterization of the inner-product and norm on the space of orientation scores $\mathbb{C}_{K}^{SE(2)}$ via its reproducing kernel
	is relatively complicated \cite{Martens}. Therefore, we provide a basic characterization of this inner-product next.
	For an admissible vector $\psi\in \mathbb{L}_2(\mathbb{R}^2)$, the span of $\{\mathcall{U}_{g}\psi \,|\, g\in G\}$, is dense in $\mathbb{L}_2(\mathbb{R}^2)$.
	The next construction is in line with general admissibility conditions in \cite[Ch.5]{Fuehrbook}.
	\begin{theorem}\label{MPsiRecon}
		Let $\psi$ be an admissible vector in the sense that (\ref{Evo21}) is satisfied.
		Then $\mathcall{W}_{\psi}: \mathbb{L}_{2}^{\rho}(\R^{2}) \to \mathbb{C}_{K}^{SE(2)}$ is unitary, and we have
		\[
		(f,g)_{\mathbb{L}_{2}(\R^{2})} = (\mathcall{W}_{\psi}f, \mathcall{W}_{\psi}g)_{M_{\psi}},
		\]
		where $(U,V)_{M_{\psi}}=(\mathcall{T}_{M_{\psi}}U, \mathcall{T}_{M_{\psi}}V)_{\mathbb{L}_{2}(SE(2))}$
		with operator $\mathcall{T}_{M_{\psi}}: \mathbb{C}_{K}^{SE(2)} \to \mathbb{L}_{2}(SE(2))$ given by
		\[
		[\mathcall{T}_{M_{\psi}}U](\ul{x},\theta)=\mathcall{F}_{\R^2}^{-1}\left[\omega \mapsto
		(2\pi)^{-\frac{1}{2}} M_{\psi}^{-\frac{1}{2}}(\omega) \mathcall{F}_{\R^2}U(\omega,\theta)\right](\ul{x}).
		\]
	\end{theorem}
	\begin{corollary}\label{corr:essence}
		Let $M_{\psi}(\www)>0$ for all $\www \in \R^{2}$.
		The space of orientation scores $\mathbb{C}_{K}^{SE(2)}$ is a closed subspace of $\mathbb{H}_{\psi} \otimes \mathbb{L}_{2}(S^{1})$,
		where $\mathbb{H}_{\psi}:= \{ f \in \mathbb{L}_{2}(\R^{2})\; |\; M_{\psi}^{-\frac{1}{2}} \mathcall{F}_{\R^2}f \in \mathbb{L}_{2}(B_{\ul{0},\varrho})\}$.
		The orthogonal projection $\mathbb{P}_{\psi}$ of $\mathbb{H}_{\psi} \otimes \mathbb{L}_{2}(S^{1})$ onto $\mathbb{C}_{K}^{SE(2)}$ is given by
		\[
		(\mathbb{P}_{\psi}U)(\ul{x},\theta)=(K_{(\ul{x},\theta)}, U)_{M_{\psi}}= (\mathcall{W}_{\psi} \mathcall{W}_{\psi}^{*,ext} U)(\ul{x},\theta),
		\]
		where $\mathcall{W}_{\psi}^{*,ext}: \mathbb{H}_{\psi}^{\varrho} \otimes \mathbb{L}_{2}(S^{1}) \to \mathbb{L}_{2}(\R^{2})$ is the natural extension of the adjoint given by
		\begin{equation} \label{Recform}
		\mathcall{W}_{\psi}^{*,ext} U= \mathcall{F}_{\R^2}^{-1}\left[ M_{\psi}^{-1}
		\mathcall{F}_{\R^2} \left[ \ul{x} \mapsto
		\int \limits_{-\pi}^{\pi} (\psi_{\theta+\pi} * U(\cdot,\theta))(\ul{x}) {\rm d}\theta \right]
		\right].
		\end{equation}
	\end{corollary}
	\begin{remark}
		In Theorem~\ref{MPsiRecon} we have restricted ourselves to disk-limited images. In Corollary~\ref{corr:essence} we did not apply such a restriction, as it is not needed. Indeed, if $U \in \mathbb{H}_{\psi}$ is such that $\mathcall{F}U(\cdot,\theta)$ has support outside the disk with radius $\varrho$ for all $\theta \in (-\pi,\pi]$, then it is mapped to zero in (\ref{Recform}), i.e. then $\mathcall{W}_{\psi}^{*,ext} U=0$. \\
		However, in order to ensure that the Sobolev type of space $\mathbb{H}_{\psi}$ is a true $\mathbb{L}_{2}$-space endowed with $\mathbb{L}_{2}$-norm a restriction to disk limited images $f \in \mathbb{L}_{2}^{\varrho}(\R^{2})$ is necessary, as $M_{\psi}$ is a continuous function vanishing at infinity.
		In that case (using $\mathbb{L}_{2}^{\varrho}(\R^{2})$, $0<\varrho < \infty$, as input space) we need to replace $\mathbb{H}_{\psi}$ by the space $\mathbb{H}_{\psi}^{\varrho}:=\{f \in \mathbb{L}_{2}^{\varrho}(\R^{2})\;|\; M_{\psi}^{-1} \mathcall{F}f \in \mathbb{L}_{2}(\R^{2})\}$. In case $M_{\psi}$ is uniformly bounded from below on $B_{\ul{0}}^{\varrho}$, the set $\mathbb{H}_{\psi}^{\varrho}$ coincides with the set
		$\mathbb{L}_{2}^{\varrho}(\R^{2})$, although it is equipped with a different equivalent norm.
		In case $M_{\psi}= 1_{B_{\ul{0},\varrho}}$, the norms coincide and then $\mathbb{H}_{\psi}^{\varrho} \otimes \mathbb{L}_{2}(S^{1}) =\mathbb{L}_{2}(SE(2))$, and (\ref{Recform}) reduces to
		\[
		(\mathcall{W}_{\psi}^{*,ext} U)(\ul{x})=
		\int \limits_{-\pi}^{\pi} (\psi_{\theta+\pi} * U(\cdot,\theta))(\ul{x}) {\rm d}\theta
		= \int \limits_{SE(2)} U(g)\; (\mathcall{U}_{g}\psi)(\ul{x}) \, {\rm d}\mu_{G}(g).
		\]
	\end{remark}
	
	\section{Asymptotical Behavior of the Kernels around the Origin in the Fourier Domain}\label{app:A}
	
	Asymptotical analysis is done for the contour enhancement case in \ref{AsymptoticalEnhancement}, while asymptotical analysis is done for the contour completion case in \ref{AsymptoticalCompletion}.
	
	\subsection{Contour Enhancement Asymptotic Formulas along $\omega_\xi$ and $\omega_\eta$-axis}\label{AsymptoticalEnhancement}
	By freezing \mbox{$\cos^2(\varphi - \theta)=1$} and dividing by $D_{33}$ within the generator in the Fourier domain Eq.~(\ref{ContourEnhancementMathieuOperator}). The formula are given as follows:
	\begin{equation}
	\left((\frac{D_{11}}{D_{33}}\rho^2+\frac{\alpha}{D_{33}})-\partial_\theta^2\right)\hat{R}_\alpha^{D_{11},D_{33}}(\www,\cdot)=\frac{\alpha}{2\pi D_{33}}\delta_0^\theta,
	\end{equation}
	in which $\rho=\omega_\xi=\cos\theta \omega_x+\sin\theta \omega_y$, and $\www=(\rho \cos\varphi, \rho \sin \varphi) \in \R^{2}$. This is solved by making continuous fit of solutions in null-space akin to Figure~\ref{fig:ContinuousFit}. Then we find
	\begin{equation}
	\hat{R}_\alpha^{D_{11},D_{33}}(\www,\cdot)=\frac{\alpha}{2\pi D_{33}W_\rho}\left\{
	\begin{array}{l}
	e^{\sqrt{\lambda}\theta}, \textrm{ for }\theta \leq 0, \\
	e^{-\sqrt{\lambda}\theta}, \textrm{ for }\theta \geq 0,
	\end{array}
	\right.
	\end{equation}
	where $\lambda=\frac{D_{11}}{D_{33}}\rho^2+\frac{\alpha}{D_{33}}$, and $W_\rho=2\sqrt{\lambda}$ denotes the Wronskian according to Eq.~(\ref{WronskianComputation}). Then, the approximation of the exact solution for contour enhancement is written as:
	\begin{equation}\label{A3}
	\hat{R}_\alpha^{D_{11},D_{33}}(\rho\cos\theta,\rho\sin\theta,\theta)\approx\frac{\alpha}{4\pi D_{33}}\frac{e^{-\sqrt{\rho^2\frac{D_{11}}{D_{33}}+\frac{\alpha}{D_{33}}}|\theta|}}{\sqrt{\rho^2\frac{D_{11}}{D_{33}}+\frac{\alpha}{D_{33}}}},
	\end{equation}
	in which $\frac{D_{33}}{D_{11}}$ should be small.
	Then, we can find the fundamental solution by taking $\lim_{\alpha\downarrow 0}\frac{\hat{R}_\alpha^{D_{11},D_{33}}(\pmb{\omega},\theta)}{\alpha}$, which can be represented as:
	\begin{equation}
	\hat{S}^{D_{11},D_{33}}(\rho\cos\theta,\rho\sin\theta,\theta)\approx\frac{1}{4 \pi}\frac{1}{\rho\sqrt{D_{11}D_{33}}}-\frac{|\theta|}{4 \pi D_{33}}+O(\theta^2)
	\end{equation}
	
	Similarly, we can also get the resolvent equation along $\omega_\eta-$ axis for small $\rho$. Here we cannot freeze $\cos(\varphi - \theta)=0$ because the $\rho$ dependence will be lost, and we must rely on higher order expansion producing the following asymptotic formula:
	\begin{equation}
	\hat{R}_\alpha^{D_{11},D_{33}}(\rho\cos\theta,\rho\sin\theta,\theta)\approx
	\frac{\alpha}{4 \pi} \frac{1}{\sqrt{\rho^2 D_{11}D_{33}+\alpha D_{33}}}-\frac{1}{4 \pi}\left(\frac{1-e^{-\sqrt{\frac{\alpha}{D_{33}}}|\theta|}}{\sqrt{\alpha}\sqrt{D_{33}}}\right),
	\end{equation}
	and again for $0<\rho\ll1$
	\begin{equation}\label{A6}
	\hat{S}^{D_{11},D_{33}}(\rho\cos\theta,\rho\sin\theta,\theta)\approx\frac{1}{4 \pi}\frac{1}{\rho\sqrt{D_{11}D_{33}}}-\frac{|\theta|}{4 \pi D_{33}}.
	\end{equation}
	
	\textbf{\boldmath Conclusion:} From Eq.~(\ref{A3}) and (\ref{A6}), we deduce that $\hat{R}_\alpha^{D_{11},D_{33}}(\pmb{\omega},\theta)$ does not have a pole at $\pmb{\omega}=0$ for $\alpha>0$. $\hat{S}^{D_{11},D_{33}}(\pmb{\omega},\theta)$ has a pole of order 1 at $\pmb{\omega}=0$.
	
	\subsection{Contour Completion Asymptotic Formulas along $\omega_\xi$ and $\omega_\eta$-axis}\label{AsymptoticalCompletion}
	We again freeze $\cos(\varphi - \theta)=1$ for $\varphi=\theta$, i.e. along the $\omega_\xi$-axis, where $\rho=\omega_\xi=\cos\theta \omega_x+\sin\theta \omega_y$ in the generator in the Fourier domain Eq.~(\ref{ContourEnhancementMathieuOperator}) and apply Taylor approximation. Then we have the approximation of the resolvent equation in the Fourier domain, which is given by
	\begin{equation}
	\hat{R}_\alpha^{D_{33}}(\rho\cos\theta,\rho\sin\theta,\theta)
	\approx \frac{\alpha}{4 \pi}\frac{e^{-\frac{|\theta|\sqrt{D_{33}}}{\sqrt{\alpha+i\rho}}}}{\sqrt{\alpha D_{33}+i \rho D_{33}}}\\
	\end{equation}
	Note that the fundamental solution
	\begin{equation}
	\begin{aligned}
	\hat{S}^{D_{33}}(\rho\cos\theta,\rho\sin\theta,\theta)
	&\approx \lim_{\alpha \downarrow 0}\frac{1}{\alpha}\hat{R}_\alpha^{D_{33}}(\rho\cos\theta,\rho\sin\theta,\theta)\\
	&= \frac{1}{4 \pi \sqrt{\rho}}e^{-\frac{|\theta|\sqrt{D_{33}}}{\sqrt{2}\rho}}\left(\cos\left(\frac{|\theta|\sqrt{D_{33}}}{\sqrt{2}\rho}-\frac{\pi}{4}\right)-i\sin\left(\frac{|\theta|\sqrt{D_{33}}}{\sqrt{2}\rho}-\frac{\pi}{4}\right)\right).
	\end{aligned}
	\end{equation}
	Therefore, we do not have a pole in the resolvent kernel, but in the fundamental solution we have a pole of order $\frac{1}{2}$ in the Fourier domain. The behavior at $\infty$ is given by
	\begin{equation}
	\hat{S}^{D_{33}}(\rho\cos\theta,\rho\sin\theta,\theta) \approx \frac{e^{-\frac{|\theta|\sqrt{D_{33}}}{\sqrt{2}\rho}}}{4 \pi \sqrt{\rho}}.
	\end{equation}
	Unlike the enhancement case, we cannot expect local isotropy at the origin.
	
	\section{Algorithm for Evaluation of Non-periodic Mathieu functions \label{app:B}}
	\noindent
	Consider the Mathieu equation
	\begin{equation} \label{Meq}
	y''(z)+ (a-2q \cos(2z))y(z)=0
	\end{equation}
	for $a \leq 0$ and $q\neq 0$. The Floquet theorem \cite{Schaefke} yields
	the existence of solutions
	\begin{equation} \label{MathieuSolution}
	y(z)= e^{ i \nu(a,q) z} \sum \limits_{\rho-\infty}^{\infty} e^{2 i \rho z} c_{2\rho}(a,q)
	\end{equation}
	with $\nu(a,q) \in \mathbb{C}$ the Floquet exponent (which is correctly implemented in \emph{Mathematica}) and with
	$(c_{2\rho}(a,q))_{\rho \in \mathbb{Z}} \in \ell_{2}(\mathbb{Z})$.
	Now the ODE has real-valued coefficients and for our second type of exact formulas in Theorem~\ref{th:exact} we are aiming for the two real-valued solutions
	\[
	\begin{array}{l}
	\textrm{me}_{\nu}(z) \rightarrow 0 \textrm{ if }z \rightarrow \infty, \ \
	\textrm{me}_{-\nu}(z) \rightarrow 0 \textrm{ if }z \rightarrow -\infty, 
	\end{array}
	\]
	with $\overline{\nu}=-\nu$ and where we take the convention $\textrm{Im}(\nu) \geq 0$. Substitution of (\ref{MathieuSolution}) into the Mathieu ODE directly provides the two-fold recursion $ c_{2\rho+2} +\frac{-a +(2\rho+\nu)^2}{q} c_{2\rho} +c_{2\rho-2}=0, \textrm{ for all } \rho \in \mathbb{Z},
	$
	with $c_0$ and $c_{2}$ such that
	\begin{equation} \label{cond}
	\lim \limits_{\rho \pm \infty}|c_{2\rho}|^{\frac{1}{|\rho|}}=0,
	\end{equation}
	cf.~\!\cite{Schaefke}, from which it follows that $(c_{2\rho})_{\rho \in \mathbb{Z}} \in \ell_{1}(\mathbb{Z})$
	we deduce that the series representations are uniformly converging (Weierstrass criterium) and furthermore their limits are continuously differentiable.
	Now we have $\textrm{me}_{-\nu}(z)=\textrm{me}_{\nu}(-z)$ and the solutions
	are real-valued if $c_{-2\rho}=\overline{c_{2\rho}}$.
	The two-fold recursion is of the type
	\[
	c_{2\rho+2} - D_{2\rho} c_{2\rho} + c_{2\rho-2}=0,
	\]
	with $D_{2\rho}= \frac{a-(2\rho+\nu(a,q))}{q}$ and we enter the theory
	of continued fractions via division by $c_{2\rho}$
	\[
	f_{\rho}= D_{2\rho} -\frac{1}{f_{\rho-1}}, \textrm{ with }f_{\rho}:= \frac{c_{2\rho+2}}{c_{2\rho}},
	\]
	and indeed under condition (\ref{cond})\cite[Eq.(3), p.106]{Schaefke} we obtain converging solutions of type II. For definitions see \cite[Section 2.22, p.107]{Schaefke}.\\
	
	\textbf{Algorithm} \\
	input: $a\leq 0$, $\R \ni q\neq 0$, $L \in \mathbb{N}$ recursion-depth at last coefficient, $2N+1 \in \mathbb{N}$ number of coefficients. \\[7pt]
	initialization: $c_{0}=1$, $f_{0}=1$, $f_{N+L}=\frac{1}{D_{2(N+L)}}$. \\[7pt]
	For $k=1, \ldots,N+L-2$ do
	$
	f_{N+L-k}:= \frac{1}{D_{2(N+L-k+1)}-f_{N+L-k+1}}.
	$ \\
	Then build $(c_{2},\ldots,c_{2N})$ by
	$
	c_{2l}= f_{l-1}c_{2l-2}$, for $l=1,\ldots,N$. \\[7pt]
	Then build $(\overline{c}_{2N},\ldots, \overline{c}_{2},1,c_{2},\ldots, c_{2N})$. \\[7pt]
	Then compute by means of $\textbf{DFT}$, $\textrm{me}_{\nu}(z)$ and $\textrm{me}_{-\nu}(z)$ from their coefficients via Eq.~\!(\ref{MathieuSolution}).\\

	\begin{figure}[!htbp]
		\centering{
			\includegraphics[scale=0.5]{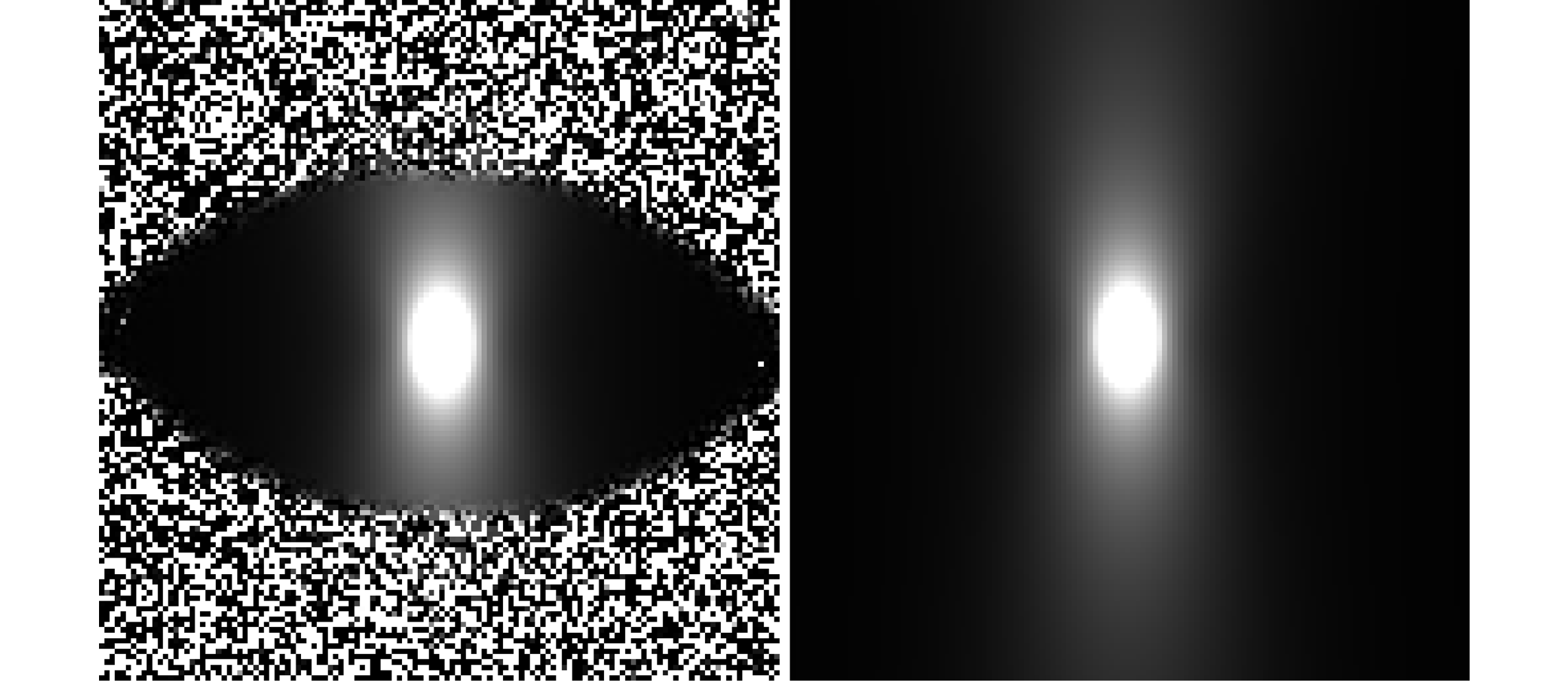}
		}
		\caption{From left to right: the final contour enhancement kernel based on the $\textit{Mathematica}$ Mathieu functions, the final contour enhancement kernel based on our own implementation of Mathieu functions. Both kernels use the same plot range and parameter settings: $\ul{D}=\{1,0,0.03\}$, $\alpha=0.025$, with sampling size $N_o = 48$ and $N_s = 128$.}
		\label{fig:MathieuImplementationComparison}
	\end{figure}
	
	Compared to the implementation of the contour enhancement kernel based on the $\textit{Mathematica}$ Mathieu functions, our own implementation of Mathieu functions is more robust and does not suffer from the numerical problems. They are much faster, see Table~\ref{tab:PPer}. Figure~\ref{fig:MathieuImplementationComparison} shows us the final kernels obtained by using the Mathieu functions of $\textit{Mathematica}$ (left) and our own implementation (right). The $\textit{Mathematica}$ Mathieu functions are shown to break down when the sampling enters into certain regions, especially with small angular diffusion. Another big advantage of our implementation is that the speed of sampling a kernel is almost 30 times faster than the $\textit{Mathematica}$ implementation. Table \ref{tab:PPer} shows us the  time requirements of the two routines for different parameter settings. We can see that our own Mathieu based implementation (OMI) is even 30 times faster than the $\textit{Mathematica}$ Mathieu based implementation (MMI).
	\begin{center}
		\begin{table}[h]
			\caption{Speed of two implementations $(\text{kernel size:}\, 48 \times 128 \times 128)$} 
			\centering 
			\begin{tabular}{l c c} 
				\hline
				\textit{\ul{Parameters}} &\textit{\ul{MMI time (s)}} & \textit{\ul{OMI time (s)}}
				\\ [0.5ex]
				\hline 
				
				$\ul{D}=\{1,0,0.03\},\;\alpha=0.025$
				& 4037 & 139\\
				
				$\ul{D}=\{1,0,0.12\},\;\alpha=0.025$
				& 3272 & 137\\
				
				$\ul{D}=\{1,0,0.03\},\;\alpha=0.05$
				& 3220 & 137\\
				
				\hline 
			\end{tabular}
			\caption*{
				\textbf{Measurement method abbreviations:} (\textit{OMI}) - Own Mathieu based implementation, (\textit{MMI}) - $\textit{Mathematica}$ Mathieu based implementation.}
			\label{tab:PPer}
		\end{table}
	\end{center}
	
	
	

\end{document}